\documentclass[11pt]{amsart}
\usepackage{hyperref}
\usepackage{anysize} \marginsize{1.3in}{1.3in}{1in}{1in}
\usepackage{amsfonts,comment}

\usepackage[all,cmtip]{xy}
\usepackage{amsmath}
\usepackage{amsthm}
\usepackage{amssymb}
\usepackage{mathrsfs}
\usepackage{enumitem}
\usepackage{tikz-cd}
\usepackage{tikz}
\usepackage[normalem]{ulem}

\newtheorem{thm}{Theorem}[section]

\newtheorem{prop}[thm]{Proposition}
\newtheorem{lemma}[thm]{Lemma}

\newtheorem{lem}[thm]{Lemma}
\newtheorem{cor}[thm]{Corollary}

\theoremstyle{definition}
\newtheorem{defn}[thm]{Definition}
\newtheorem*{definition*}         {Definition}

\newtheorem{eg}[thm]{Example}
\newtheorem{setting}[thm]{Setting}
\theoremstyle{remark}
\newtheorem*{claim}{Claim}
\newtheorem{remark}[thm]{Remark}

\renewcommand{\O}{\mathcal{O}}
\newcommand*{\Q}{\mathbb{Q}}

\newcommand*{\Z}{\mathbb{Z}}
\newcommand*{\G}{\mathbf{G}}

\newcommand*{\R}{\mathbb{R}}

\newcommand*{\C}{\mathbb{C}}

\newcommand*{\PP}{\mathbb{P}}

\newcommand*{\Hom}{\textrm{Hom}}
\DeclareMathOperator{\spec}{Spec}

\newcommand*{\ra}{\rightarrow}
\newcommand*{\rra}{\twoheadrightarrow}

\newcommand*{\ol}{\overline}

\renewcommand{\phi}{\varphi}

\def\Sym{\operatorname{Sym}}

\def\id{\operatorname{id}}

\def\Hom{\operatorname{Hom}}

\def\Nm{\operatorname{Nm}}

\def\Imag{\operatorname{Im}}

\def\Supp{\operatorname{Supp}}
\def\rk{\operatorname{rk}}

\def\Imag{\operatorname{Im}}

\def\colim{\operatorname{colim}}

\def\Coh{\operatorname{Coh}}
\def\Hilb{\operatorname{Hilb}}
\def\Module{\operatorname{Mod}}

\def\define{\mathrm{def}}
\def\reduced{\mathrm{red}}
\def\Define{\mathrm{Def}}
\def\An{\mathrm{An}}

\def\Coh{\mathrm{\mathbf{Coh}}}

\newcommand{\df}{\mathrm{def}}
\newcommand{\an}{\mathrm{an}}

\renewcommand{\bar}[1]{\overline{#1}}

\usepackage{amscd,amssymb,amsmath}
\usepackage{color}

\title{o-minimal GAGA and a conjecture of Griffiths} 

 \author[B. Bakker]{Benjamin Bakker}
\address{\noindent B. Bakker:  Dept. of Mathematics, Statistics, and Computer Science, University of Illinois at Chicago, Chicago, USA.}
\email{bakker.uic@gmail.com}

\author[Y. Brunebarbe]{Yohan Brunebarbe}
\address{\noindent Y. Brunebarbe:  Dept. of Mathematics, Univ. Bordeaux, Talence, France.}
\email{yohan.brunebarbe@math.u-bordeaux.fr}

\author[J. Tsimerman]{Jacob Tsimerman}
\address{\noindent J. Tsimerman:  Dept. of Mathematics, University of Toronto, Toronto, Canada.}
\email{jacobt@math.toronto.edu}

\begin{document}
\begin{abstract}We prove a conjecture of Griffiths on the quasi-projectivity of images of period maps using algebraization results arising from o-minimal geometry.  Specifically, we first develop a theory of analytic spaces and coherent sheaves that are definable with respect to a given o-minimal structure, and prove a GAGA-type theorem algebraizing definable coherent sheaves on complex algebraic spaces.  We then combine this with algebraization theorems of Artin to show that proper definable images of complex algebraic spaces are algebraic.  Applying this to period maps, we conclude that the images of period maps are quasi-projective and that the restriction of the Griffiths bundle is ample.   \end{abstract}
\maketitle

\section{Introduction}

Let $X$ be a smooth complex algebraic variety supporting a pure polarized integral variation of Hodge structures $(V_\Z,F^\bullet, Q)$.  Let $\Omega$ be the associated pure polarized period domain with generic Mumford--Tate group $\G$, and $\Gamma\subset \G(\Q)$ an arithmetic lattice containing the image of the monodromy representation of $V_\Z$.  There is a natural action of  $\Gamma$ on $\Omega$, and the quotient $\Gamma\backslash \Omega$ parameterizes pure Hodge structures up to integral framing in $\Gamma$.  Associated to a variation $(V_\Z,F^\bullet,Q)$ with monodromy in $\Gamma$ is a complex analytic period map $\phi:X^\an\to \Gamma\backslash \Omega$, where $X^\an$ is the analytification of $X$, that is, $X(\C)$ endowed with its natural structure as a complex analytic manifold.  The period map satisfies Griffiths transversality:  the derivative lands in a naturally defined distribution on $\Gamma\backslash \Omega$ (see \cite[pp.224-225]{Schmid}).  In general, for $X$ a reduced separated algebraic space of finite type over $\C$, we define a period map $\phi:X^\an\to\Gamma\backslash \Omega$ to be a complex analytic map which locally lifts to $\Omega$ and which satisfies Griffiths transversality on the regular locus of $X^\an$.  The main source of such period maps (and the variations of Hodge structures they entail) are local systems of singular cohomology groups of smooth projective families of algebraic varieties over $X$. 

The complex analytic variety $\Gamma\backslash \Omega$ itself rarely has an algebraic structure \cite{CT, GRT}; nonetheless, the closure of the image of a period map $\phi:X^\an\to\Gamma\backslash \Omega$ as above was conjectured by Griffiths \cite[p.259]{G1} to be a quasi-projective algebraic variety.  Griffiths' main motivation was the existence of a natural line bundle (which we call the Griffiths bundle) $L:=\bigotimes_i\det F^i$ which exists universally on $\Gamma\backslash \Omega$ as a $\Q$-bundle and has natural positivity properties in Griffiths transverse directions.  Aside from this, a strong piece of evidence for the conjecture is the result of Cattani--Deligne--Kaplan \cite{CDK} on the algebraicity of Hodge loci, which implies that the (reduced) analytic equivalence relation $X^\an\times_{\Gamma\backslash\Omega}X^\an\subset X^\an\times X^\an$ defining the image of $\phi$ set-theoretically is algebraic. 

Our main result is the following theorem, providing a solution to the conjecture:
\begin{thm}\label{maingriffiths}Let $X$ be a reduced separated algebraic space of finite type over $\C$ and $\phi:X^\an\to \Gamma\backslash \Omega$ a period map.  Then
\begin{enumerate}
\item $\phi$ factors (uniquely up to unique isomorphism) as $\phi=\iota\circ f^\an$ where $f:X\to Y$ is a dominant map of (reduced) finite-type algebraic spaces and $\iota:Y^\an\to \Gamma\backslash\Omega$ is a closed immersion of analytic spaces;
\item the Griffiths $\Q$-bundle $L$ restricted to $Y$ is the analytification of an ample algebraic $\Q$-bundle, and in particular $Y$ is a quasi-projective variety.

\end{enumerate}
\end{thm}

Note that if the period map $\phi$ is proper and if $X$ and the Griffiths bundle on $X$ are both defined over a subfield $k$ of $\C$ (for example, if the variation comes from a smooth projective family defined over $k$), then it follows that the first map $g:X\to Y'$ in the Stein factorization $X\xrightarrow{g}Y'\xrightarrow{h} Y$ of $f$ will also be defined over $k$, as this map is given by the complete linear system of sections of a high enough power of the natural extension of $L$ on a log smooth compactification of $X$ (see Theorem \ref{qproj}).

As a sample application, we have the following immediate corollary:

\begin{cor}\label{introcoarse} Let $\mathcal{M}$ be a reduced separated Deligne--Mumford stack of finite type over $\C$ admitting a quasi-finite period map.  Then the coarse moduli space of $\mathcal{M}$ is quasi-projective.
\end{cor}
The existence of the coarse moduli space is a general result of Keel--Mori \cite{KM}; see \S\ref{sectionapp} for a precise discussion of period maps on Deligne--Mumford stacks.  Corollary \ref{introcoarse} for instance will apply to a reduced separated Deligne--Mumford moduli stack of smooth polarized varieties with an infinitesimal Torelli theorem.  This provides an alternate approach to results of Viehweg \cite{Viehweg} on the quasi-projectivity of (normalizations of) coarse moduli spaces of smooth polarized varieties $X$ without assuming any positivity of $K_X$.  In particular, Corollary \ref{introcoarse} also applies to uniruled $X$ provided deformations can be detected by Hodge theory (for instance, low-degree complete intersections).

The strategy of the proof of Theorem \ref{maingriffiths} hinges on algebraization results in o-minimal geometry.  Briefly, an o-minimal structure specifies a class of ``tame" subsets of $\R^n$ with strong finiteness properties.  Such 
subsets are said to be definable with respect to the structure.  The resulting geometric category of complex analytic varieties that are pieced together by finitely many definable charts (which we call definable complex analytic varieties, see \S\ref{sectiondefscheme}) on the one hand allows 
some of the local  flexibility of the analytic category but on the other hand behaves globally like the algebraic category.  An excellent example of this is the celebrated ``definable Chow theorem" of Peterzil--Starchenko \cite[Corollary 4.5]{PSdefchow}, asserting that a 
closed complex analytic subvariety of a (not necessarily proper) complex algebraic variety which is definable in an o-minimal structure is in fact algebraic.

In \cite{BKT}, it is shown that $\Gamma\backslash \Omega$ is in this sense a definable complex analytic variety, and that period maps are definable with respect to this structure.  To prove the first part of Theorem \ref{maingriffiths}, we prove a ``dual" version of Peterzil--Starchenko's definable Chow theorem, showing that \emph{images} of algebraic spaces under definable proper complex analytic maps are algebraic: 
\begin{thm}\label{mainimage} Let $X$ be a separated algebraic space of finite type over $\C$, $\mathcal{S}$ a definable complex analytic space, and $\phi:X^\df\to \mathcal{S}$ a proper definable complex analytic map.  Then $\phi: X^\df\to \phi(X^\df)$ is (uniquely up to unique isomorphism) the definabilization of a morphism of algebraic spaces.
\end{thm}
To prove Theorem \ref{mainimage} we use Artin's theorems \cite{A} 
on the algebraization of formal modifications to inductively algebraize $\phi$ on strata. The category of algebraic spaces is needed in Artin's theorems and so is the natural setting for Theorem \ref{mainimage}---even if $X$ is an algebraic variety, the image may not be.  To apply Artin's theorems, one must necessarily consider nilpotent thickenings and thus deal with non-reduced spaces, even if one is only interested primarily in varieties.  In fact, the naive generalization of Theorem \ref{maingriffiths} to non-reduced spaces is false, as we show in Example \ref{needdef}. One of the benefits of working in the definable complex analytic category is that it provides a natural admissibility condition
to extend Theorem \ref{maingriffiths} to this setting, and we prove the more general statement in \S\ref{sectionperiodmap}. 

To algebraize the maps on nilpotent thickenings that arise when applying Artin's theorem, we develop a theory of coherent sheaves in the definable complex analytic category, and a GAGA-type theorem for definable coherent sheaves:

\begin{thm}\label{mainGAGA} Let $X$ be a separated algebraic space of finite type over $\C$ and $X^\df$ the associated definable complex analytic space.  The ``definabilization" functor $\Coh(X)\to\Coh(X^\df)$ is fully faithful, exact,
 and its essential image is closed under subobjects and quotients.
\end{thm}
It follows for example that definable coherent subsheaves of algebraic coherent sheaves are algebraic.  Note that $X$ is \emph{not} required to be proper over $\C$, but in contrast to Serre's classical GAGA theorem \cite{Serre} (as well as most other GAGA-type theorems for proper algebraic spaces), it is \emph{not} true that every definable coherent sheaf is algebraic (see Example \ref{countereg}).

Briefly, the proof of Theorem \ref{mainGAGA} is as follows.  One must first develop the theory of coherent sheaves on definable complex analytic spaces, and in particular prove an Oka coherence theorem (on the coherence of the structure sheaf, see Theorem \ref{coherenceofo}) as well as a Nullstellensatz (Theorem \ref{nullstellensatz}) in this category.  The key point is to carefully keep track of the open refinements of covers needed in the classical proofs in the complex analytic category, and to show that in fact definable (in particular \emph{finite}) refinements suffice.  With the sheaf theory in place, the main claim of Theorem \ref{mainGAGA} (that definable coherent subsheaves of algebraic coherent sheaves are algebraic) follows inductively using the Nullstellensatz from the fact that definable vector subbundles of an algebraic vector bundle are algebraic, by applying Peterzil--Starchenko's definable Chow theorem to the associated geometric total space. 

A tempting alternative to the use of Theorem \ref{mainimage} and the tameness of the period map is provided by the result of Cattani--Deligne--Kaplan mentioned earlier:  one could try to prove that a surjective proper complex analytic map $X^\an\to\mathcal{S}$ from an algebraic variety to an analytic variety with algebraic equivalence relation\footnote{It is important to include the natural scheme structure on the equivalence relation.} $X^\an\times_{\mathcal{S}}X^\an\subset X^\an\times X^\an$ is algebraic.  This is not true at this level of generality---see Example \ref{algbraizing:counterexample} and the surrounding discussion.

It is in general difficult to relate the metric positivity of the Hodge bundle to its ampleness on $Y^\an$ as the latter might be quite singular, and this has been the main obstacle in proving the conjecture directly from the positivity.  One can, however, use it to show $L$ is big and nef on a log resolution, and once $Y^\an$ is known to be algebraic, algebraic sections from a resolution can be descended to deduce the second statement in Theorem \ref{maingriffiths}.

Theorem \ref{maingriffiths} combined with the o-minimal algebraization results have a number of applications, and we describe a few in the final section including:  
\begin{enumerate}
\item A version of the Borel algebraicity theorem for period images (section \ref{sectborel}).
\item As a concrete example of Corollary \ref{introcoarse}, we deduce a general result about the quasi-projectivity of moduli spaces of complete intersections (section \ref{sectcomplete}).
\item A theorem showing that pure polarized integral variations of Hodge structures over dense Zariski open subsets of compact \emph{K\"ahler manifolds} are pulled back from algebraic varieties (section \ref{sectkahler}).
\item A version of the ampleness result in Theorem \ref{maingriffiths} for the \emph{Hodge} bundle (section \ref{secthodge}).
\end{enumerate}
\subsection{Previous results}

Griffiths proved his conjecture in the case that the image $\phi^{\an}(X^{\an})$ is compact \cite[III.9.7]{G2}. Sommese \cite{Som1} proved the conjecture in the case that the image has only isolated singularities, and later \cite{Som2} proved a function field variant. In particular, he proved that the image of a period map admits a proper desingularization which is quasi-projective and such that the induced meromorphic map is rational. However, for example it does not follow from their works that period images admit a compactification by a compact analytic space.

The subject of o-minimal sheaves and the development of a cohomology theory were treated in \cite{EJP}, and this was further developed
in subsequent papers. Variants of the o-minimal Nullstellensatz and Weierstrass preparation theorems were proven by Kaiser \cite{K}.

Kashiwara--Schapira \cite{KS} have constructed a subanalytic site as well as a theory of subanalytic sheaves which is in general different from our construction in \S\ref{sectiondefscheme} for the subanalytic o-minimal structure $\R_{\mathrm{an}}$---see the end of \S\ref{sectiondefscheme} for a more precise discussion.  Petit \cite{P} has defined a ``tempered analytification" functor on smooth algebraic varieties and proven a conditional GAGA theorem reminiscent of Theorem \ref{mainGAGA} on the subanalytic sites of smooth algebraic
varieties in the sense of Kashiwara--Schapira.

\subsection{Outline}  In \S\ref{sectiondefscheme} we develop the theory of definable coherent sheaves and definable complex analytic spaces.  We also define and prove some basic properties of the definabilization functor on algebraic spaces and the analytification functor on definable complex analytic spaces.  In \S\ref{sectiondefGAGA} we prove Theorem \ref{mainGAGA} (see Theorem \ref{gaga}), and in section \ref{sectiondefimage} we prove Theorem \ref{mainimage} (see Theorem \ref{defpropmap}).  We prove a general quasi-projectivity criterion in \S\ref{sectionqpcrit}.  In \S\ref{sectionperiodmap} we apply the results of \S\ref{sectiondefimage},\ref{sectionqpcrit} to prove a stronger version of Theorem \ref{maingriffiths} allowing for non-reduced bases (see Theorem \ref{hodgebetter} and Theorem \ref{qproj}).  In \S\ref{sectionapp} we deduce some applications, including Corollary \ref{introcoarse} (see Corollary \ref{coarse}).

\subsection{Acknowledgements}

J.T. would like to thank Vivek Shende, Jonathan Pila, and Ryan Keast for useful conversations. B.B. 
would like to thank Valery Alexeev and Johan de Jong for useful conversations. Y.B. would like to thank Olivier Benoist, Patrick Brosnan, and Wushi Goldring for useful conversations.  The authors would also like to thank Ariyan Javanpeykar for useful remarks, specifically regarding \S\ref{sectborel}. This paper, and in 
particular \S\ref{sectiondefscheme} owes a lot to the works of Peterzil and Starchenko, who initiated the study of o-minimal complex geometry. B.B. was partially supported by NSF grant DMS-1702149.  The authors are indebted to the referees for their careful reading and for greatly improving the exposition.

\subsection{Notation}
All schemes and algebraic spaces are assumed to be separated and of finite type over $\C$, and all definable topological spaces, definable complex analytic spaces, and analytic spaces are assumed to be Hausdorff.  When helpful (mostly in \S\ref{sectiondefGAGA}, \S\ref{sectiondefimage}, and \S\ref{sectionperiodmap}), we will loosely adopt the convention that algebraic objects are denoted by roman letters, and (definable) analytic objects by script letters.

Throughout, we fix an o-minimal structure with respect to which we will use the word ``definable".  The reader unfamiliar with these notions may assume for concreteness the structure $\R_{\mathrm{alg}}$ for which the definable subsets of $\R^n$ are the real semi-algebraic subsets.  
For the applications to Hodge theory in \S \ref{sectionperiodmap}  we restrict to the o-minimal
structure  $\R_{\an,\exp}$. For a general introduction to o-minimality, see \cite{Vdd}, and \cite{VddM} for a discussion of o-minimality in a
similar language to this paper. 

\section{definable complex analytic spaces}\label{sectiondefscheme}

\subsection{Definable topological spaces}\label{subsectdeftop}
Definable subsets $U\subset\R^n$ have important finiteness properties.  To develop a theory of topological spaces which are locally modeled on definable sets and which preserves these finiteness properties, it is important to insist that only finite covers by open sets are used.  

We begin with a straightforward definition (cf. \cite[Chapter 10]{Vdd}):
\begin{defn}Let $X$ be a topological space.  A \emph{definable atlas} $\{(U_i,\phi_i)\}$ for $X$ is a finite open covering $\{U_i\}$ of $X$ and homeomorphisms $\phi_i:U_i\xrightarrow{\cong} V_i\subset\R^{n_i}$ such that
	\begin{enumerate}
	\item The $V_i$ and the pairwise intersections $V_{ij}:=\phi_i(U_i\cap U_j)$ are definable;
	\item The transition functions $\phi_{ij}:=\phi_j\circ \phi_i^{-1}:V_{ij}\to V_{ji}$ are definable.	
	\end{enumerate}
	For topological spaces $X,Y$ equipped with definable atlases $\{(U_i,\phi_i)\}, \{(U'_{i'},\phi'_{i'})\}$, we say a map $f: X\to Y$ is definable if for all $i$ and $i'$ the composition
	\[\phi_i(U_i\cap f^{-1}(U'_{i'}))\xrightarrow{\phi_i^{-1}}f^{-1}(U'_{i'})\xrightarrow{f}U'_{i'}\xrightarrow{\phi'_{i'}}V'_{i'}\]
	is definable.  Note that this is a condition both on the source and the map.
	
	Finally, we say two atlases $\{(U_i,\phi_i)\}, \{(U'_{i'},\phi'_{i'})\}$ on $X$ are equivalent if the identity $\id:X\to X$ is definable with respect to $\{(U_i,\phi_i)\}$ on the source and $ \{(U'_{i'},\phi'_{i'})\}$ on the target.
	
	\end{defn}
	\begin{defn}
	A \emph{definable topological space} $X=(|X|,\xi_X)$ is a Hausdorff topological space $|X|$ with a choice of equivalence class $\xi_X$ of definable atlases on $|X|$.  A morphism $f:X\to Y$ of definable topological spaces is a continuous map $|f|:|X|\to |Y|$ which is definable with respect to any choice of atlases in $\xi_X,\xi_Y$.  We denote the category of definable topological spaces by $\mbox{(DefTopSp)}$, suppressing the implicit o-minimal structure. 
	\end{defn}

There is an obvious functor $|\cdot|:\mbox{(DefTopSp)}\to\mbox{(TopSp)}$ to the category of topological spaces sending $X$ to $|X|$.   Given a topological space $S$, we refer to a lift of $S$ to $\mbox{(DefTopSp)}$ as a definable structure on $S$.  If $X$ is a definable topological space, we say a subspace $T\subset |X|$ is definable (in $X$) if $\phi_i(T\cap U_i)\subset\R^{n_i}$ is definable for all $i$.  In this case there is a natural definable structure $Z$ on $T$ for which the inclusion $Z\to X$ is a morphism, and it is the unique one with this property.  We refer to such a $Z$ as a definable subspace $Z\subset X$, and we often blur the notational distinction between definable subspaces $Z\subset X$ and subspaces $Z\subset|X|$ which are definable (in $X$).  Note that for a definable topological space $X$ and a choice of atlas $\{(U_i,\phi_i)\}$, the open sets $U_i\subset |X|$ have natural definable structures as open definable subspaces $U_i\subset X$.

If $X,Y$ are definable topological spaces, $X\times Y$ naturally acquires the structure of a definable topological space, and we say a map $\phi:|X|\to|Y|$ is definable (in $X$ and $Y$) if the graph is in $X\times Y$.  One easily shows that a morphism $f:X\to Y$ of definable topological spaces is equivalent to a definable continuous map $|f|:|X|\to|Y|$, and that images and preimages of definable subsets under a morphism $f:X\to Y$ are definable.

We finish this section by studying finite maps in the definable category.  Recall that a topological space $X$ is regular if for every point $x\in X$ and open $x\in U\subset X$ there is an open $x\in V\subset U$ such that the closure $\bar V$ of $V$ in $X$ is contained in $U$.  We say a definable topological space is regular if the underlying topological space is. 

\begin{defn}

Let $f:X\ra Y$ be a morphism of definable topological spaces. We say that $f$ is \emph{quasi-finite} if $|f|$ has finite
fibers, and \emph{proper} if $|f|$ is.  We say $f$ is \emph{finite} if it is quasi-finite and proper.

\end{defn}

\begin{prop}\label{propercover}

Let $f:X\to Y$ be a finite morphism of regular definable topological spaces, and let $\{X_i\}$ 
be a definable open cover of $X$.   Then there is a definable open cover $\{W_j\}$ refining $\{X_i\}$ and a definable open cover $\{Y_k\}$ of $Y$ such that each $f^{-1}(Y_k)$ is a disjoint union of $W_j$.

\end{prop}
\begin{proof}
By \cite[Chapter 10 \S 1.8]{Vdd} $X$ (resp. $Y$) can be definably embedded as a definable subspace of $\R^m$ (resp. $\R^n$).  Moreover, by passing to a definable cover of $Y$, we may assume there is a coordinate of $\R^m$ which separates the points in each fiber of $f$.  Thus by projecting we may assume $X\subset Y\times\R$ and $f$ is the first projection.

\def\tot{\operatorname{tot}}

Recall that a definable triangulation of a definable topological space $X$ is a definable homeomorphism $\Phi:X\to\tot(K)$ for a (finite) simplicial complex $K$ (see \cite[Chapter 8]{Vdd}).
\begin{lem}
Let $Y\subset \R^n$ be a definable set, $X\subset Y\times\R$ a definable set such that the first projection $f:X\to Y$ is proper.  Let $\{A_i\}$ be a finite set of definable subsets of $X$.  Then there exist definable triangulations of $X$ and $Y$ such that
\begin{enumerate}
    \item each $A_i$ is a subcomplex of $X$ with respect to the triangulation;
    \item for each open simplex $D$ of $Y$, $f^{-1}(D)$ is a disjoint union of open simplices of $X$, each mapping isomorphically to $D$;
    \item the closure of each simplex of $X$ injects into $Y$.
\end{enumerate}
\end{lem}
\begin{proof}
Applying normal definable cell decomposition to $X$ \cite[Chapter 3 \S 2.11]{Vdd}, we obtain cell decompositions of $X$ and $Y$ such that each $A_i$ is a union of cells of $X$ and each cell of $X$ is the graph of a continuous definable function over a cell of $Y$.  In particular, for any cell $D$ of $Y$ the preimage $f^{-1}(D)$ is a disjoint union of cells each mapping isomorphically to $D$.

By definable triangulation \cite[Chapter 8 \S 2.9]{Vdd}, there is a definable triangulation of $Y$ for which each of the above cells of $Y$ is a subcomplex.  By \cite[Chapter 8 \S 2.8]{Vdd} this triangulation lifts to $X$, and clearly satisfies properties (1) and (2).  Note that in the terminology of \cite{Vdd}, the properness of $f$ guarantees the multivalued function $\pi_2\circ f^{-1}$ is closed via \cite[Chapter 8 \S 2.6]{Vdd}, and we may reduce to the case that $\pi_2\circ f^{-1}$ is full as in the proof of \cite[Chapter 8 \S 2.9]{Vdd}.

By taking the barycentric subdivisions of these triangulations, properties (1) and (2) still hold, and we claim we additionally have property (3).  Indeed, the closure of each simplex of the subdivision of a simplex $\Delta$ only intersects one face of $\Delta$ of each dimension, and so (3) follows from (2).
\end{proof}
Let $\{C_j\}$ (resp. $\{D_k\}$) be the open simplices of the triangulation of $X$ (resp. $Y$) guaranteed by the lemma, taking $\{A_i\}=\{X_i\}$.  For each $C\in\{C_j\}$, let $X(C)$ be the union of open simplicies in $\{C_{j}\}$ having $C$ as a face; likewise for $D\in\{D_{k}\}$ define $Y(D)$.  We claim that $\{Y_k\}=\{Y(D_{k})\}$ and $\{W_j\}=\{X(C_{j})\}$ are the desired open covers.

First, it is clear that each $X(C)$ is definable and open in $X$, and likewise for each $Y(D)$.  Moreover, if $C\subset X_i$ then $X(C)\subset X_i$, so $\{W_j\}$ refines $\{X_i\}$.  Next, suppose $D\in\{D_{k}\}$, and $C,C'\in\{C_{j}\}$ are distinct open simplices in $X$ mapping to $D$. By property (3) of the lemma no open simplex has both $C$ and $C'$ as faces, so $X(C)$ and $X(C')$ are disjoint. 

We finally claim that $f^{-1}(Y(D))$ is the disjoint union of $X(C)$ for open simplices $C$ of $X$ mapping to $D$. To see this, if is sufficient to know that if $D'$ is an open simplex having $D$ as a face, then every lift
$C'$ of $D'$ has \emph{some} lift of $D$ as a face. This is immediate by the properness of $f$ and property (2), and the proof is therefore complete.
\end{proof}

\begin{remark}
\label{simply}
By definable triangulation \cite[Chapter 8, \S 2.9]{Vdd}, any finite definable open cover of a definable topological space can be refined by a finite cover by simply-connected definable open subsets.
\end{remark}

\subsection{Sheaves on definable topological spaces}In this section we collect some basic notions regarding sheaves on definable topological spaces.  Because of the insistence on finite covers, the sheaf theory requires a very mild use of Grothendieck topologies.  

\def\Ab{\operatorname{Ab}}
\begin{defn}Let $X$ be a definable topological space.  The \emph{definable site} $\underline{X}$ of $X$ is the site whose underlying category is the category of definable open subsets of $X$ (with inclusions as morphisms) and whose coverings are \emph{finite} coverings by definable open sets.  
\end{defn}
We sometimes abusively refer to sheaves on the definable site as sheaves on $X$.  Given a morphism $f:X\to Y$ of definable topological spaces, there are in the usual way adjoint functors $f_*:\Ab(\underline{X})\to \Ab(\underline{Y})$ and $f^{-1}:\Ab(\underline{Y})\to\Ab(\underline{X})$ on the categories of abelian sheaves.

\begin{remark}\label{rmk: no points}We remark that exactness in $\Ab(\underline{X})$ \emph{cannot} be checked on stalks.  See Example \ref{eg: no points}.  There is a space obtained by adjoining model-theoretic ``generic points" called types whose conventional category of sheaves is equivalent to sheaves on the definable site, and this is the perspective taken by, e.g., Edmundo--Jones--Peatfield \cite{EJP}.  In particular, exactness can be checked on stalks if we include these additional points.
\end{remark}

From Proposition \ref{propercover} we deduce the following:

\begin{cor} \label{finiteexact} Let $f:X \to Y$ be a finite morphism of regular definable topological spaces.  Then $f_*:\Ab(\underline{X})\to\Ab(\underline{Y})$ is exact.
\end{cor}
\begin{proof}  Let $A\to B\to C$ be an exact sequence of sheaves on $X$; we want to prove the exactness of $f_*A\to f_*B\to f_* C$.  For a definable open $U$ of $Y$, if a section $s$ in $f_*B(U)=B(f^{-1}(U))$ is zero in $f_*C(U)$, then after taking an open definable cover of $f^{-1}(U)$, $s$ is in the image of $A$.  By Proposition \ref{propercover} we refine our open definable cover
by components of $f^{-1}(Y_k)$, where $\{Y_k\}$ is an open cover of $Y$. It follows that for each
$k$, $s|_{Y_k}$ is in the image of $ f_*A(Y_k)$, completing the proof.
\end{proof}

\begin{defn}  A locally $\C$-ringed definable space $(X,\O_X)$ is a definable topological space $X$ and a sheaf $\O_X$ of $\C$-algebras on the definable site $\underline{X}$ whose stalks are local rings.  A morphism of locally $\C$-ringed definable spaces $f:(X,\O_X)\to(Y,\O_Y)$ is a morphism $f:X\to Y$ of definable spaces and a morphism $f^\sharp:f^{-1}\O_Y\to\O_X$ of sheaves of $\C$-algebras which is local on stalks.
\end{defn}

\begin{remark}
In general some care must be taken to define a locally ringed site when the site does not have enough points, see for example the discussion surrounding \cite[\href{https://stacks.math.columbia.edu/tag/04EU}{Tag 04EU}]{stacks-project}.  For our purposes the above definition will suffice.
\end{remark}
\begin{remark}\label{rmk closed immersions}
The notions of closed and open immersions of locally ringed spaces naturally generalize to locally $\C$-ringed definable spaces.  See for example \cite[\href{https://stacks.math.columbia.edu/tag/01HK}{Tag 01HK},\href{https://stacks.math.columbia.edu/tag/01HE}{Tag 01HE}]{stacks-project}.
\end{remark}

For $X$ a locally $\C$-ringed definable space, denote by $\Module(\O_X)$ the abelian category of $\O_X$-modules.
Given a morphism $f:X\to Y$ of locally $\C$-ringed definable spaces, we naturally have a functor $f_*:\Module(\O_X)\to\Module(\O_Y)$, and we define a functor $f^*:\Module(\O_Y)\to\Module(\O_X)$ via
\[f^*:F\mapsto \O_X\otimes_{f^{-1}\O_Y} f^{-1}F\]   
where as usual we have used the adjoint map $f^\sharp:f^{-1}\O_Y\to\O_X$ to make $\O_X$ an $f^{-1}\O_Y$-algebra.

\begin{defn}\label{def coh}Let $X$ be a locally $\C$-ringed definable space.  Given an $\O_X$-module $M$, we say that $M$ is \emph{of finite type} (as an $\O_X$-module) if there exists a definable cover $X_i$ of $X$ and surjections $\O_{X_i}^n\rra M_{X_i}$ for some positive integer $n$ on each of  
those open sets. We say $M$ is \emph{of finite presentation} (as an $\O_X$-module) if there is a definable cover $X_i$ of $X$ and finite presentations
\[\O^m_{X_i}\to\O^n_{X_i}\to M_{X_i}\to 0.\]We say that $M$ is \emph{coherent} (as an $\O_X$-module) if it is of finite type, and given any definable open $U\subset X$ and any $\O_U$-module homomorphism $\phi:\O_U^n\ra M_U$, the kernel of $\phi$ is of finite type.  
\end{defn}
Note that it easily follows that if $M$ is a coherent $\O_X$-module and $N\subset M$ is an $\O_X$-submodule of finite type, then $N$ is coherent.  Moreover, the kernel of any homomorphism $M\to M'$ of coherent $\O_X$-modules is of finite type and therefore coherent.  The following is also standard and we include the proof to give a flavor of the types of arguments used.
\begin{lemma}\label{basiccoherence}

Let $0\ra M_1\ra M \ra M_2 \ra 0$ be an exact sequence of sheaves on a locally $\C$-ringed definable space $X$. If two of $\{M,M_1,M_2\}$ are coherent then so is the third.

\end{lemma}

\begin{proof}\hspace{1in}

\begin{enumerate}

\item Assume $M,M_1$ are coherent. Since $M$ is of finite type, so is $M_2$. Let us show that $M_2$ is coherent.  Suppose $V\subset X$ is a definable open and $\phi:\O_V^n\ra {M_2}_{|V}$ is any map. 
The map $\phi$ is determined by the image of a basis. Since $M$ surjects onto $M_2$, by further restricting to a finite open cover we can assume that $\phi$ lifts to a map $\phi':\O_V^n\ra M_{|V}$. 

Since $M_1$ is coherent we may choose a surjection $\psi:\O_V^m\ra {M_1}_{|V}$ by further restricting to a finite open cover. Consider $\psi\oplus\phi': \O_V^m\oplus\O_V^n\ra M_{|V}$. Then the kernel of $\psi\oplus\phi'$ is finitely generated since $M$ is coherent, and surjects onto the kernel of $\phi$. Thus the kernel of $\phi$ is finitely generated, and so $M_2$ is coherent.

\item Assume $M,M_2$ are coherent. Then any map to $M_1$ is also a map to $M$, and thus has finitely generated kernel. Moreover, if $\phi:\O_X^n\rra M$, then the kernel of the induced map to $M_2$ is finitely generated since
$M_2$ is coherent, and surjects to $M_1$.

\item Assume $M_1,M_2$ are coherent. To see that $M$ is of finite type, we first restrict to a finite open covering so that one can choose surjections $\phi_i:\O_X^{n_I}\ra M_i$. By further restricting, we may lift $\phi_2$ to 
a map $\phi'_2:\O_X^{n_2}\ra M$. Now the map $\phi_1\oplus \phi_2':\O_X^{n_1+n_2}\ra M$ is a surjection.

 Finally, let $\phi:\O_X^m\ra M$ be any map.  When continued to $M_2$, the kernel $K$ of $\phi_0:\O_X^m\rra M_2$ is of finite type. The induced map from $K$ to $M_1$ has kernel which is of finite type, and 
 this kernel is in fact $\ker\phi$. This completes the proof.
\end{enumerate}
\end{proof}

\begin{cor}\label{cor coh is ab}The full subcategory $\Coh(\O_X)\subset\Module(\O_X)$ of coherent $\O_X$-modules is an extension closed abelian subcategory.
\end{cor}
\begin{proof}By the lemma and the remarks after Definition \ref{def coh}.
\end{proof}

\begin{cor}\label{cor fp=coh}Assume $\O_X$ is a coherent $\O_X$-module.  Then:
\begin{enumerate}
\item $\O_X^n$ is coherent for any $n$.
\item An $\O_X$-module $M$ is coherent iff it is of finite presentation.
\end{enumerate}
\end{cor}

\subsection{Basic definable complex analytic spaces}Identify $\C=\R^2$ using the real and imaginary parts, and give $\C^n$ the definable structure coming from the identification $\C^n=\R^{2n}$.  For a definable open set 
$U\subset \C^n$ we let $\O_{\C^n}(U)$ be the definable holomorphic functions on $U$, that is the maps $U \rightarrow \C$ that are both definable and holomorphic. 

\begin{lemma}

The presheaf $\O_{\C^n}:\underline{\C^n}\to \mathrm{Ab} $ which to $U\in\underline{\C^n}$ associates $\O_{\C^n}(U)$ is a sheaf on $\underline{\C^n}$.

\end{lemma}

\begin{proof}
Let $U\subset\C^n$ be a definable open set and let $U_i$ be a finite definable covering of $U$. If a function $f\in \O_{\C^n}(U)$ vanishes on each $U_i$, it must be identically 0. Moreover, if $f_i$ are definable holomorphic functions on $U_i$ which agree on overlaps, they by analytic continuation glue to a 
single holomorphic function $f$ on $U$. Since the $U_i$ are a finite covering of $U$ and each $f_i$ is definable, it follows that $f$ is also definable and hence $f\in\O_{\C^n}(U)$ as required.
\end{proof}
\def\colim{\operatorname{colim}}
Note that the stalks $\O_{\C^n,x}:=\colim_{x\in U} \O_{\C^n}(U)$ are local rings.

\begin{eg}\label{eg: no points} The sheaf $\mathcal{O}_{(\C^n)^\an}$ of holomorphic functions is a sheaf on $\underline{\C^n}$.  If our structure contains $\R_\an$, then $\O_{\C^n}\subset\O_{(\C^n)^\an}$ have the same stalks but are not equal, and therefore exactness on the definable site cannot be checked on stalks.  Crucially, we will show (see Corollary \ref{cor exact on stalk}) that exactness in $\Coh(\O_{\C^n})$ \emph{can} be checked on stalks.
\end{eg}
\begin{defn}

Given an open definable subset $U\subset \C^n$ and a finitely generated ideal $I$ of $\O_{\C^n}(U)$, the vanishing locus\footnote{Here we use $|\cdot|$ to denote the vanishing locus as a definable topological space---that is, forgetting the sheaf of functions---rather than the underlying topological space as in \S\ref{subsectdeftop}.} $X=|V(I)|$ is naturally a definable topological space.  We call the data of $ U\subset\C^n$ and $I$ a \emph{basic definable complex analytic space}.  We often refer to the basic definable complex analytic space via $X\subset U\subset\C^n$, and denote by $I_X:=I\O_U$. 

There is a sheaf
$\O_{U}/I_X$ on $\underline{U}$ which is supported on $\underline{X}$. We set $\O_X$ to be the restriction of $\O_U/I_X$ to $\underline{X}$, and refer to the pair $(X,\O_X)$ as the associated locally $\C$-ringed definable space.
\end{defn}
\begin{remark}We will eventually see in Corollary \ref{cor sub ringed} that given two basic definable complex analytic spaces $X\subset U\subset \C^n$ and $Y\subset V\subset\C^n$, a morphism of the associated locally $\C$-ringed definable spaces $(X,\O_X)\to(Y,\O_Y)$ is, after passing to a definable cover of $X$ in $U$, the natural one induced by a definable holomorphic map $f:U\to V$ for which $f^\sharp (f^{-1}I_Y)\subset I_X$.  This will allow us to glue basic definable complex analytic spaces by gluing the $\C$-locally ringed definable spaces.
\end{remark}

\subsection{Definable Oka coherence}\label{sect oka}  In this section we prove the analog of the Oka coherence theorem \cite[Chapter 2 \S5.2]{Grauert} for basic definable complex analytic spaces:

\begin{thm}\label{coherenceofoCn}
The definable structure sheaf $\O_{\C^n}$ of $\C^n$ is a coherent $\O_{\C^n}$-module.
\end{thm}

The statement of Theorem \ref{coherenceofoCn} is local.  The proof will largely follow the classical proof (e.g. \cite[Chapter 2 \S5]{Grauert}) by observing that whenever one must pass to a refinement of an open cover in the classical setting, a definable refinement is sufficient in our setting.  One example is the following definable version of Weierstrass division:
\begin{lemma}\label{Weierstrass}

Let $V\subset\C^n$ be a definable open set, $P\in\O_{\C^n}(V)[w]$ a monic polynomial in $w$ with coefficients that are definable holomorphic functions on $V$.  Let $U\subset V\times \C$ 
be a definable open set containing $X:=|V(P)|\subset V\times \C$.  Then given any definable holomorphic function $f$ on $U$, one can uniquely write $f=QP+R$ for definable holomorphic functions $Q,R$ on $U$ with $R\in \O_{\C^n}(V)[w]$ of degree 
less than the degree of $P$.

\end{lemma}

\begin{proof}
The claimed $Q,R$ exist uniquely in the analytic category \cite[Chapter 2 \S 1.2]{Grauert}, so it suffices to prove they are definable.  Let $X_i$ be the irreducible analytic components of $X$ and $P_i$ be the minimal polynomial of $w$ over $X_i$. Note that the $X_i$ are definable sets and so each $P_i$ is definable. Also, $P$
must be a product of the $P_i$, and so by induction on $\deg P_i$ it suffices to prove the theorem for each $P_i$ one at a time. 
We may thus assume that $P$ is irreducible. Let $V_1\subset V$ be the dense open set where $P(w)$ has distinct roots, which is definable. On $V_1$ the coefficients of $R$ are the $a_0,\dots,a_{n-1}\in\O_{\C^n}(V_1)$ such that 
$\sum_{i=0}^{n-1} a_iw^i$ agrees with $f$ on $X$, where $n=\deg P$. Thus, it follows that $R_1:=R|_{V_1}$ is definable.  Since $U_1:=V_1\cap U$ is dense in $U$ it follows that $R$ is definable as well, since the graph of $R$ is the closure of the graph of $R_1$.
Hence $Q$ is definable since $Q=\frac{f-R}{P}$, and the proof is complete.
\end{proof}

Another important input of a similar flavor is a definable version of Noether normalization:
\begin{thm}[Peterzil--Starchenko {\cite[Theorem 2.14]{PS}}]\label{thm proper coord}Given an open definable subset $U\subset\C^n$ and a closed definable complex analytic subset $X\subset U$ of dimension $d$, there is a definable cover $\{U_i\}$ of $U$ and linear projections $\pi_i:\C^n\to \C^d$ such that the restrictions $p_i:X_i\to\pi_i(U_i)$ are finite, where $X_i=X\cap U_i$.
\end{thm}

\begin{proof}[Proof of Theorem \ref{coherenceofoCn}]
Following \cite[Chapter 2 \S 5.1]{Grauert}, we start with a coherence criterion.
\begin{lemma}The following are equivalent.
\begin{enumerate}
\item For any connected open definable $U\subset \C^n$ and any nonzero definable holomorphic function $f\in\O_{\C^n}(U)$ we have that $M=\O_U/f\O_U$ is a coherent $M$-module.
\item $\O_{\C^n}$ is a coherent $\O_{\C^n}$-module.
\end{enumerate}
\end{lemma}
\begin{proof}The backward implication is immediate from Corollary \ref{cor coh is ab}.  For the forward implication, suppose $U\subset \C^n$ is a definable open which we may assume is connected and $\phi:\O_U^m\to\O_U$ an $\O_U$-module homomorphism given by $f_1,\ldots,f_m\in\O_{\C^n}(U)$.  Evidently $\ker(\phi)$ is finitely generated if all the $f_i$ vanish, so we may assume without loss of generality that $f=f_1$ is nonzero.

Consider the projection $\pi:\O_U\to M:=\O_U/f\O_U$ and note we have a commutative diagram
\[\xymatrix{
\O_U^m\ar[d]_{\pi^m}\ar[r]^\phi&\O_U\ar[d]^\pi\\
M^m\ar[r]^{\bar{\phi}}&M.
}\]
The vertical maps clearly have finitely generated kernels (as $\O_U$-modules).  As $M$ is coherent by hypothesis, $\ker (\bar{\phi})$ is of finite type as an $M$-module (and therefore also an $\O_U$-module), and it follows by lifting the generators (possibly after passing to a finite refinement) that $\ker(\pi\circ\phi)$ is a finite type $\O_U$-module.  The $\O_U$-module homomorphism $s\mapsto s-(\phi(s)/f)e_1$ gives a section of the inclusion $\ker(\phi)\to\ker(\pi\circ\phi)$, and therefore $\ker(\phi)$ is of finite type.
\end{proof}
It therefore suffices to prove the criterion in the lemma; we do so by induction on $n$, the case $n=0$ being obvious.  We thus assume that $\O_{\C^{n-1}}$ is coherent.  

Let $U\subset \C^n$ be a connected definable open set and $f\in\O_{\C^n}(U)$ nonzero.  Let $X:=|V(f)|$ be the zero set of $f$.  Using Lemma \ref{thm proper coord}, there is a covering of $U$ by finitely many definable open sets $U_i$ such that for each $U_i$ there 
is a linear set of coordinates for which $X_i=X\cap U_i$ is finite over its projection down to $\C^{n-1}$.  Replacing $X$ with $X_i$ we may therefore assume without loss of generality that there is a linear projection $\pi:U\to\C^{n-1}$ whose restriction $p:X\to V$ is finite over its image $V:=\pi(U)\subset\C^{n-1}$. It follows that $V$ is a definable open set of $\C^{n-1}$.

Let $W_j$ be the irreducible components of $X$, and let $P_j(w)\in\O_{(\C^{n-1})^\an}(V)[w]$ be the unique irreducible polynomials whose zero-locus is $W_j$. Note that the coefficients of $P_j(w)$ are definable since it can be defined on the dense definable open $W'_j\subset W_j$ where $W_j\to \pi(W_j)$ is \'etale as 
$P_j(v,w)=\prod_{(v,t)\in W_j'}(w-t)$. By the analytic Weierstrass preparation theorem, there are positive integers $k_i$ such that $\frac{f}{\prod_i P_i(w)^{k_i}}$ is nowhere vanishing, and thus must be a definable unit. Hence
we may assume $f=P(w):=\prod_i P_i(w)^{k_i}$.

Let $k=\deg P$ and let $M=\O_U/f\O_U$, which we consider (by restricting) as a sheaf on $X$.  It suffices to show that the kernel of any homomorphism $\phi:M^n\to M$ of $M$-modules is of finite type (as an $M$-module).  By Lemma \ref{Weierstrass} we have $p_*M\cong \O_V^k$ as $\O_V$-modules.  Thus, $p_*\phi$ is a homomorphism of coherent $\O_V$-modules by the inductive hypothesis and Corollary \ref{cor fp=coh}, hence $\ker(p_*\phi)$ is of finite type as an $\O_V$-module.  By Corollary \ref{finiteexact} (any definable subspace of $\R^n$ is regular) we have $p_*\ker(\phi)=\ker(p_*\phi)$, and if $p_*\ker(\phi)$ is of finite type as an $\O_V$-module clearly $\ker(\phi)$ is of finite type as an $M$-module.
\end{proof}

\begin{cor}\label{coherenceofobasic}
For a basic definable complex analytic space $X$ the structure sheaf $\O_X$ is a coherent $\O_X$-module.
\end{cor}

\begin{proof}
First, for any definable open $U\subset \C^n$, $\O_U$ is clearly coherent.  Let $X=V(I_X)$ where $U\subset \C^n$ is definable open and $I_X\subset \O_U$ is a finitely generated subsheaf. Let $i:X\ra U$
be the natural injection, which is a closed immersion on locally $\C$-ringed definable spaces. Note that $M\ra i_*M$ gives an equivalence of categories between $\O_X$-modules on $X$ and $\O_U$-modules on $U$ killed by $I$, with inverse $M\ra i^{-1}M$.  By definition $i_*\O_X=\O_U/I_X$ is a finitely presented $\O_U$-module and therefore coherent by Theorem \ref{coherenceofoCn} and Corollary \ref{cor coh is ab}.  

As any open subset of $X$ is itself a basic definable complex analytic space, it is enough to check that for an $\O_X$-module homomorphism $\phi:\O_X^m\ra \O_X$ the kernel is of finite type.  We may consider $i_*\phi:i_*\O_X^m\ra i_*\O_X$ which
is a map of coherent $\O_U$-modules. Since $\O_U$ is coherent, we may locally form an exact sequence
$\O_U^t\ra i_*\O_X^m\ra i_*\O_X$. The first map is killed by $I$, so we get an exact sequence
$i_*\O_X^t\ra i_*\O_X^m\ra i_*\O_X$, and thus an exact sequence 
$\O_X^t\ra \O_X^m\ra \O_X$ as desired.
\end{proof}

\subsection{Analytification}\label{sect an}
\renewcommand{\An}{(-)^\an}
Given a basic definable complex analytic space $X\subset U\subset\C^n$, we may naturally consider $X$ as an analytic space, which we denote $X^\an$.  We for simplicity denote $\Coh(X):=\Coh(\O_X)$ and $\Coh(X^\an):=\Coh(\O_{X^\an})$.  There is a natural morphism $g:(X^\an,\O_{X^\an})\to(\underline{X},\O_X)$ of locally $\C$-ringed sites, and a resulting analytification functor $\An:\Coh(X)\to\Coh(X^\an)$ 
given by $F^\an:=\O_{X^\an}\otimes_{g^{-1}\O_X}g^{-1}F$ together with a natural identification $\O_X^\an\cong\O_{X^\an}$.  

\begin{eg}\label{eg local ring}It is instructive to observe that if the underlying o-minimal structure contains $\R_{\an}$, then $\An:\Coh(X)\to\Coh(X^\an)$ is just sheafification in the analytic topology.  In particular, $\An$ is exact and for any $x\in X$ we canonically have $\O_{X,x}=\O_{X^\an,x}$.
\end{eg}
Given the above example, the more contentful part of the following result is the faithfulness statement.

\begin{thm}\label{basicanfaithful}Let $X$ be a basic definable complex analytic space and $\An:\Coh(X)\to\Coh(X^\an)$ the analytification functor.  Then
\begin{enumerate}
\item $\An$ is exact;
\item $\An$ is faithful.
\end{enumerate}
\end{thm}

For the proof of Theorem \ref{basicanfaithful}, we first need some preliminary observations.

\begin{lemma}\label{noetherian}

For $X$ a basic definable complex analytic space and $x\in X$, the stalk $\O_{X,x}$ is a Noetherian ring.

\end{lemma}

\begin{proof}

Suppose $X=V(I)\subset U\subset \C^k$ for $I$ a finitely generated ideal. Then $\O_{X,x}$ is a quotient of $\O_{\C^n,x}$. Thus it is sufficient to prove $\O_{\C^n,x}$ is Noetherian.

We proceed by induction on $n$. Suppose $0\neq f\in\O_{\C^n,x}$.  As in the proof of Theorem \ref{coherenceofoCn}, using Theorem \ref{thm proper coord} we can change coordinates such that $f$ is a unit times a Weierstrass polynomial $P(w)\in\O_{\C^{n-1},x}[w]$. Thus 
$\O_{\C^n,x}/(f)$ is finite over $\O_{\C^{n-1},x}$ by Lemma \ref{Weierstrass}. As a finite extension of a Noetherian ring is Noetherian, the result follows by induction.
\end{proof}

\begin{lemma}\label{isoco}

For $X$ a basic definable complex analytic space and $x\in X$,  the completions of $\O_{X,x}$ and $\O^{\an}_{X,x}$ are canonically isomorphic.

\end{lemma}

\begin{proof}

For $X$ an open set in $\C^n$ the claim is clear since both completions are canonically the formal power series ring $R_n$ in $n$ variables. By the Artin--Rees lemma, it follows that tensoring with $\O^{\an}_{\C^n,x}$
over $\O_{\C^n,x}$ is exact for finitely generated modules.

Suppose $X=V(I)\subset U$. By the above $I_x^{\an}:=I_p\otimes_{\O_{U,x}}\O^{\an}_{U,x}$ is an ideal of $\O^{\an}_{U,x}$, and we have the isomorphisms \[\O_{X,x}\cong \O_{U,x}/I_x, \textrm{ and }\O^{\an}_{X,x}\cong \O^{\an}_{U,x}/I_x^{\an}\]
It follows that the completions of $\O_{X,x}$ and $\O^{\an}_{X,x}$ are both isomorphic to $R_n/(I_x\otimes_{\O_{U,x}} R_n)$. 
\end{proof}

\begin{cor}\label{excellent}
For $X$ a basic definable complex analytic space and $x\in X$, the stalk $\O_{X,x}$ is an excellent ring.
\end{cor}
\begin{proof}
As $\O_{X,x}$ is a quotient of $\O_{\C^n,x}$, it suffices to take $X=\C^n$ \cite[\href{https://stacks.math.columbia.edu/tag/07QU}{Tag 07QU}]{stacks-project}.  The previous two lemmas then imply that $\O_{\C^n,x}$ is regular, since regularity of Noetherian local rings can be checked on completions.  Using \cite[Theorem 102]{matsumura} and the fact that derivatives of definable holomorphic functions are definable, the claim follows. 
\end{proof}

\begin{proof}[Proof of Theorem \ref{basicanfaithful}]

Sheafification in the analytic topology is exact and tensor products are always right exact, so it is sufficient to prove left-exactness of the tensor product. Suppose that $0\ra E \ra F$ is an exact sequence of definable coherent sheaves. Then we get an injection of stalks $0\ra E_x\ra F_x$  for $x\in X$.
To show that $E^{\an}$ injects into $F^{\an}$ it is sufficient to prove that $E^{\an}_x$ injects into $F^{\an}_x$. Note that $E^{\an}_x \cong E_x\otimes_{\O_{X,x}} \O^{\an}_{X,x}$. 
Since both $\O_{X,x}$ and $\O^{\an}_{X,x}$ are Noetherian local rings by Lemma \ref{noetherian}, the completion is faithfully flat \cite[\href{https://stacks.math.columbia.edu/tag/00MC}{Tag 00MC}]{stacks-project}. Since they have isomorphic completions by Lemma \ref{isoco}, claim (1) follows.

For the second part, we need to show that if we have $E\xrightarrow{f} F$ in $\Coh(X)$ such that $f^\an=0$, then $f=0$.  By considering the image, it is enough to show that if for $F\in \Coh(X)$ we have $F^\an=0$, then $F=0$.  The statement is local, so we may assume $F$ has a presentation
\[\O^m_X\xrightarrow{g}\O_X^n\to F\to0\]
and by part (1) we reduce to the following lemma.
\begin{lemma} If $g^\an$ is surjective then $g$ is.
\end{lemma}
\begin{proof}
We may think of $g$ as an $n\times m$ matrix $M$ consisting of elements of $\O_X(X)$. Since $g^{\an}$ admits a section at each point, at each point some $n\times n$ minor of $M$ is invertible.  Thus on a definable cover given by the nonvanishing of these minors, a section is given by a rational function in the entries of $g$, which is therefore definable.  It follows that $g$ is surjective.
\end{proof}
\end{proof}

\begin{cor}\label{cor basic exact on stalk}For $X$ a basic definable complex analytic space, a sequence $M'\to M\to M''$ of coherent $\O_X$-modules is exact if and only if it is exact on stalks (or even analytic stalks).
\end{cor}
\begin{proof}  By the exactness of $\An$, it suffices to show that if $M^\an=0$ then $M=0$, but this is exactly the faithfulness of $\An$.
\end{proof}

\begin{cor}\label{instalks}\label{cor is zero} Given coherent sheaves $E\subset F$ and a section $s\in F(X)$, then $s\in E(X)$ if and only if $(s^\an)_x\in(E^\an)_x$ for all $x\in X$.
\end{cor}

In view of Remark \ref{rmk: no points} (and Example \ref{eg: no points}), Corollary \ref{cor basic exact on stalk} is quite strong.  In particular, it implies that basic definable complex analytic spaces can be glued as locally $\C$-ringed definable spaces:
\begin{cor} \label{cor sub ringed} Let $X\subset U\subset \C^n$ and $Y\subset V\subset \C^m$ be basic definable complex analytic spaces and $\phi:(X,\O_X)\to(Y,\O_Y)$ a morphism of the associated locally $\C$-ringed definable spaces.  Then there are definable open subsets $U_j\subset U$ covering $X\subset U$ and definable holomorphic maps $g_j:U_j\to V$ with $g_j^\sharp (g_j^{-1}I_Y)\subset I_{X\cap U_j}$ which induce $\phi|_{X\cap U_j}$ in the natural way.
\end{cor}
\begin{proof}  The coordinates give sections $z_1,\ldots,z_m$ of $\O_Y(Y)$ which pull back to functions $w_i=\phi^\sharp z_i\in \O_X(X)$.  By the definition of $\O_X$, there are open definable subset $U_j\subset U$ covering $X$ on which all of the $w_i$ extend to definable holomorphic functions.  Replacing $X\subset U$ with $X\cap U_j\subset U_j$, we may therefore assume the sections $w_i$ lift to the coordinates of a definable holomorphic function $g:U\to \C^m$.  From classical theory we know that $g$ restricts to an analytic morphism $X^\an\to Y^\an$ which induces $\phi^\an$.  Thus, on the level of definable topological spaces $g$ induces $\phi$.  Moreover, from Corollary \ref{instalks} we have that $g^\sharp(g^{-1}I_Y)\subset I_X$, and it remains to show that the induced pullback map $g^\sharp:|\phi|^{-1}\O_Y\to\O_X$ is the same as $\phi^\sharp$.  But on the one hand by Lemmas \ref{noetherian} and \ref{isoco} both pullbacks agree on completions since they agree on the coordinates $z_i$, and therefore they also agree on stalks.  On the other hand, sections are determined by their stalks by Corollary \ref{cor is zero}, so the lemma is proved.  
\end{proof}

\subsection{Definable complex analytic spaces}  Equipped with Corollary \ref{cor sub ringed}, we are in a position to give a concise definition of global spaces locally modeled on basic definable complex analytic spaces.

\begin{defn}We say a locally $\C$-ringed definable spaces $(X,\O_X)$ is \emph{locally a basic definable complex analytic space} if on a definable cover it is isomorphic to the locally $\C$-ringed definable space associated to a basic definable complex analytic space.  We define the category of \emph{definable complex analytic spaces} $(\mbox{DefAnSp}/\C)$ to be the full subcategory of the category of locally $\C$-ringed definable spaces consisting of $(X,\O_X)$ which are locally a basic definable complex analytic space.
\end{defn}

\begin{remark}
We require the underlying definable topological space $X$ to be Hausdorff.  In particular, as $X$ is locally compact (as it is locally a locally closed subset of $\R^n$), it is regular.  
\end{remark}

\begin{remark}\label{rmk def closed immersions}
As in Remark \ref{rmk closed immersions}, we define closed (resp. open) immersions of definable complex analytic spaces to be closed (resp. open) immersions on the level of $\C$-ringed definable spaces.
\end{remark}
The local results of the previous sections immediately globalize; we record them here for convenience.
\begin{thm}\label{coherenceofo}Let $X$ be a definable complex analytic space.  Then $\O_X$ is a coherent $\O_X$-module.
\end{thm}
Denote by $(\mbox{AnSp}/\C)$ the category of complex analytic spaces.  As in the previous section, there is naturally an analytification functor $(-)^\an:(\mbox{DefAnSp}/\C)\to(\mbox{AnSp}/\C)$, as well as analytification functors $(-)^\an:\Coh(X)\to\Coh(X^\an)$ on the level of sheaves for which we have a natural identification $\O_X^\an\cong\O_{X^\an}$.
\begin{thm}\label{anfaithful}Let $X$ be a definable complex analytic space.  Then the analytification functor $(-)^\an:\Coh(X)\to\Coh(X^\an)$ is exact and faithful.
\end{thm}
\begin{cor}\label{cor exact on stalk}For $X$ a definable complex analytic space, a sequence $M'\to M\to M''$ of coherent $\O_X$-modules is exact if and only if it is exact on stalks (or even analytic stalks).
\end{cor}

Finally, as a concrete example and sanity check, we have the following:

\begin{lemma}\label{sectionsaremapstoc}

Let $X$ be a definable complex analytic space. Then elements of $\Gamma(X,\O_X)$ are in natural bijection with morphisms of definable complex analytic spaces $f:X\ra\C$.

\end{lemma}

\begin{proof}

Given a morphism $f:X\ra\C$ we get a map $f^\#:\Gamma(\C,\O_\C)\ra \Gamma(\C,f_*\O_X)$ and we pullback the $\C$-coordinate
$f^\#(z)$ to obtain a global section of $\O_X$. 

We now define the inverse correspondence from sections $s\in\Gamma(X,\O_X)$ to morphisms $s^+:X\to \C$.  It is enough to consider $X$ a basic definable complex analytic space, as the resulting morphisms of definable complex analytic spaces $X\to \C$ glue together.  Thus suppose $X=V(I)$ where $I$ is a finitely generated ideal sheaf in a definable open set $U\subset\C^n$.
Given $s\in\Gamma(X,\O_X)$, after passing to a definable cover $s$ extends to a section $t\in \Gamma(U,\O_U)$, and thus to a morphism of definable complex analytic spaces $t^+:U\ra\C$
which restricts to a morphism $s^+:X\ra \C$.  Note that if we pick a different section lift $t'$ then $t-t'\in \Gamma(U,I)$ and we obtain the same morphism.  To see this, note that it is obvious that $t^+,t'^+$ give the same map on points $|X|\ra\C$. As $t^+,t'^+$ induce the same analytic morphism, it follows from Theorem \ref{anfaithful} that they induce the same map on the sheaf of rings.  One easily check that $s\mapsto s^+$ is inverse to $f\mapsto f^\sharp z$.
\end{proof}

\subsection{Reduced spaces}\label{S reduced spaces}
The goal of this section is to show that any definable complex analytic space $X$ has a canonical reduced subspace $X^\reduced$ which analytifies to the analytic reduced subspace of $X^\an$. 
\begin{defn}For $X$ a definable complex analytic space, we define $\mathcal{N}_X\subset \O_X$ to be the sheaf of ideals given by nilpotent elements of $\O_X$.  We say $X$ is reduced if $\mathcal{N}_X$ is the zero ideal.
\end{defn}
Note that for any $x\in X$, the stalk $\mathcal{N}_{X,x}$ is the ideal of nilpotents of $\O_{X,x}$.

\begin{prop}\label{red coh}
Let $X$ be a definable complex analytic space.  Then $\mathcal{N}_X$ is a coherent sheaf of ideals.
\end{prop}

\begin{proof}
We may assume $X\subset U\subset \C^n$ is a basic definable complex analytic space.  The underlying set $|X|\subset U$ is set-theoretically cut out by generators for the ideal of $X$ in $U$, so $|X|$ is a $\C$-analytic set in the terminology of \cite{PS}.  Let $I\subset \O_U$ be the ideal sheaf of $|X|$; it suffices to prove that $I$ is a coherent sheaf of ideals.  By \cite[Theorem 11.1]{PS}, up to a definable cover there is a finitely generated ideal sheaf $J\subset I\subset\O_U$ which agrees with $I$ on stalks. By Corollary \ref{cor exact on stalk} we have $J=I$.
\end{proof}

\begin{cor}\label{cor reducedstructure}
Let $X$ be a definable complex analytic space.  There is a unique closed definable complex analytic subspace $X^\reduced\subset X$ for which $(X^\reduced)^\an=(X^\an)^\reduced$.  Moreover, $X^\reduced$ is reduced.
\end{cor}
\begin{proof} The uniqueness follows from Theorem \ref{anfaithful}.  For the existence take $X^\reduced=V(\mathcal{N}_X)$, which is clearly reduced.  Recall that an excellent local ring is reduced if and only the completion is \cite[7.8.3(v)]{egaiv}.  From Corollary \ref{excellent}, Lemma \ref{isoco}, and the excellence of analytic local rings we deduce that $\mathcal{N}_{X,x}$ analytifies to the ideal $\mathcal{N}_{X^\an,x}\subset\O_{X^\an,x}$ of nilpotents in the analytic local ring.  Thus, $(X^\reduced)^\an=(X^\an)^\reduced$.  
\end{proof}
We call the subspace $X^\reduced\subset X$ of the corollary the reduced subspace.  For the rest of this section and subsequently, by a closed definable complex analytic subset $\mathcal{Y}\subset X$ of a definable complex analytic space $X$ we mean a subset $\mathcal{Y}\subset X$ on the level of points which is simultaneously a closed analytic subset of $X^\an$ and a definable subset of the definable topological space underlying $X$.

\begin{prop}\label{defanset}
Let $X$ be a definable complex analytic space and $\mathcal{Y}\subset X$ a closed definable complex analytic subset.  Then $\mathcal{Y}$ canonically has the structure of a reduced closed definable complex analytic subspace $Y\subset X$.  
\end{prop}

\begin{proof}  By Corollary \ref{cor reducedstructure}, it suffices to find a closed definable analytic subspace $Y'\subset X$ whose underlying definable topological space is $\mathcal{Y}$.  We may assume $\mathcal{Y}$ is equidimensional by passing to irreducible components.  By passing to definable covers, we may first assume that $X=U\subset \C^n$ is a definable open subset of $\C^n$ and then by Lemma \ref{thm proper coord} that there are linear coordinates $\C^n\cong \C^{n-d}\times\C^d$ for which projection to the second factor $\pi:U\to\C^{d}$ restricts to a finite map $p:\mathcal{Y}\to V$ where $V:=\pi(U)\subset\C^{d}$.  

As $p$ is analytically \'etale over a dense open subset $V_0\subset V$, after possibly passing to connected components (of $V$) we obtain a definable holomorphic map $f_0:V_0\to\Sym^k\C^{n-d}$ mapping $v\mapsto p^{-1}(v)\subset \C^{n-d}$ whose image is contained in the complement $W\subset\Sym^k\C^{n-d}$ of the diagonals.  Note that $\Sym^k\C^{n-d}$ is an affine complex algebraic variety and therefore naturally a definable complex analytic space.  Let $Z\subset \C^{n-d}\times \Sym^k\C^{n-d}$ be the closure of the universal reduced length $k$ subscheme of $\C^{n-d}$ over $W$, which is also naturally an affine complex algebraic variety.  The coordinate functions of $\Sym^k\C^{n-d}$ are clearly locally bounded around $V\setminus V_0$, so $f_0$ extends to a definable holomorphic function $f:V\to \Sym^k\C^{n-d}$, and the base-change of $Z$ along $f$ yields the desired $Y'$. 
\end{proof}

\subsection{Noetherian induction and the Nullstellensatz}\label{subsectnoether}

\begin{prop}[Definable Noetherian induction]\label{induction}

Let $X$ be a definable complex analytic space and $F$ a coherent sheaf on $X$. Any increasing chain of coherent subsheaves of $F$ must stabilize.

\end{prop}

\begin{proof}

It is enough to prove the statement on every open of a definable cover.  As $F$ is locally a quotient of $\O_X^m$, by pulling back our chain we may assume $F=\O_X^m$. The statement for $\O_X^m$ clearly follows from the statement for $\O_X$ so we may assume $F=\O_X$.  We may take $X$ to be a basic definable complex analytic space, and then as $\O_X$ is a quotient of $\O_{\C^n}$ we assume $U\subset \C^n$ is an open definable set.

We now induct on $n$ to show the claim for $\O_U$ for $U\subset \C^n$ open.  Our chain of definable coherent subsheaves corresponds to a chain of ideal sheaves $I_j$.  We may assume after passing to a further cover that all of the $I_j$ contain a function $f\in \O_U(U)$.  As in the proof of Theorem \ref{coherenceofoCn}, we may assume we have a linear projection $\pi:\C^n\to\C^{n-1}$ with $V=\pi(U)$ and that $f=P\in\O(V)[w]$ is a Weierstrass polynomial with zero locus $X=V(P)\subset U$ such that $p=\pi|_X:X\to V$ is finite.  Letting $Q_j=I_j/P\O_X$, the $Q_j$ are coherent sheaves supported on $X$ and it is sufficient to show that the $Q_j$ stabilize.

\begin{lemma}\label{finitewexact}

With the above notation, the pushforward map $p_*$  takes coherent sheaves to coherent sheaves.

\end{lemma}

\begin{proof}

By Lemma \ref{Weierstrass} we know that $p_*\O_X\cong \O_V^{\deg P}$. Let $Q$ be a coherent sheaf. This means that $Q$ has a finite presentation on a definable open cover, and by Proposition \ref{propercover} we may assume $Q$ has a global finite presentation.  By Corollary \ref{finiteexact} this yields a presentation of $p_*Q$.
\end{proof}

By induction, the sequence $p_*Q_j$ stabilizes. The theorem will thus follow if we show that $p_*Q_j=p_*Q_{j+1}$ implies that $Q_j=Q_{j+1}$. By Corollary \ref{finiteexact}
the pushforward $p_*$ is exact, and thus it suffices to show that for a coherent sheaf $Q$, $p_*Q=0$ implies that $Q=0$.  This easily follows from Proposition \ref{propercover}.
\end{proof}

\def\sHom{\mathcal{H}\mathrm{om}\,}

\begin{lemma}\label{lem supp}Let $X$ be a definable complex analytic space and $F,F'$ definable coherent sheaves on $X$.  Then $\sHom_{\O_X}(F,F')$ is a definable coherent sheaf.  Moreover, if $\Supp(F)$ is the subspace cut out by the kernel of the natural map $\O_X\to\sHom_{\O_X}(F,F)$, then:
\begin{enumerate}
\item $\Supp(F)^\an=\Supp(F^\an)$.
\item The underlying definable complex analytic set of $\Supp(F)$ is the set of $x\in X$ for which $F_x\neq 0$.  
\end{enumerate}
\end{lemma}
\begin{proof}For the first claim, we may assume $F$ has a presentation $$\O_X^n\xrightarrow{g}\O_X^m\to F\to0$$ in which case we have an exact sequence
\[0\to \sHom_{\O_X}(F,F')\to \sHom_{\O_X}(\O_X^m,F')\to\sHom_{\O_X}(\O_X^n,F')\]
and therefore $\sHom_{\O_X}(F,F')$ is coherent.  The remaining parts of the lemma follow from Theorem \ref{anfaithful} and the same statements in the analytic category. 
\end{proof}

\begin{cor} \label{setinduction} Let $X$ be a definable complex analytic space.  
\begin{enumerate}
\item
Any decreasing chain of closed definable complex analytic subspaces stabilizes.  
\item
Any decreasing chain of closed definable complex analytic sets stabilizes.
\end{enumerate}
\end{cor}
\begin{proof}  For (1), consider the corresponding chain of ideals.  This also handles (2), by endowing the subsets with the reduced induced structure provided by Proposition \ref{defanset}.  Note that by the lemma a definable complex analytic set $\mathcal{Y}$ may be recovered by the ideal sheaf $I_Y$ defining the subspace $Y$ with the reduced induced structure as the underlying set of $\Supp(\O_X/I_Y)$.
\end{proof}

We therefore deduce a definable Nullstellensatz:
\begin{cor}\label{nullstellensatz} Let $X$ be a definable complex analytic space and $I_{X^\reduced}\subset \O_X$ the ideal sheaf of the reduced subspace $X^\reduced\subset X$.  Then $I_{X^\reduced}^n=0$ for some integer $n>0$.
\end{cor}
\begin{proof} For each $x\in X$ we have $I_{X^\reduced,x}^n=0$ for some $n$, since $\O_{X,x}$ is Noetherian.  By the previous lemma, for any inclusion of definable coherent sheaves $E\subset E'$ on $X$ we have $\Supp(E)\subset \Supp(E')$.  Thus, $\Supp(I_{X^\reduced}^k)$ gives a decreasing chain of definable complex analytic subspaces which must eventually not contain any given point.  Therefore, by Corollary \ref{setinduction} we have that $\Supp(I_{X^\reduced}^k)$ is eventually empty, and thus by the lemma $I_{X^\reduced}^n=0$ for some positive integer $n$.
\end{proof}
\begin{cor}Let $X$ be a definable complex analytic space and $Z\subset X$ a closed definable complex analytic subspace.  Then for some integer $n>0$ we have $I_{Z^\reduced}^n\subset I_Z$.
\end{cor}

 \subsection{Finite push-forward}
\begin{prop}\label{finitepush}
Let $f:X\to Y$ be a finite morphism of definable complex analytic spaces.  Then $f_*$ takes coherent $\O_X$-modules to coherent $\O_Y$-modules and commutes with analytification.
\end{prop}
\begin{proof}
Let $X_0\subset X$ be a closed definable complex analytic subspace with a square-zero ideal.  For any coherent $\O_X$-module $F$ we have a short exact sequence
\[0\to E\to F\to F_0\to0\]
where $F_0$ is the restriction of $F$ to $X_0$ and both $E$ and $F_0$ are coherent $\O_{X_0}$-modules.  If $f_*E$ and $f_*F_0$ are coherent $\O_Y$-modules which analytify to $(f^\an)_*E^\an$ and $(f^\an)_*F_0^\an$, then $f_*F$ is coherent and analytifies to $(f^\an)_*F^\an$ by Corollary \ref{finiteexact} and Theorem \ref{anfaithful}.  Therefore by induction using Corollary \ref{nullstellensatz} we may assume $X$ is reduced.  As $X^\reduced\to Y$ factors through the reduction of $Y$ and the claim is obviously true for closed immersions, we may assume $Y$ is reduced as well.

Likewise, for any sheaf $F$ and any irreducible component $X_0$ of $X$ (with its reduced structure), we have a short exact sequence
\[0\to E\to F\to F_0\to0\]
where $F_0$ is the restriction of $F$ to $X_0$ and $E$ has support a subspace supported on the union of the other irreducible components of $X$.  By induction we may thus assume that $X$ (and therefore also $Y$) is reduced and irreducible.

Using Proposition \ref{propercover}, we may assume $X\subset U\subset\C^m$ and $Y\subset V\subset \C^n$ are both basic definable complex analytic spaces.  By considering the graph, we reduce to the following:
\begin{claim}
Let $Y\subset V\subset \C^n$ be a reduced and irreducible basic definable complex analytic space  
 and $X\subset \C^m\times Y$ a reduced and irreducible closed definable complex analytic subspace.  Assume the second projection $p:X\to Y$ is proper.  Then $p_*$ takes coherent sheaves to coherent sheaves and commutes with analytification.\end{claim}
 \begin{proof}
 We proceed by induction on $m$, the base case being trivial.  Take a linear projection $\C^m\to\C^{m-1}$ and consider the morphism $g:X\to\C^{m-1}\times Y$.  As $g$ is proper, by Remmert's theorem the image of $g^\an$ is a closed complex analytic subvariety and obviously definable, hence by Proposition \ref{defanset} the image canonically has the structure of a reduced and irreducible closed complex analytic subspace $Y'\subset \C^{m-1}\times Y$.  By induction, the claim is true for the push-forward along $Y'\to Y$, so it is enough to show that push-forward along $X\to Y'$ sends coherent sheaves to coherent sheaves and commutes with analytification.  As $X\subset\C\times Y'$ is a closed definable complex analytic subspace which is proper over $Y'$, we are reduced to the following:
 \begin{lemma}
 Let $Y$ be a reduced and irreducible basic definable complex analytic space and $X\subset \C\times Y$ a reduced and irreducible closed complex analytic subspace which maps finitely and surjectively onto $Y$ via the second projection $\pi:X\to Y$.  Then $\pi_*$ takes coherent sheaves to coherent sheaves.  
 \end{lemma}
 \begin{proof}
 Let $w$ be the $\C$-coordinate in $\C\times Y$. Since $X,Y$ are irreducible and $\pi$ is surjective, the number $d$ of pre-images (with multiplicity) is constant, and $w$
is a root of the polynomial $P(t):=\prod_{(s,y)\in X} (t-s)\in \O_Y(Y)[t]$.  

Let $W$ be the analytic subspace cut out by $P$ and $\psi:W\to Y$ the projection. We claim that $\psi_*\O_W$ is free over $\O_Y$. When $Y$ is a domain in $\C^n$, this follows from 
Lemma \ref{Weierstrass} and Proposition \ref{propercover}. In the general case, we have to prove that every function $g$ on $W$ can uniquely be written as a polynomial in $w$ of degree $d-1$ over $\O_Y$.

To show existence, note that we can find a neighborhood $V$ of $Y$ which is open in $\C^n$ such that $P$ extends to $V$ and cuts out a definable complex analytic space $W_V$. Shrinking further
and using Proposition \ref{propercover} we may assume that $g$ extends to $W_V$, and so it can be written as a polynomial in $w$ of degree $d-1$ over $\O_V$. Restricting to $Y$ proves existence.
Uniqueness is true in the analytic category (see e.g. \cite[p.56]{Grauert}) so follows from Theorem \ref{anfaithful}.

As $\psi_*\O_W\cong\O_Y^d$, it follows that $\psi_*$ takes coherent sheave to coherent sheaves (as in Lemma \ref{finitewexact}) and commutes with analytification.  The same is obviously true for push-forward along the closed embedding $X\to W$, an therefore also for the composition $\pi_*$.
 \end{proof}
 \end{proof}
\end{proof}
\begin{cor}\label{finiteimage}
Let $f:X\to Y$ be a finite morphism of definable complex analytic spaces.  Then there is a diagram of definable complex analytic spaces
\[\begin{tikzcd}
 X\ar[rd,swap,"g"]\ar[rr,"f"]&&Y\\
 &Z\ar[ru,swap,"i"]&
\end{tikzcd}\]
where $i$ is a closed immersion, $g$ is surjective on points, and $\O_Z\to g_*\O_X$ is injective.  Moreover, $Z$ analytifies to the analytic image.
\end{cor}
\begin{proof}
The ideal of $Z$ is the kernel of the map $\O_Y\to f_*\O_X$ which is coherent by the proposition.  The remaining statements are clear.
\end{proof}

\subsection{Analytic factorization.}

The purpose of this section is to prove the following factorization statement, which says that if $h:X\to Y$ is a ``scheme"-theoretically surjective morphism of definable complex analytic spaces, then a morphism $g:X\to Z$ factors through $h$ if and only if it factors analytically.  

\begin{prop}\label{prop mapdescent}Let $X,Y,Z$ be definable complex analytic spaces and suppose we have (solid) diagrams
\[\begin{tikzcd}
X\ar{d}{g}\ar{r}{h}&Y\ar[ld,dashed,"f"]		&&		X^\an\ar{d}{g^\an}\ar{r}{h^\an}&Y^\an\ar{ld}{\phi}\\
Z&		&&		Z^\an&
\end{tikzcd}\]
such that $h$ is proper, surjective on points, and $\O_Y\to h_*\O_X$ is injective.  Then a unique $f$ exists such that $f^\an=\phi$.
\end{prop}

In preparation, we need the following lemma:
\begin{lem}\label{descoh}

Let $f:X\ra Y$ be a proper morphism of definable complex analytic spaces that is surjective on points and such that $\O_Y\to f_*\O_X$ is injective.
Let $s\in \Gamma(Y^\an,\O_{Y^\an})$ be such that $(f^\an)^\sharp s\in  \Gamma(X,\O_X)$. Then $s\in \Gamma(Y,\O_Y)$.

\end{lem}

\begin{proof}

The section $(f^\an)^\sharp s$ corresponds to a morphism $g:X\to \C$ by Lemma \ref{sectionsaremapstoc}. The resulting morphism $h=f\times g:X\to Y\times \C$ is proper, so by Proposition \ref{defanset} the reduced analytic image is naturally a definable complex analytic subspace $Z\subset Y\times \C$.  Note that the projection $Z\to Y$ is finite. Let $I_Z$ be the coherent ideal sheaf of $Z$ in $Y\times\C$. The
 pullback  $(h^\sharp I_Z)\O_X$ is a nilpotent coherent sheaf on $X$ and thus some power of it is $0$ by Theorem \ref{nullstellensatz}. Say $(h^\sharp I_Z)^k\O_X=0$.
 Set $Z_k\subset Y\times\C$ to be the definable complex analytic space cut out by $I_Z^k$. Then the map $h$ factors through $Z_k$, and thus the morphism
 $a:Z_k\ra Y$ is surjective on points, with the natural map $\O_Y\to a_*\O_{Z_k}$ being injective. By Proposition \ref{finitepush} we see that $a_*\O_{Z_k}$
 is a coherent sheaf. Let $w$ be the $\C$ coordinate of $Y\times\C$.  Then $w\in\Gamma(Y,a_*\O_{Z_k})$ is the image of $s\in\Gamma(Y^\an,\O_{Y^\an})$, and so the claim follows by Corollary \ref{cor is zero}.
\end{proof}

\begin{proof}[Proof of Proposition \ref{prop mapdescent}]

The uniqueness statement follows immediately from Theorem \ref{anfaithful} so we need only show the existence of $f$. By definable choice $\phi$ is a morphism of definable topological spaces. Let $U\subset Z$ be definable open and 
$s\in\O_Z(U)$. Then by Lemma \ref{descoh} the section $\phi^\sharp s\in \Gamma(\phi^{-1}(U), \O_{Y^{\an}})$ is actually in $\Gamma(\phi^{-1}(U), \O_Y)$. We thus get a morphism $f:Y\ra Z$ and it 
follows from 
Theorem \ref{anfaithful} that $g=f\circ h$.
\end{proof}

\subsection{\'Etale descent}\label{S et descent}
The purpose of this section is to show that quotients by closed \'etale equivalence relations exist in the category of definable complex analytic spaces.

For $X$ a definable complex analytic space, an equivalence relation in the category of definable complex analytic spaces is a diagram $R\rightrightarrows X$ such that for any definable complex analytic space $S$, $\Hom(S,R)\rightrightarrows \Hom(S,X)$ is an equivalence relation\footnote{That is, the resulting map $\Hom(S,R)\to \Hom(S,X)\times\Hom(S,X)$ is the inclusion of a set-theoretic equivalence relation}.  We define the big definable complex analytic site $\underline{\mathrm{DefAnSp}/\C}$ to be the category of definable complex analytic spaces and whose covers are finite covers by open definable subspaces. Given an equivalence relation $R\rightrightarrows X$, we define the sheaf $X/R$ on $\underline{\mathrm{DefAnSp}/\C}$ to be the sheafification of
\[S\mapsto \Hom(S,X)/\Hom(S,R).\]
We say $\pi:X\to Y$ is a quotient of $X$ by $R$ if $Y$ is a definable complex analytic space which represents the sheaf $X/R$.  Concretely, this means that a morphism $S\to Y$ is given by taking a definable cover $S_i $ of $S$ and giving morphisms $S_i\to X$ that agree on overlaps up to the equivalence relation.  A quotient is unique up to unique isomorphism provided it exists.

We say that a morphism $f:X\to Y$ of definable complex analytic spaces is \'etale if it is open and locally an isomorphism onto its image (or equivalently if it is analytically \'etale, by Theorem \ref{anfaithful}).  We say an equivalence relation $R\rightrightarrows X$ is \'etale if the two maps are \'etale and closed if $R\to X\times X$ is a closed immersion.    

\begin{prop}\label{trivcover}
Let $U$ be a definable complex analytic space and $R\rightrightarrows U$ a closed \'etale definable equivalence relation.  Then there exist finitely many definable open sets $U_i$ of $U$ such that $R\cap (U_i\times U_i) = \Delta_{U_i}$, and such that $\bigcup U_i$ surjects on the set-theoretic quotient $U/R$.

\end{prop}

\begin{proof}

\item\vskip1em\noindent\emph{Step 1.}
By definable choice\cite[Chapter 6 \S 1.2]{Vdd}, we can find a definable subset $T$ of $U$ which has exactly one point for each $R$-representative class.  Let us stratify $T$ by submanifolds $T_i$\cite[Theorem 6.1]{PS}. For each $i$ let 
$S_i$ be the set of all points equivalent to $T_i$ but not actually in $T_i$. It is easy to see that $S_i$ is also a submanifold. Now we will show how to further stratify such that $T_i$ is disjoint from $\bar{S_i}$.
To do this, note that $T_i\cap \bar{S_i}$ is of smaller dimension than $T_i$. Thus by successively iterating in this way we can obtain our desired stratification. By further stratifying, we can assume that the number of 
$R$-pre-images along $T_i$ is constant, and that each $T_i$ is a cell and is therefore simply-connected.

\item\vskip1em\noindent\emph{Step 2.}  By the argument in Proposition \ref{propercover} we may take $V_i$ to be a definable open neighbourhood of $T_i$ such that $R\cap (V_i\times U)$ consists of $k$ \'etale sections over $V_i$---which we denote by $R_0$---and another piece $R'$ which does not 
intersect $T_i\times U$. 

\item\vskip1em\noindent\emph{Step 3.}  Pick a definable distance function $d(x,y)$ on $U\times U$, and pick a definable exhaustion function $E:U\ra\R_{\geq 0}$. In other words, $E^{-1}([0,c])$ is compact for all $c\in\R$. For a set $S\subset U$ we write
$S^{c}$ to mean $S\cap E^{-1}([0,c])$.

\item\vskip1em\noindent\emph{Step 4.}  By definable choice we may let $h:\R^2_{\geq 0}\ra (0, 1)$ be a definable, positive function such that: for all $(c,c')$, if we set set $\epsilon=h(c,c')$ then $R'\cap B_{d,\epsilon}(T^c_i)\times B_{d,\epsilon}(T^{c'}_i)=\emptyset$.
Consider the function \[g(c):=\min_{c_1,c_2<c}\frac{h(c_1,c_2)}{2}.\] We let $f(c)$ be a definable positive, continuous, decreasing function strictly smaller than $g(c)$. Note that 
$h(c,c')>\min(f(c),f(c')).$

\item\vskip1em\noindent\emph{Step 5.}  Define $d'(u,T_i):=\min_{c,t\in T^c_i} d(u,t) f(c)^{-1}$.  Define $W_i$ to consist of all points $u\in V_i$ such that $d'(u,T_i)<\min_{u'\in R(u)\setminus u}d'(u',T_i).$ We claim that $W_i$ contains an open neighbourhood around $T_i$. 
Let $t\in T_i$. For $\epsilon>0$, consider the ball $B_{d',\epsilon}(t)$. It is clear that for sufficiently small $\epsilon$, 
$d'$ is smaller on this ball then on $R_0$, and $d'$ is smaller then $1/2$. Suppose that $u\in B_{d',\epsilon}(t), u'\in R'$ and  $d(u',t')\leq f(c)$ for $t'\in T^c_i$. It follows that the point
$(u,u')\in R'\cap  B_{d,\epsilon}(T^c_i)\times B_{d,\epsilon}(T^{c'}_i)$ for $\epsilon = \min(f(c),f(c')) < h(c,c')$.  This is a contradiction.
Setting $U_i\subset W_i$ to be the maximal open subset (which is a definable condition), the proof is completed.
\end{proof}

\begin{cor}\label{quotient}Quotients by closed \'etale equivalence relations exist in the category of definable complex analytic spaces.
\end{cor}
\begin{proof}The quotient can be glued together from the cover provided from Proposition \ref{trivcover}.
\end{proof}

\begin{cor}Let $X,Y$ be definable complex analytic spaces and $f:X\to Y$ an \'etale morphism.  Then there is a definable open cover $X_i$ of $X$ such that the restrictions $f_j:X_j\to Y$ are open immersions.
\end{cor}
\begin{proof}Apply the proposition to the equivalence relation $X\times_Y X\subset X\times X$.
\end{proof}

As an application we end this subsection with a definable version of Riemann existence.
\begin{lemma}\label{riemann exist}
Let $Y$ be a definable complex analytic space and $\phi:\mathcal{X}\to Y^\an$ a finite \'etale morphism.  Then $\phi$ is the analytification of a finite \'etale morphism $f:X\to Y$ which is uniquely determined by the property that for any morphism $g:Z\to Y$, any analytic lift of $g$ to $X$ is definable.
\end{lemma}

\begin{proof}  First, the lifting property is easily seen to uniquely specify the definable structure on $\mathcal{X}$.  For the existence, we may assume $Y$ (and $\mathcal{X}$) are reduced, as the structure sheaf on $X$ is uniquely determined as the pull-back from $Y$.  Let $U_i\to Y$ be a definable simply-connected cover (see Remark \ref{simply}), and denote by $(U_i^j)^\an\to \mathcal{X}$ the possible lifts of $U_i^\an\to Y^\an$ to $\mathcal{X}$.  Let $U=\sqcup_{i,j}U_i^j$ and $U^\an\to\mathcal{X}$ the obvious morphism.  We claim the closed analytic subvariety $U^\an\times_\mathcal{X}U^\an\subset U^\an\times U^\an$ has a natural definable structure $R\subset U\times U$.  Indeed, $U^\an\times_\mathcal{X}U^\an\to U^\an\times U^\an$ is a disjoint union of $(U_i^j)^\an\times_\mathcal{X}(U_{i'}^{j'})^\an\to (U_i^j)^\an\times (U_{i'}^{j'})^\an$, which is a union of components of $U_i^\an\times_{Y^\an}U_{i'}^\an\to U_i^\an\times U_{i'}^\an$ and therefore can be given a definable structure by taking the union of the same components of $U_i\times_Y\times U_{i'}\to U_i\times U_{i'}$.  The quotient $X$ of the equivalence relation $R\subset U\times U$ has a finite \'etale morphism $f:X\to Y$, definabilizes to $\phi$, and clearly has the lifting property.
\end{proof}

\subsection{Definabilization}
\def\Zariski{\mathrm{Zar}}
\def\fZariski{\mathrm{fZar}}
\def\Sch{\mathrm{Sch}}
\renewcommand{\Define}{(-)^\df}  
Recall that throughout by scheme (resp. algebraic space) we mean a finite type separated scheme over $\C$ (resp. a finite type separated algebraic space over $\C$).  If $X$ is an affine scheme presented as $\spec\C[x_1,\dots,x_n]/I$ we define the definabilization $X^{\df}$ to be the definable complex analytic subspace of $\C^n$ given by the coherent ideal sheaf $I\O_{\C^n}$.  Note that the category of sheaves on the Zariski site $X^\Zariski$ of $X$ is naturally equivalent to the category of sheaves on the site $X^{\fZariski}$ only allowing finite Zariski covers, as an arbitrary cover is refined by a finite one.  It is then easy to see we obtain a functor from affine schemes to definable complex analytic spaces which is functorial and maps finite open covers to open covers, and thereby extends uniquely to a functor from schemes to definable complex analytic spaces $\Define:(\mbox{Sch}/\C)\to(\mbox{DefAnSp}/\C)$.

For $X$ a scheme, the definabilization functor yields a morphism of locally $\C$-ringed definable spaces
\[g:(X^\df,\O_{X^\df})\to (X^{\fZariski},\O_X)\]
as there is a natural map $g^{-1}\O_X\to\O_{X^\df}$.  Let $\Coh(X)$ be the category of coherent sheaves on $X$, and $\Coh(X^\df)$ the category of definable coherent sheaves on $X^\df$.  We then define a definabilization functor 
\[\Define :\Coh(X)\to\Coh(X^\df):F\mapsto F^\df:=\O_{X^\df}\otimes_{g^{-1}\O_X}g^{-1}F.\]
Evidently there is a natural isomorphism $(\O_{X})^\df\cong \O_{X^\df}$.  

\def\Aff{\mathrm{Aff}}
\def\etale{\mathrm{\acute{e}t}}
\def\fetale{\mathrm{f\acute{e}t}}
\def\AlgSp{\mathrm{AlgSp}}
\def\DefAnSp{\mathrm{DefAnSp}}
\def\Desc{\mathrm{Desc}}
We now extend this picture to algebraic spaces.  This level of generality is necessary for \S\ref{sectiondefimage} as Artin's algebraization theorem does not hold true for schemes.  

\begin{prop}\label{prop ext alg sp}There is a unique extension
\[\Define:(\AlgSp/\C)\to(\DefAnSp/\C)\]
of the definabilization functor on affine schemes to algebraic spaces.  Moreover, for each algebraic space $X$ there is a definabilization functor on sheaves 
\[\Define:\Coh(X)\to\Coh(X^\df)\]
which is compatible with the definabilization of sheaves on affine schemes for any affine (\'etale) open of $X$.
\end{prop}
\begin{proof}
Suppose that $X$ is an algebraic space, and that we have a presentation $R\rightrightarrows U\to X$ as the quotient of $U$ by a closed \'etale equivalence relation $R\rightrightarrows U$ where $R,U$ are schemes.  Recall that we think of $X$ as a sheaf on the big \'etale site $(\Sch/\C)^\etale$ of schemes, and that for $X$ to be presented by $R\rightrightarrows U$ means that we have a morphism $\pi:U\to X$ of sheaves on $(\Sch/\C)^\etale$ identifying $X$ as the quotient $U/R$ (see \S\ref{S et descent}).  

We obtain a definable closed \'etale equivalence relation $R^\df\rightrightarrows U^\df$ (using Theorem \ref{anfaithful}), and by Corollary \ref{quotient} we can define the definabilization $X^\df$ of $X$ to be the quotient.  We claim that this is independent of the presentation and yields a functorial extension $\Define:(\AlgSp/\C)\to(\DefAnSp/\C)$.  It suffices to show that for any algebraic spaces $X,Y$ with presentations $R\rightrightarrows U$ and $T\rightrightarrows V$ and any morphism $f:X\to Y$ we obtain a morphism $f^\df:X^\df\to Y^\df$ and that the formation of $f^\df$ is compatible with compositions.  Both claims are clear from the universal property satisfied by the quotient.

We finally show the existence of the definabilization functor on sheaves $\Define:\Coh(X)\to \Coh(X^\df)$.  For this, given a presentation $R\rightrightarrows U\to X$, the category $\Coh(X)$ is naturally equivalent via pullback to the category $\Desc(R\rightrightarrows U)$ of descent data \cite[\href{https://stacks.math.columbia.edu/tag/03M3}{Tag 03M3}]{stacks-project}:  pairs $(F,\phi)$ where $F\in\Coh(U)$ and $\phi:\pi_1^*F\to \pi_2^*F$ is an isomorphism such that on $R\times_U R$ we have $\pi^*_{13}\phi=\pi^*_{23}\phi\circ\pi^*_{12}\phi$, where $\pi_i:U\times U\to U$ and $\pi_{ij}:R\times_U R\to U\times U$ are the natural projections.  Corollary \ref{quotient} likewise shows that for a quotient $\mathcal{R}\rightrightarrows \mathcal{U}\to\mathcal{X}$ of a definable complex analytic space by a closed \'etale equivalence relation, the natural functor $\Coh(\mathcal{X})\to \Desc(\mathcal{R}\rightrightarrows \mathcal{U})$ is an equivalence.  With these identifications we then define $\Define:\Coh(X)\to \Coh(X^\df)$ as $\Define: \Desc(R\rightrightarrows U)\to \Desc(R^\df\rightrightarrows U^\df)$ by $(F,\phi)\mapsto (F^\df,\phi^\df)$, and this is easily seen as above to be independent of the choice of presentation and compatible with restrictions to open subspaces.
\end{proof}

\subsection{Quotients by finite groups}
The purpose of this section is to show that quotients by finite group actions exist in the category of definable complex analytic spaces.  This will be used to endow $\Gamma\backslash\Omega$ with a definable structure when $\Gamma$ is not torsion-free.

\begin{defn}
Let $X$ be a definable complex analytic space with a left action (in the category of definable complex analytic spaces) by a finite group $G$.  A geometric quotient of $X$ by $G$ is definable complex analytic space $Y$ and a morphism
$q:X\ra Y$ which on the level of topological spaces is the quotient map to the set of orbits (with the quotient topology) and such that $\O_Y=(q_*\O_X)^G$.  A geometric quotient is a categorical quotient and is therefore unique up to unique isomorphism when it exists, in which case we will usually denote it $q:X\to G\backslash X$. 
\end{defn}

\begin{prop}\label{finitequotient}

Let $X$ be a definable complex analytic space with a left action by a finite group $G$. Then the geometric quotient $q:X\to G\backslash X$ of $X$ by $G$ exists. Moreover, $q$ analytifies to 
the analytic geometric quotient.

\end{prop}
\begin{proof}We first prove a special case. Recall that for affine varieties, geometric quotients by finite groups exist.

\begin{lemma}\label{algquot}
Let $V$ be an affine complex algebraic variety with a left action by a finite group $G$.  Then the definabilization of the algebraic geometric quotient $q:V\to G\backslash V$ is the definable geometric quotient.
\end{lemma}

\begin{proof}

By \cite[Remark 1.6]{Popp} the analytification of $q:V\to G\backslash V$ is an analytic geometric quotient.  Thus, $q^\df:V^{\df}\ra (G\backslash V)^{\df}$ is the definable quotient on the level of topological spaces.  It remains to show that $\O_{(G\backslash V)^\df}=(q^\df_*\O_{V^\df})^G$.  Observe that $q$ is finite so $q^\df_*\O_{V^\df}$ is a coherent $\O_{(G\backslash V)^\df}$-module by Proposition \ref{finitepush}.  It follows that $(q^\df_*\O_{V^\df})^G$ is a coherent $\O_{(G\backslash V)^\df}$-module, since it is the intersection of the kernels of the maps $g^\sharp-1:q^\df_*\O_{V^\df}\to q^\df_*\O_{V^\df}$ for all $g\in G$.  Thus by Theorem \ref{anfaithful}, the image of the natural morphism $\O_{(G\backslash V)^\df}\to q^\df_*\O_{V^\df}$ is $(q^\df_*\O_{V^\df})^G$ since the analytifications agree.
\end{proof}

\begin{lemma}\label{closedquot}
Let $X$ be a definable complex analytic space with a left action by a finite group $G$ and $q:X\to G\backslash X$ the geometric quotient.  Let $Z\subset X$ be a closed $G$-invariant definable complex analytic subspace.  Then $q:Z\to q(Z)$ is a geometric quotient of $Z$ by $G$.
\end{lemma}

\begin{proof}Note that $q$ is finite.  According to Corollary \ref{finiteimage}, the image $q(Z)$ exists as a closed definable complex analytic subspace of $G\backslash X$.  By Proposition \ref{prop mapdescent} it is enough to show that $q^\an:Z^\an\to q(Z)^\an$ is the analytic quotient.  On the underlying topological spaces this is clear, and the image of $(q^\an_*\O_{X^\an})^G$ in $q^\an_*\O_{Z^\an}$ is clearly $(q^\an_*\O_{Z^\an})^G$.

\end{proof}

We now prove the proposition. Suppose first that $X$ is a basic definable analytic space, given by definable
ideal sheaf $I\subset \O_{\C^n}(U)$, where $U\subset\C^n$ is open and definable. Denote by $U^G$ the $|G|$-fold cartesian product of $U$, indexed by the elements of $G$.  There is a natural map $i:X\hookrightarrow U^G$ given by $i(x):= (g^{-1}x)_g$, which is $G$-equivariant. 

Now $U^G$ is open inside $(\C^n)^G$ which is naturally an algebraic variety. Thus the geometric quotient $G\backslash(\C^n)^G$ exists by Lemma \ref{algquot}. It trivially follows that $G\backslash U^G$ also exists as it is just a definable open subspace. Thus, by Lemma \ref{closedquot} the quotient $G\backslash X$ also exists. 

Finally, we handle the case of general $X$. Note first that the quotient $q:X\ra G\backslash X$ exists in the category of definable topological spaces by \cite[Cor 10.2.18]{Vdd}. Pick a covering of $X$ by basic definable open subspaces $U_i$. By Lemma \ref{propercover} we may thus pick a covering of $G\backslash X$ such that the inverse image of every open is a disjoint union of opens subsets each of which is contained in
some $U_i$, and hence themselves basic definable open spaces. Since the geometric quotient is local on $G\backslash X$ for the definable site, the proof is complete.

\end{proof}

\subsection{Previous related work}

For a real analytic manifold $M$, Kashiwara--Schapira \cite[\S 7]{KS} have introduced the \emph{subanalytic site} $M_{\mathrm{sa}}$ of $M$. The objects of $M_{\mathrm{sa}}$ consist of subanalytic open subsets of $M$ whose coverings satisfy a local finiteness condition:  for any subanalytic open set $U\subset M$, any covering $U_i$ of $U$ in $M_{\mathrm{sa}}$, and any relatively compact $K\subset M$, the cover $K\cap U_i$ of $K\cap U$ has a finite refinement.

It is natural to compare this site to our notion of a definable topological space and its associated definable site when working with the subanalytic o-minimal structure $\R_{\mathrm{an}}$ (see e.g. \cite{VddM}).  On the one hand, a compact real analytic manifold $M$ admits a unique structure of an $\R_\an$-definable topological space. Moreover, all coverings refine to be finite coverings, and so in this case the Kashiwara--Schapira site and the $\R_\an$-definable site give equivalent categories of sheaves. 

On the other hand, for non-compact $M$ the two sites end up being different in a few important ways:

\begin{enumerate}
\item If $M$ is non-compact, then $M$ does not have a canonical structure as an $\R_{\mathrm{an}}$-definable topological space. This is because the classical notion of subanalyticity of a subset $Z\subset \R^n$ (see e.g. \cite{BM}) is a \emph{local} condition, which does not see the behavior `at infinity'. By contrast,  $\R_{\mathrm{an}}$-definability of a subset $Z\subset\R^n$ is a stronger condition, which roughly says that $Z$ is \emph{globally} subanalytic up to a finite cover.
\item If $M$ is non-compact, even if one equips $M$ with the structure of an $\R_{\mathrm{an}}$-definable space, the objects of $M_{\mathrm{sa}}$ and the definable site are different. Indeed, an open set $U\subset\R^n$ is definable in $\R_{\mathrm{an}}$ iff its closure $\ol{U}\subset\PP^n(\R)$ in real projective space is subanalytic. 
\item The definable site only allows finite coverings, so the sheaf axiom is much \emph{less} restrictive. As a notable example, taking $M$ to be affine space the definable site does not allow the covering by open balls of radius $\epsilon$, whereas $M_{\mathrm{sa}}$ does.
\end{enumerate}

As demonstrated more clearly by (3) above, the Kashiwara-Schapira site does not restrict behaviour at infinity. As such, it is inadequate for our purposes as one of our main motivations is to provide non-trivial global restrictions on holomorphic functions beyond what one sees locally.

 \section{Definable GAGA}\label{sectiondefGAGA}  In this section we prove an algebraization theorem for definable coherent sheaves on algebraic spaces.  Precisely, we show:
 
 \begin{thm}\label{gaga} Let $X$ be an algebraic space and $\Define:\Coh(X)\to\Coh(X^\define)$ the definabilization functor.  Then
\begin{enumerate}
\item $\Define$ is fully faithful and exact.
\item The essential image of $\Define$ is closed under taking subobjects and quotients.	
\end{enumerate}
\end{thm}
 
 \begin{eg}\label{countereg}
$\Define$ is \emph{not} essentially surjective.  Let $X=\G_m$ and let $\alpha\in\C$.  Note that the rank one $\C$-local system $V$ on $X^\an$ with monodromy $\lambda=e^{2\pi i \alpha}$ can be trivialized on a definable open cover---take for instance a finite union of overlapping angular sectors.  It follows that $\mathcal{F}=V\otimes_{\C_{X^\df}} \O_{X^\df}$ is a definable coherent sheaf.  Note that the only algebraic line bundle on $X$ is the trivial bundle $\O_{X}$.  

We claim that $\mathcal{F}$ can be nontrivial as a definable coherent sheaf; in fact, if $\alpha\notin\R$, $\mathcal{F}$ will not be trivial in any o-minimal structure.  A trivializing section is of the form $v\otimes f$ for a nowhere zero multivalued holomorphic function $f$ on $\C^*$ with monodromy $\lambda$.  Taking $q$ to be the standard coordinate on $\G_m$, after multiplying by some power $q^n$ we may assume $f=e^{\alpha\log q+g(q)}$ for a holomorphic function $g:\C^*\to\C$.  As $f'/f=\alpha q^{-1}+g'(q)$ is single-valued and definable, it cannot have essential singularities at $0$ or $\infty$ (or else it would have infinite fibers), and therefore $g$ is algebraic---in particular, a polynomial in $q,q^{-1}$.  But restricting to positive real $q$, we have that
\[\{q\in\R_{>0}\mid f(q)\in\R\}=\{q\in\R_{>0}\mid (\Imag g)(q)+(\Imag \alpha)\log q\in \pi\Z\}\]
is definable, which is only the case if $\Imag g$ is constant and $\Imag \alpha=0$.
\end{eg}

We will make extensive use of the following version of the definable Chow theorem of Peterzil--Starchenko:  
\begin{thm}[Peterzil--Starchenko {\cite[Corollary 4.5]{PSdefchow}}]\label{defchow}  Let $Y$ be a reduced algebraic space and $\mathcal{X}\subset Y^\df$ a closed definable complex analytic subset.  Then $\mathcal{X}$ is algebraic.
\end{thm}
 \begin{proof}
 The statement in \cite[Corollary 4.5]{PSdefchow} is for $Y$ affine; see also the version in \cite[Theorem 2.2]{MPT} for $Y$ a variety.  We may deduce the same statement for algebraic spaces using an \'etale cover or the fact that every algebraic space has a dense open subspace which is a scheme \cite[\href{https://stacks.math.columbia.edu/tag/06NH}{Tag 06NH}]{stacks-project}. 
 \end{proof}
 
 Before the proof we make some preliminary observations.  
 \begin{lemma}\label{defineexact}
 $\Define$ is faithful and exact.
 \end{lemma}
 
 \begin{proof}
 
By Lemma \ref{anfaithful} the map $\An:\Coh(X^\define)\to\Coh(X^\an)$ is faithful and exact. By \cite[Prop. 10 a,b]{Serre}, the usual analytification functor $\An\circ\Define$ is faithful and exact. It follows that $\Define$ is also faithful
 and exact. 
 \end{proof}
  
  Observe that a homomorphism $F_1\to F_2$ of coherent sheaves can be recovered from its graph as a subsheaf of $F_1\oplus F_2$.  It follows that part (2) of Theorem \ref{gaga} implies part (1) using Lemma \ref{defineexact}.  Moreover, the first part of (2) clearly implies the second part by considering the kernel and using the exactness part of Lemma \ref{defineexact}.  We therefore have:
 
\begin{lemma}\label{lemma it is enough}
Let $X$ be an algebraic space.  Then Theorem \ref{gaga} holds for $X$ if and only if for every algebraic coherent sheaf $F$ on $X$, any definable coherent subsheaf $\mathcal{E}\subset F^\df$ is the definabilization of an algebraic coherent subsheaf $E\subset F$.
\end{lemma}
 
The two preceding observations together imply that Theorem \ref{gaga} holds on $X$ if and only if it holds on the reduction $X^\reduced$:

\begin{lemma}\label{lemma nilp gaga}Let $X$ be an algebraic space with a nilpotent sheaf of ideals $I$ cutting out a subspace $X_0$.  Then Theorem \ref{gaga} holds for $X_0$ if and only if it holds for $X$.
\end{lemma}
\begin{proof}Note that for any definable complex analytic space $\mathcal{Y}$ and closed definable complex analytic subspace $\mathcal{X}\subset\mathcal{Y}$ cut out by an ideal $\mathcal{I}$, $\Coh(\mathcal{X})$ is naturally identified via push-forward with the full subcategory of sheaves in $\Coh(\mathcal{Y})$ annihilated by $\mathcal{I}$. The if direction is therefore obvious. Let us prove the converse direction. By induction on the order of nilpotence of $I$, we may assume $I$ is square-zero.

Let $F$ be a coherent sheaf on $X$ and $\mathcal{E}\subset F^\df$ a definable coherent subsheaf.  By Lemma \ref{lemma it is enough} it will be enough to show that $\mathcal{E}$ is algebraic.  Using Lemma \ref{defineexact} we have the diagram
\[\xymatrix{
0\ar[r]&(IF)^\df\ar[r]&F^\df\ar[r]&(F/IF)^\df\ar[r]&0\\
0\ar[r]&I^\df \mathcal{E}\ar[r]\ar[u]&\mathcal{E}\ar[r]\ar[u]&\mathcal{E}/I^\df\mathcal{E}\ar[r]\ar[u]&0\\
}\]

Note that the ideal $I$ being square-zero, both $(IF)^\df$ and $I^\df \mathcal{E}$ are coherent $\O_{X_0^\df}$-modules. Since Theorem \ref{gaga} holds for $X_0$, we have that $I^{\define}\mathcal{E}= M^{\define}$ for a coherent $M\subset IF$. But $\mathcal{E}$ is equal to the preimage of its image by the map $F^\df \rightarrow (F/M)^{\df}$; we may thus replace $F$ by $F/M$, and reduce to the case $I^{\define}\mathcal{E}=0$. Likewise, $\mathcal{E}$ maps to $(F/IF)^{\define}$ and must have algebraic image $N^{\define}$ for a coherent $N\subset F/IF$. Replacing $F$ by the inverse image of $N$, we may assume that $\mathcal{E}$ maps isomorphically
to $(F/IF)^{\define}$. Thus we are reduced to showing that if $F\to(F/IF)$ has a definable section then it is algebraic. Note that this section would have to land in $P^\df$, where $P\subset F$ is the subsheaf annihilated by $I$.  Since both $F/IF$ and $P$ are both coherent sheaves on $X_0$, this
follows from Theorem \ref{gaga} for $X_0$.
\end{proof}

\begin{proof}[Proof of Theorem \ref{gaga}]
We proceed by Noetherian induction on $X$, assuming Theorem \ref{gaga} holds for every proper subspace of $X$.  By Lemma \ref{lemma nilp gaga} we may assume $X$ is reduced.  Let $F$ be an algebraic coherent sheaf on $X$ and $\mathcal{E}\subset F^\df$ a definable coherent subsheaf.  By Lemma \ref{lemma it is enough} we must show that $\mathcal{E}$ is the definabilization of an algebraic coherent subsheaf $E\subset F$.

\vskip1em\noindent\emph{Step 1.}

\begin{lemma}\label{vectgaga}Any exact sequence
\[0 \ra \mathcal{E}\ra F^\df \ra \mathcal{G} \ra 0\] in $\Coh(X^\define)$ for which $\mathcal{E}$ and $\mathcal{G}$ are locally free is the definabilization of an exact sequence \[0 \ra E\ra F \ra G \ra 0\] in $\Coh(X)$ where $E$ and $G$ are locally free.
\end{lemma}
\begin{proof} Observe that $F^\df$ (and hence $F$) is locally free.  It is sufficient to construct the quotient $G$ and then define $E$ as the kernel of $F \ra G  \ra 0$. By working separately on every connected component of $X$, one can assume that $\mathcal{G}$ has constant rank $r$. Let $\mathrm{Gr}(r,F)$ be the Grassmannian of quotient modules of $F$ that are locally free of rank $r$. Then $\mathcal{G}$ corresponds to a definable section of $\mathrm{Gr}(r,F)^\df$, which is necessarily algebraic by Theorem \ref{defchow}, as $X$ is reduced.
\end{proof}

\vskip1em\noindent\emph{Step 2.}

\begin{lemma}  For some dense open $U\subset X$, $\mathcal{E}_{|U}$ is algebraic.	
\end{lemma}
\begin{proof}
On some dense open set $U$, $F$ is locally free since $X$ is reduced. The (reduced) locus where $\mathcal{E}$ and $F^\df / \mathcal{E}$ have non-maximal rank is definable, analytic, and closed,  hence algebraic by Theorem \ref{defchow}.   After possibly shrinking $U$ to a smaller dense open set, the claim then follows from the previous step. 
\end{proof}

\vskip1em\noindent\emph{Step 3.}\vskip1em

With the notation of the previous step, let $E_{U}$ be the algebraic sheaf on $U$ for which $(E_{U})^\define\cong \mathcal{E}_{|U}$.  Let $\widetilde E$ be the ``closure" of $E_{U}$ in $F$, i.e. the pullback
\[\xymatrix{
F\ar[r]&j_*j^*F\\
\widetilde E\ar[u]\ar[r]&j_*E_{U}\ar[u]
}\]
where $j: U \hookrightarrow X$ denotes the inclusion. The sheaf $\widetilde E$ is evidently quasi-coherent and so it is coherent since it is a subsheaf of $F$.  Thus, $\widetilde{E}^\df$ and $\mathcal{E}$ are both definable coherent subsheaves of $F^\df$, and therefore
so is their intersection $\mathcal{G}$.

Let $I_Z$ be the ideal sheaf of $Z=X\smallsetminus U$ with the reduced algebraic space structure, and $\mathcal{I}=I_Z^\df$.
\begin{lem}  Suppose we have definable coherent sheaves $\mathcal{G}\subset\mathcal{G}'$ for which $\mathcal{G}_{|U}=\mathcal{G}'_{|U}$.  Then for some positive integer $n$, $\mathcal{I}^n\mathcal{G}'\subset \mathcal{G}$.
\end{lem}
\begin{proof}  Take the quotient
\[0\to \mathcal{G}'\to \mathcal{G}\to\mathcal{Q}\to0.\]
By Lemma \ref{nullstellensatz}, $\mathcal{I}^n$ kills $\mathcal{Q}$ for some positive integer $n$, and thus $\mathcal{I}^n\mathcal{G}\subset \mathcal{G}'$.
\end{proof}

Applying the lemma to $\mathcal{G}\subset \widetilde{E}^\df$, we have $(I_{Z}^n\widetilde{E})^\df\subset \mathcal{E}$ for some positive integer $n$.  The quotient $\mathcal{E}'$ is then a subsheaf of $(F')^\df$, where $F'=F/I_{Z}^n\widetilde{E}$ is supported on a subspace whose reduction is $Z$.  By the inductive hypothesis, $\mathcal{E}'$ is algebraic, and $\mathcal{E}$ is the preimage in $F$, hence algebraic, so the proof is complete. 
\end{proof}

As an immediate corollary to Theorem \ref{gaga}, we obtain a version of the definable Chow theorem for arbitrary (nonreduced) algebraic spaces.
\begin{cor}\label{defchowsub}Let $Y$ be an algebraic space and $\mathcal{X}\subset Y^\df$ a closed definable complex analytic subspace.  Then $\mathcal{X}$ is (uniquely) the definabilization of an algebraic subspace.
\end{cor}
\begin{proof}We need only algebraize the quotient $\O_Y^\df\to\O_{\mathcal{X}}$, which follows from Theorem \ref{gaga}.
\end{proof}
\begin{cor}\label{defchowmap}Let $X,Y$ be algebraic spaces.  Then any morphism $X^\df\to Y^\df$ of definable complex analytic spaces is (uniquely) the definabilization of an algebraic morphism.
\end{cor}
\begin{proof}Apply the previous corollary to the graph.
\end{proof}

 \section{Definable images}\label{sectiondefimage}
 
 The purpose of this section is to prove an algebraization theorem for definable images of algebraic spaces.  
 
 \subsection{Main statement}
 
 For convenience we make the following definition.
\begin{defn}A morphism $f:X\to Y$ of algebraic spaces is dominant if $\O_Y\to f_*\O_X$ is injective.
\end{defn} 
 Note that a proper dominant morphism is surjective on complex points.  Our goal is to prove the following result.
 
 \begin{thm}\label{defpropmap}

Let $X$ be an algebraic space, $\mathcal{S}$ a definable complex analytic space, and $\phi: X^{\df}\ra \mathcal{S}$ a proper definable complex analytic morphism.  Then there exists a (unique) factorization
\[\begin{tikzcd}
X^\df\arrow{rr}{\phi}\arrow{dr}[swap]{f^\df}&&\mathcal{S}.\\
&Y^\df\arrow[ru,hook,swap,"\iota"]&
\end{tikzcd}\]
where $f:X\to Y$ is dominant algebraic and $\iota$ is a definable closed immersion.  Moreover, $\iota^{\an}(Y^{\an})$ coincides with the image $\phi^{\an}(X^{\an})$.

\end{thm}	
\begin{remark}\label{unicity}
The uniqueness property in Theorem \ref{defpropmap} is the following:  for any other factorization $\phi=\iota'\circ f'^\df$ with $f':X\to Y'$ dominant and $\iota':Y'^\df\to \mathcal{S}$ a definable closed immersion there is a unique isomorphism $g:Y\to Y'$ for which the following diagram commutes.
\[
\begin{tikzcd}
X^\df\arrow[dddr,swap,"f'^\df"]\arrow{rr}{\phi}\arrow{dr}{f^\df}&&\mathcal{S}.\\
&Y^\df\arrow[dd,"g^\df"]\arrow[ru,hook,"\iota"]&\\
&&\\
&Y'^\df\arrow[ruuu,hook,swap,"\iota'"]&
\end{tikzcd}
\]
\end{remark}
\begin{remark}

We expect the theorem to hold without the properness assumption on $\phi$.

\end{remark}

\subsection{First reductions in the proof of Theorem \ref{defpropmap}}\label{sectrestrict}

The proof of Theorem \ref{defpropmap} involves two main steps:  the case that $X$ is reduced and the reduction to this case.  In this subsection we give the proof of Theorem \ref{defpropmap} assuming Propositions \ref{sqzerolift} and \ref{defhilbert} below, which are the main steps in these two reductions.  The proofs of these propositions are given in the subsequent subsections.

For a closed algebraic subspace $X\subset X'$ for which $X$ and $X'$ have the same associated topological space, we say $X'$ is a thickening of $X$ (see for example \cite[\href{https://stacks.math.columbia.edu/tag/05ZK}{Tag 05ZK}]{stacks-project}), and a square-zero thickening if the ideal $I$ of $X$ in $X'$ is square-zero.  Likewise for definable complex analytic spaces and analytic spaces. 

The following proposition allows us to lift algebraizations through definable thickenings.

\begin{prop}\label{sqzerolift}

Let $f:W\ra Z$ be a proper dominant morphism of algebraic spaces. Suppose we have an algebraic square-zero thickening $W\ra W'$, a 
definable closed immersion $Z^\df\ra \mathcal{Z}'$, and a morphism $\phi':W'^\df\ra \mathcal{Z}'$ which fits into a commutative diagram  
\[\xymatrix{
 W^\df\ar[d]_{f^{\df}}\ar[r]&W'^\df\ar[d]^{\phi'}\\
 Z^\df\ar[r]&\mathcal{Z}'
 }\] 
 
Then the following are uniquely defined: an algebraic square-zero thickening $Z\to Z''$, a definable closed immersion $Z''^\df\to \mathcal{Z}'$, and a (proper) dominant morphism $f':W'\to Z''$ of algebraic spaces, such that we have commutative diagrams

\[\xymatrix{
 W\ar[dd]_{f}\ar[r]&W'\ar[dd]^{f'}		&&		W'^\df\ar[rd]^{\phi'}\ar[dd]_{f'^\df}&\\
 &		&&		&\mathcal{Z}'\\
 Z\ar[r]&Z''		&&		Z''^\df\ar[ur]&.
 }\] 

\end{prop}

For an algebraic space $X$, let $\Hilb(X)$ be the Hilbert space (see for example \cite[\href{https://stacks.math.columbia.edu/tag/0D01}{Tag 0D01}]{stacks-project}) of proper algebraic subspaces of $X$, and $\mathcal{D}(X^{\an})$ be the Douady space of compact analytic subspaces of $X^{\an}$ \cite{Douady}.  Since the analytification of a flat family of proper algebraic subspaces of $X$ yields a flat family of compact analytic subspaces of $X^{\an}$, the universal family on $\Hilb(X)$ yields a canonical analytic map $\Hilb(X)^{\an} \rightarrow \mathcal{D}(X^{\an})$, which is a bijection on points since every compact analytic subspace of $X^{\an}$ is algebraic by ordinary GAGA (or for example Theorem \ref{defchow}).  As the functors represented by $\Hilb(X)$ and $\mathcal{D}(X^\an)$ are the same over artinian rings (since by GAGA again the deformation spaces are the same and the algebraic obstruction clearly analytifies to the analytic one), $\Hilb(X)^{\an} \rightarrow \mathcal{D}(X^{\an})$ is in fact an isomorphism.

The next proposition essentially says that when $X$ is reduced in the setup of Theorem \ref{defpropmap}, the morphism $\phi$ is generically algebraic.  

\begin{prop}\label{defhilbert}
Let $X$ be a smooth algebraic space, $\mathcal{U}$ a smooth definable complex analytic space, and $\phi:X^\df\to\mathcal{U}$ a smooth proper definable analytic morphism.  Then $\phi:X^\df\to\mathcal{U}$ is the definabilization of an algebraic morphism $f:X\to U$ and the associated morphism $U\to \Hilb(X)$ is a closed embedding.
\end{prop}

\begin{proof}[Proof that Propositions \ref{sqzerolift} and \ref{defhilbert} imply Theorem \ref{defpropmap}]The uniqueness property follows immediately from Proposition \ref{prop mapdescent}, Corollary \ref{defchowmap}, and the corresponding uniqueness property of the analytic image, so we need only prove the existence of such a factorization.

The proof proceeds by induction on $\dim X$, the base case being trivial.  By repeatedly applying Proposition \ref{sqzerolift}, we may assume $X$ is reduced.  By Remmert's theorem, the analytic image $\phi^\an(X^\an)$ is an analytic subvariety of $\mathcal{S}^\an$.  By Propposition \ref{defanset} there is a unique structure of a reduced definable complex analytic space on $\phi^\an(X^\an)$ through which $\phi$ factors, and so replacing $\mathcal{S}$ with $\phi^\an(X^\an)$ we may assume $\phi^\an$ is surjective on points and $\mathcal{S}$ is reduced.  By Corollary \ref{defchowmap} it suffices to algebraize $\mathcal{S}$.

We next reduce to the case that the analytification of $\phi$ is a proper modification---that is, $\phi^\an:X^\an\to \mathcal{S}^\an$ is proper and induces an isomorphism $(\phi^{\an})^{-1}(\mathcal{U})\to\mathcal{U}$ for a dense analytic Zariski open subset $\mathcal{U}\subset\mathcal{S}^{\an}$. For this reduction it is enough to assume $\mathcal{S}$ is (analytically) irreducible:  if $\mathcal{S}_k$ is the image of an irreducible component $X_k$ of $X$, and if there is a proper modification $\psi:B_k^\df\to\mathcal{S}_k$, then we may replace $\phi$ with the proper modification $\sqcup_k\psi_k:\sqcup_k B_k\to\mathcal{S}$.  We may additionally assume $X$ is smooth and irreducible by replacing $X$ with a component of a resolution which dominates $\mathcal{S}$.   

Observe that $\phi$ is smooth over a dense smooth definable Zariski open subset $\mathcal{U}\subset\mathcal{S}$.  Indeed, the regular locus $\mathcal{S}^\mathrm{reg}\subset \mathcal{S}$ is a dense definable Zariski open subset, and the smooth locus over $\mathcal{S}^\mathrm{reg}$ is determined by the rank of the Jacobian.  Both of these conditions are clearly definable on covers by basic definable analytic varieties.  Let $(X_\mathcal{U})^\df=\phi^{-1}(\mathcal{U})$, which is algebraic by Theorem \ref{defchow}.  Applying Proposition \ref{defhilbert} we conclude that $\phi_\mathcal{U}:X_\mathcal{U}^\df\to\mathcal{U}$ is the definabilization of $f_U:X\to U$ and that $U\to\Hilb(X_\mathcal{U})$ is a closed embedding.  Evidently $f_U$ is the restriction of the universal family.

Obviously $\Hilb(X_\mathcal{U})$ is an open subset of $\Hilb(X)$.  Let $B$ be the closure of the image of $U$ in $\Hilb(X)$, $V_B\to B$ the restriction of the universal family, $\tilde{B}\to B$ the normalization, and $V_{\tilde{B}}\to\tilde B$ the base-change of the universal family.
We then have solid diagrams
\[\begin{tikzcd}
&&V_{\tilde B}\ar{ld}\ar{rd}&&&&\\
&V_{B}\ar{rd}\ar{ld}&  	&\tilde{B}\ar{ld}&  	&(V_{\tilde{B}})^\df\ar{dd}\ar{ld}\ar{rd}&\\
X&&B 	&&	X^\df\ar{rd}&&\tilde{B}^\df.\ar[ld,"\psi",dashed]\\
&&   	&&  	&\mathcal{S}&
\end{tikzcd}\] 

The compact complex analytic subspaces of $X^\an$ that are contained in a fibre of $\phi^\an$ form a closed analytic subset of $\mathcal{D}(X^{\an})$. Since it contains the image of $U^\an$, it contains its closure $B^\an$. It follows that the vertical arrow $V_{\tilde{B}}^\df \rightarrow \mathcal{S}$ set-theoretically factors through $\tilde B^\an$, hence topologically factors through $\tilde B^\an$ (since we get a continuous map by precomposing with the open map $V_{\tilde{B}}^\an \rightarrow \tilde B^\an$), hence analytically factors through $\tilde B^\an$ (since by the normality of $\tilde B$, the holomorphic functions on an open subset of $\tilde B^\an$ are the continuous functions that are holomorphic in restriction to the trace of $(B')^\an$), hence definably factors through $\tilde{B}^\df$ by Proposition \ref{prop mapdescent}.  The resulting morphism $\psi$ is a proper modification, as it is proper (since $V^\df_{\tilde B}\to\mathcal{S}$ is proper and $V_{\tilde{B}}\to \tilde{B}$ is surjective) and an isomorphism over $\mathcal{U}$.

\item\vskip1em
We may therefore assume that the analytification of $X^\df\to\mathcal{S}$ is a proper modification.  By the inductive hypothesis and Theorem \ref{defchow}, the center\footnote{That is, the image of the exceptional locus in $X^\df$.} of $X^\df\to \mathcal{S}$ can be algebraized, so let $Z^{\df}\subset\mathcal{S}$ be the reduced center, and $W^\df=\phi^{-1}(Z^\df)$ equipped
with its reduced induced structure. For every positive integer $k$ let $W_k$ be the $k$th order thickening of $W$, and  $\mathcal{Z}_k$ the $k$th order thickening of $Z^{\df}$ in $\mathcal{S}$.

\def\colim{\operatorname{colim}}

By the inductive hypothesis, the induced morphism $\phi_k:W_k^\df\to \mathcal{S}$ factors as $W_k^\df\xrightarrow{f_k^\df} Z_k^\df\to \mathcal{S}$ where the first map is dominant and the second map is a closed immersion---in particular a definable thickening of $Z$ by Theorem \ref{gaga}.  Let $\bar W$ be the completion of $X$ along $W$ and $\bar Z=\colim Z_k$, both of which are formal algebraic spaces (see e.g. \cite[\href{https://stacks.math.columbia.edu/tag/0AIL}{Tag 0AIL}]{stacks-project}); we then have a morphism $\bar f:\bar W\to \bar Z$.  Let $\bar{\mathcal{Z}^\an}$ be the completion of $\mathcal{S}^\an$ along $\mathcal{Z}^\an$, which is a formal analytic space.  By the theorem on formal functions in the analytic category (see e.g. \cite[Corollary 4.5]{banica}), the natural map $(\phi^\an_*\O_{X^\an})^\wedge\to \lim \phi_*\O_{W_k^\an}$ is an isomorphism, and therefore the $Z_k^\an$ are cofinal in the $\mathcal{Z}^\an_k$.  Thus, the natural map $(\bar Z)^\an\to\bar{\mathcal{Z}^\an}$ is an isomorphism.

\begin{lem}\label{formal mod}
The morphism $\bar f:\bar W\to \bar Z$ is a formal modification.
\end{lem}
\begin{proof}
We refer to \cite[Definition (1.7)]{A} for the notion of a formal modification.  The proof of the claim is the same word for word as in \cite[Lemma (7.7)]{A}, using that $X^{\an}\ra \mathcal{S}^\an$ is a proper modification, as the verification\footnote{Note that Artin uses the notation 
\[
\xymatrix@R-5ex{
X'\ar[r]&X &&X\ar[r]&S\\
&&\mbox{for our}&&\\
Y'\ar[r]\ar[uu]&Y\ar[uu] && W\ar[r]\ar[uu]&Z.\ar[uu]
}
\]
} is entirely Zariski-local on $Z$.
\end{proof}

\begin{remark}
The reason \cite[Lemma (7.7)]{A} requires compactness is because compact spaces have at most \emph{one} algebraization, making the statement a lot cleaner. Note that we are essentially getting around this non-uniqueness issue by working with the ambient structure of a \emph{definable} complex analytic space, which is compatible with at most one algebraization.
\end{remark}

By the following theorem of Artin, we then conclude from Lemma \ref{formal mod} that $\phi :X^\an\to \mathcal{S}^\an$ is algebraized by $f:X\to S$, and by Proposition \ref{prop mapdescent} we have $S^{\df}\cong \mathcal{S}$.

\begin{thm}[Artin {\cite[Theorem (3.1)]{A}}]
Let $X$ be an algebraic space of finite type over a field, and $W\subset X$ be a closed subspace. Let $\bar{W}$ denote the formal completion of $X$ along $W$, and suppose that $\bar f:\bar{W}\rightarrow \bar{Z}$ is a formal modification. Then there is a modification $f:X\ra S$ which is an isomorphism on the complement of $W$ and such that the completion of $f$ along $W$ is isomorphic to $\bar f$.
\end{thm}

\end{proof}

\subsection{Proof of Proposition \ref{sqzerolift}} 

We have an exact sequence of coherent $\O_{W'}$-modules on $W'$: 
\begin{equation}0\ra I\ra \O_{W'}\ra \O_W\ra0\label{sequenceW}\end{equation} 
where both $I$ and $\O_W$ are coherent $\O_W$-modules. Taking the analytification on $W'$, by Theorem \ref{anfaithful} we get a sequence of coherent $\O_{W'^{\an}}$-modules\footnote{Note that the analytifications of $\O_W,I$ as $\O_{W'}$-modules are naturally the analytifications as $\O_W$-modules.} :

\begin{equation}0\ra I^{\an}\ra \O_{W'^{\an}}\ra \O_{W^{\an}}\ra0.\label{sequenceWan}\end{equation} 

 Viewing \eqref{sequenceW} (resp. \eqref{sequenceWan}) in the category of sheaves of abelian groups on $W$ (resp. $W^\an$), we have a natural coboundary map of sheaves $ f_*\O_W \xrightarrow{\partial} R^1f_*I$ (resp. $f^{\an}_*\O_{W^\an} \xrightarrow{\partial'} R^1f^{\an}_*I^{\an}$). Furthermore, the sheaves $f_*\O_W$ and $R^1f_*I$ (resp. $f^{\an}_*\O_{W^\an}$ and $R^1f^{\an}_*I^{\an}$) canonically have the structure of $\O_Z$-modules (resp. $\O_{Z^\an}$-modules).  Finally, we have natural morphisms of $\O_{Z^\an}$-modules $(f_*\O_W)^\an\to f^\an_*\O_{W^\an}$ and $(R^1f_*I)^\an\to R^1f^\an_*I^\an$ which are isomorphisms by ordinary GAGA (see e.g. \cite[Th\'eor\`eme 5.10]{toe} for the statement for algebraic spaces---here we are using that $f$ is compactifiable).

Observe that $\O_{\mathcal{Z'}^\an}$ surjects onto $\O_{Z^{\an}}$ by Theorem \ref{anfaithful} as $\O_{\mathcal{Z'}}$ surjects onto $\O_{Z^\df}$.  Thus, we have a commutative diagram of homomorphisms of $\O_{\mathcal{Z'^\an}}$-modules with exact rows
\begin{equation}\label{surject}
\begin{tikzcd}
  f^\an_*\O_{W'^\an}\arrow[r]&f^\an_*\O_{W^\an}\arrow{r}{\partial'}&R^1f^\an_*I^\an\\
  \O_{\mathcal{Z'^\an}}\arrow[r]\arrow[u]&\O_{Z^\an}\arrow[r]\arrow[u]&0
\end{tikzcd}
\end{equation}
where the first row comes from the long exact sequence associated to \eqref{sequenceWan}.  

\begin{lemma}\label{an cobound}

The coboundary map $f_*\O_W\xrightarrow{\partial} R^1f_*I$ associated to \eqref{sequenceW} is a homomorphism of $\O_Z$-modules and analytifies to the coboundary map $f^{\an}_*\O_{W^\an}\xrightarrow{\partial'} R^1f^{\an}_*I^{\an}$ associated to \eqref{sequenceWan}.

\end{lemma}

\begin{proof}
Both statements are local on $Z$, so we replace $Z$ with an affine (\'etale) open.  By the Leray spectral sequence (and the fact that affines are Stein), the canonical $\C$-linear maps $H^1(W , I)  \to H^0(Z, R^1f_*I)$ and $H^1(W^\an,I^\an)  \to H^0(Z^\an, R^1f^{\an}_*I^{\an})$ are isomorphisms. It is therefore enough to prove the corresponding two statements for the coboundary map on global cohomology. Taking the cohomology of \eqref{sequenceW} and \eqref{sequenceWan}, it follows that we have a natural diagram
\begin{equation}\label{partialsq}\begin{tikzcd}
  H^0(W,\O_W)\arrow{d}\arrow{r}{\partial}&H^1(W,I)\arrow{d}\\
  H^0(W^\an,\O_{W^\an})\arrow{r}{\partial'}&H^1(W^\an,I^\an).
\end{tikzcd}\end{equation}
As the cohomology of algebraic and analytic coherent sheaves can  both be computed via \v{C}ech cohomology with respect to an affine (\'etale) cover, it follows that this diagram commutes.  

Note that for any algebraic coherent sheaf $F$ on $W$, the map $H^0(W,F)\to H^0(W^\an,F^\an)$ is injective since the maps $F_w\to F_w^\an$ on stalks are injective.  This together with the ordinary GAGA isomorphisms mentioned above imply both vertical maps are injective. Now, by a diagram chase in \eqref{surject} we have that $H^0(Z^\an,\O_{Z^\an})$ is killed by $\partial'$, and therefore that $H^0(Z,\O_Z)$ is killed by $\partial$.  Since $\partial$ is a derivation, it follows that it is a homomorphism of $H^0(Z,\O_Z)$-modules, and the second claim follows from the commutativity of \eqref{partialsq} and the ordinary GAGA isomorphisms.
\end{proof}

Let $F$ be the image of $f_*\O_{W'}\to f_*\O_W$, which is also the kernel of $f_*\O_W\ra R^1f_*I$.  By the preceding lemma, $F$ is an $\O_Z$-module and analytifies to the kernel of $f^\an_*\O_{W^\an}\ra R^1f^\an_*I^\an$.  From \eqref{surject}, $F^\an$ contains the image of $\O_{Z^\an}$ in $f^\an_*\O_W^{\an}$, and so by faithfulness of ordinary analytification $F$ contains the image of $\O_Z$ in $f_*\O_W$.  We define the sheaf of rings $R$ on $Z$ as $R=\O_Z\oplus_{f_*\O_W} f_*\O_{W'}$. It follows that $R$ surjects onto $\O_Z$, with square-zero kernel $J=f_*I$. 

\begin{lemma}\label{schemethick}

Suppose $Z$ is an algebraic space, and $J$ is a coherent sheaf on $Z$.  Let $R$ be a sheaf of rings on the \'etale site $Z^{\mathrm{\acute{e}t}}$ of $Z$ such that
\[0\ra J\ra R\ra \O_Z\ra 0\]
is a first order thickening\footnote{Recall this means that $R\ra \O_Z$ is a homomorphisms of sheaves of rings and that $J$ with its induced ideal structure is of square zero.}. Then $(Z^{\mathrm{\acute{e}t}},R)$ is an algebraic space.
\end{lemma}

\begin{proof}See for example \cite[\href{https://stacks.math.columbia.edu/tag/05ZT}{Tag 05ZT}]{stacks-project}. 
\end{proof}

We thus have an algebraic thickening $Z_0=(Z^{\mathrm{\acute{e}t}},R)$ and a diagram
 \[\xymatrix{
 W\ar[d]\ar[r]&W'\ar[d]\\
 Z\ar[r]&Z_0
 }\] 
where $W'\to Z_0$ is dominant, since $f$ is dominant.  By Proposition \ref{prop mapdescent} we also have a diagram
 \[\xymatrix{
 W^\df\ar[dd]\ar[r]&W'^\df\ar[dd]\ar[dr]&\\
 &&\mathcal{Z'}\\
 Z^\df\ar[r]&Z_0^\df\ar[ur]
 }\] 

Note that $Z_0^{\df}\to\mathcal{Z'}$ may well not be immersive. We claim that the image is algebraic.  The definable complex analytic space structure on the image is defined by the image $\mathcal{T}$ of the map $\O_{\mathcal{Z'}}\to R^\df$, and we have a diagram with exact rows
\[\xymatrix{
0\ar[r]&J^\df\ar[r]&R^\df\ar[r]&\O_{Z^\df}\ar[r]&0\\
0\ar[r]&\mathcal {K}\ar[u]\ar[r]&\mathcal{T}\ar[u]\ar[r]&\O_{Z^\df}\ar@{=}[u]\ar[r]&0.
}\]  

Now $\mathcal{K}$ is a coherent $\O_{Z^\df}$-submodule of $J^\df$ and therefore the definabilization of an algebraic $K\subset J$ by Theorem \ref{gaga}.  Letting $R'=R/K$, we have another algebraic space 
$(Z^{\mathrm{\acute{e}t}},R')$ by Lemma \ref{schemethick}. Note that $\mathcal{T}/K^\df\cong \O_{Z^\df}$ is a section of $R'^\df\to\O_{Z^\df}$ which is algebraic by Theorem \ref{gaga} as it gives an isomorphism $R'^\df\cong \O_{Z^\df}\oplus (J/K)^\df$. It is easy to check that the sheaf of rings $T=R\oplus_{R'}\O_Z$ definabilizes to $\mathcal{T}$, 
and by Lemma \ref{schemethick}, $Z''=(Z^{\mathrm{\acute{e}t}},T)$ is algebraic.

Since $(Z^{\df}, \mathcal{T})$ is the image of $\phi'$ by construction, this concludes the proof of Proposition \ref{sqzerolift}.\qed

\subsection{Proof of Proposition \ref{defhilbert}} First observe the following:
\begin{lemma}\label{lem hilb morph}
For a smooth proper morphism of smooth irreducible complex analytic spaces $g:\mathcal{Y}\to \mathcal{Z}$ with $g_{*}\O_{\mathcal{Y}}=\O_\mathcal{Z}$, the irreducible component $\mathcal{D}$ of $\mathcal{D}(\mathcal{Y})$ containing the fibers of $g$ is canonically identified with $\mathcal{Z}$ together with $g$ as the universal family.
\end{lemma}
\begin{proof}
We will show the induced holomorphic map $\mathcal{Z}\to\mathcal{D}$ is an isomorphism.

For any $z \in \mathcal{Z}$, and recalling that $\mathcal{Z}$, $\mathcal{Y}$ and $\mathcal{Y}_z$ are smooth, we have on the one hand that the normal bundle $\mathrm{N}_{\mathcal{Y}_z / \mathcal{Y}}$ of $\mathcal{Y}_z$ in $\mathcal{Y}$ is canonically identified with $T_{\mathcal{Z}, z} \otimes \O_{\mathcal{Y}_z}$, whereas on the other hand the tangent space of $\mathcal{D}(\mathcal{Y})$ at the point corresponding to $\mathcal{Y}_s$ is canonically identified with $\mathrm{H}^0(\mathcal{Y}_s, \mathrm{N}_{\mathcal{Y}_z / \mathcal{Y}})$. Since $g_{*}\O_{\mathcal{Y}}=\O_\mathcal{Z}$, it follows that the holomorphic map $\mathcal{Z}\to\mathcal{D}$ is an isomorphism on tangent spaces.  As $\mathcal{Z}$ is smooth, it follows that $\mathcal{Z}\to\mathcal{D}$ is an isomorphism on completed local rings.

It remains to show that $\mathcal{Z}\to\mathcal{D}$ is bijective on points.  The injectivity is clear since there is only one reduced compact analytic space of maximal dimension contained in a given (necessarily irreducible) fibre of $g$.

It follows that $\mathcal{Z}\to\mathcal{D}$ is an open immersion with dense image. But the compact complex analytic subspaces of $\mathcal{Y}$ that are mapped to a point by $g$ form a closed analytic subset of $\mathcal{D}(\mathcal{Y})$. Therefore, the compact complex analytic subspaces of $\mathcal{Y}$ that correspond to points of $\mathcal{D}$ are mapped to a point by $g$. Letting $\mathcal{V}\to\mathcal{D}$ be the universal family, it follows that the composition $\mathcal{V}\to \mathcal{Y}\to \mathcal{Z}$ of the evaluation map $\mathcal{V}\to\mathcal{Y}$ with $g$ factors through the Stein factorization of the projection $\mathcal{V}\to\mathcal{D}'\to\mathcal{D}$ by the universal property of Stein factorization (see \cite[p.214]{Grauert}). Since the resulting composition $\mathcal{D}'\to \mathcal{Z}\to\mathcal{D}$ is $\mathcal{D}'\to \mathcal{D}$, the map $\mathcal{Z}\to\mathcal{D}$ is surjective on points. \end{proof}

It then follows from the lemma that for $d$ such maps $g_j:\mathcal{Y}_j\to\mathcal{Z}$ with ${g_{j}}_{*}\O_{\mathcal{Y}_j}=\O_\mathcal{Z}$, letting $g:\mathcal{Y}:=\sqcup_j\mathcal{Y}_j\to \mathcal{Z}$, the irreducible component of $\mathcal{D}(\mathcal{Y})$ containing the fibres of $g$ is canonically identified with $\mathcal{Z}^d$ with universal family given by
\[\bigsqcup_{j=1}^d \mathcal{Z}\times\cdots\times \mathcal{Y}_j\times\cdots\times \mathcal{Z}\to\mathcal{Z}^d\]
where $\mathcal{Y}_j$ is inserted in the $j$th slot in the $j$th factor.

Returning to the setup of Proposition \ref{defhilbert}, let $X^\an\rightarrow\mathscr{U}'\rightarrow\mathcal{U}^\an$ be the Stein factorization of $\phi^\an$.  Since $\mathscr{U}'\to\mathcal{U}^\an$ is finite \'etale, by Lemma \ref{riemann exist} and Proposition \ref{propercover} this diagram is the analytification of a diagram $X^\df\xrightarrow{\phi'}\mathcal{U}'\xrightarrow{\nu}\mathcal{U}^\df$, as $\phi'$ is obtained by taking connected components of $\phi$ on a definable cover of $\mathcal{U}$.

Let $H$ be the component of $\Hilb(X)$ which contains the general fiber of $X^\df\to \mathcal{U}$, and $V\subset X\times H$ the universal subscheme with projections $p_1:V\to X$ and $p_2:V\to H$.  We therefore obtain a cartesian diagram in the analytic category

\begin{equation}
\label{pullback sq}
\begin{tikzcd}
  X^\an\arrow{d}[swap]{\phi^\an}\arrow{r}&V^\an\arrow{d}{p_2^\an}\\
  \mathcal{U}^\an\arrow{r}{\beta}&H^\an.
\end{tikzcd}
\end{equation}

By definable Chow (Theorem \ref{defchow}), the proposition will follow if this diagram is the analytification of a diagram of definable complex analytic spaces and moreover if $\beta$ is a closed immersion.  It will be sufficient to verify that this is the case on a definable open cover of $H^\df$.

Let $\mathcal{U}_i$ be a definable simply-connected cover of $\mathcal{U}$ as in Remark \ref{simply}, and let $\mathcal{U}'_{i,1},\ldots,\mathcal{U}'_{i,d}$ be the $d=\deg(\mathcal{U}'/\mathcal{U})$ components of $\nu^{-1}(\mathcal{U}_i)$. For a fixed $i$, consider in $ H^\an$ the subset of compact analytic subspaces of $X^{\an}$ which meet each component $\phi'^{-1}(\mathcal{U}'_{i,j})$.
This is naturally a definable open subset $\mathcal{H}_i$ of $H^\df$, since in the diagram
\begin{equation}\begin{tikzcd}
&V^\df\ar{ld}[swap]{p_1}\ar{rd}{p_2}&\\
X^\df\ar{rd}[swap]{\phi'}&&H^\df\\
&\mathcal{U}'&
\end{tikzcd}\label{factorhilb}\end{equation} 
it is the intersection of $p_2 \left( (\phi'\circ p_1)^{-1}(\mathcal{U}'_{i,j}) \right)$ for $j=1,\ldots,d$, and $p_2$ is flat hence open. Moreover, $\beta(\mathcal{U}^\an)$ is contained in the union of the $(\mathcal{H}_i)^\an$.

Let $\mathcal{X}_{i,j}:=\phi'^{-1}(\mathcal{U}'_{i,j})$ and consider 
\[\mathscr{X}_i:=\bigsqcup_{j=1}^d \mathcal{U}'_{i,1}\times\cdots\times \mathcal{X}_{i,j}\times\cdots\times \mathcal{U}'_{i,d}\to\mathscr{U}'_i:=\prod_{j=1}^d\mathcal{U}'_{i,j}.\]
This is the analytic family of subschemes of $X$ obtained by taking a union of fibers of $\phi':X^\df\to\mathcal{U}'$ over each of the open sets $\mathcal{U}_{i,j}'\subset\mathcal{U}'$ for $j=1,\ldots,d$.  In particular, it is a flat analytic family of subvarieties parametrized by $\mathcal{H}_i$, and by Lemma \ref{lem hilb morph} and the ensuing discussion, the horizontal morphisms in the resulting cartesian diagram

\begin{equation}
\label{pullback sq2}
\begin{tikzcd}
  \mathscr{X}_i^\an\arrow{d}\arrow{r}&(V^\an)_{\mathcal{H}_i}\arrow{d}{p_2^\an}\\
  (\mathscr{U}'_i)^\an\arrow{r}&\mathcal{H}_i^\an
\end{tikzcd}
\end{equation}are isomorphisms.  The inverse of the top map is obtained by taking a union of base-changes of connected components of the map $p_1^\df:(V^\df)_{\mathcal{H}_i}\to X_{\mathcal{U}_i}$ and is definable.  Thus, by Proposition \ref{prop mapdescent}, \eqref{pullback sq2} is the analytification of the right part of the diagram

\begin{equation}
\label{pullback sq3}\notag
\begin{tikzcd}
  X_{\mathcal{U}_i}\ar{r}\ar{d}[swap]{\phi_{\mathcal{U}_i}}&\mathscr{X}_i\arrow{d}\arrow{r}{\cong}&(V^\df)_{\mathcal{H}_i}\arrow{d}{p_2^\df}\\
  \mathcal{U}_i\ar{r}&\mathscr{U}'_i\arrow{r}{\cong}&\mathcal{H}_i.
\end{tikzcd}
\end{equation}
The left part of the diagram consists of the natural maps and is obviously definable, and $\mathcal{U}_i\to\mathscr{U}_i'$ is a clearly closed immersion.  The outer square analytifies to \eqref{pullback sq} restricted to $\mathcal{H}_i$, thus proving the claim.\qed

\subsection{Algebraizing analytic maps from algebraic varieties}
We conclude this section with a brief discussion of some of the subtleties involved in algebraizing a proper analytic morphism $\phi:X^\an\to\mathcal{S}$ from an algebraic space $X$ without tameness hypotheses.  For simplicity assume $\mathcal{S}$ and $X$ are irreducible.

The Hilbert space part of the proof in section \ref{sectrestrict} that Propositions \ref{sqzerolift}and \ref{defhilbert} imply Proposition \ref{defpropmap} shows that if $X^\an\to \mathcal{S}$ is flat with reduced irreducible generic fiber, then $\phi$ identifies $\mathcal{S}$ with a component of the Hilbert space of $X$ and is therefore algebraic.  This is the same argument used by Sommese to prove \cite[Proposition III and Remark III-C]{Som2} (see Theorem \ref{sommese thm}).

On the other hand, it is not hard to produce examples of nonalgebraizable finite flat maps $\phi:X^\an\to\mathcal{S}$.  The following example shows this is possible even assuming the ``equivalence relation" $R_\phi :=X^\an\times_\mathcal{S}X^\an\subset X^\an\times X^\an$ of $\phi$ is the analytification of a (possibly non-reduced) algebraic subspace, albeit for $\mathcal{S}$ non-normal.

\begin{eg}\label{algbraizing:counterexample}
Let $X=\mathbb{A}^2$ with coordinates $(x,y)$.  Let $\mathcal{S}\subset \C^4$ with coordinates $x,A,B,C$ be cut out by $A^5=B^3,AC=B^2\sin(x),BC=A^4\sin(x)$.  Let $X^\an\to \mathcal{S}$ be the map $(x,y)\mapsto (x,y^3,y^5,y^7\sin(x))$.  First, observe that 
\[R_\phi=V(x_1-x_2, y_1^3-y_2^3,y_1^5-y_2^5)\subset X^\an\times X^\an\]
where $(x_1,y_1)$ and $(x_2,y_2)$ are the coordinates on the two factors.  Indeed, we can think of $X^\an$ as defined in $\mathcal{S}\times\C$ (with $y$ being the last coordinate) by the ideal $(A-y^3,B-y^5,C-y^7\sin(x))$, in which case $R_\phi$ is cut out in $X^\an\times \C$ (with $y_2$ being the last coordinate) by the ideal $(y_1^3-y_2^3,y_1^5-y_2^5,(y_1^7-y_2^7)\sin(x))$, but
\[y_1^7-y_2^7=-y_1^2y_2^2(y_1^3-y_2^3)+(y_1^2-y_2^2)(y_1^5-y_2^5).\]
On the other hand, the rational function $y$ on $X$ descends to a meromorphic function on $\mathcal{S}$ but cannot be rational on $\mathcal{S}$ with respect to any algebraic structure, since it is nonregular along infinitely many divisors.
\end{eg}

It seems likely to the authors that for $X^\an,\mathcal{S}$ normal analytic varieties (in particular irreducible and reduced), a proper analytic morphism $X^\an\to\mathcal{S}$ with connected fibers and an algebraic equivalence relation could still be non-algebraizable, but we have not been able to construct such an example.

\section{A quasi-projectivity criterion}\label{sectionqpcrit}
Recall that by convention all the algebraic spaces that we consider are of finite type over $\C$.

Given a line bundle $L$ on an algebraic space $X$, we prove in this section two criteria ensuring that $X$ is a scheme and that $L$ is ample.

\begin{prop}\label{embedding criterion}
Let $X$ be an algebraic space and $L$ a line bundle on $X$. Let $S \subset \Gamma_\ast(X, L) := \bigoplus_{d \geq 0} \Gamma(X, L^d)$ be an integrally closed graded subalgebra that separates points of $X$. Then there exist an integer $d \geq 1$ and a finite-dimensional subspace $V \subset S_d$ such that the corresponding morphism $X \rightarrow \mathbb{P}(V^\vee)$ is defined everywhere and is an immersion. In particular, $X$ is a scheme and $L$ is ample.
\end{prop}

In the statement, the condition that $S$ separates points means that for any two distinct points $P$ and $Q$ in $X$ there exists a section in some $S_d$ that vanishes on $P$ but not on $Q$.

\begin{proof}
If $V$ is a finite-dimensional subspace of $S_d$ for some positive integer $d$, and $B_V$ is the reduced support of the cokernel of the canonical morphism of coherent $\O_X$-modules $V \otimes_{\C} \O_X \rightarrow L^d  $, then we have a canonical map $\phi_V : X - B_V \rightarrow \mathbb{P}(V^\vee)$ such that $\phi_V^\ast \O(1) \simeq (L^d)_{|X-B_V} $.

We denote by $B_d$ the intersection of the $B_V$'s over all finite-dimensional $V \subset S_d$. Observe that $B_{d \cdot d^\prime} \subset B_d \cap B_{d'}$ for every integers $d, d^\prime \geq 1$ since $S$ is a graded algebra. Since by assumption the intersection of the $B_d, d \geq 1$, is empty, the noetherianity of $X$ implies that there exists $d \geq 1$ and a finite-dimensional $V \subset S_d$ such that $B_V$ is empty.

Given a finite-dimensional $V \subset S_d$ for some $d \geq 1$ such that $B_V$ is empty, let $R_V \subset X \times X$ denote the reduced equivalence relation induced by $\phi_V$. Since $S$ separates points, the intersection of the $R_V$'s over all such $V$'s is equal to the reduced diagonal, therefore by noetherianity of $X$ there exists $V$ such that $\phi_V$ is defined everywhere and injective on points.

In particular, $X$ is in fact a scheme since it admits a quasi-finite map to a scheme \cite[\href{https://stacks.math.columbia.edu/tag/0417}{Tag 0417}]{stacks-project}.
Moreover, by Zariski's main theorem \cite[\href{https://stacks.math.columbia.edu/tag/082K}{Tag 082K}]{stacks-project}, the map $\phi_V : X \rightarrow \mathbb{P}(V^\vee)$ factors as $g\circ i$ for $g:X' \rightarrow \mathbb{P}(V^\vee)$ finite and $i:X\to X'$ an open immersion. Since the pull-back of an ample line bundle by a finite map or an immersion is still ample, we get that $L$ is ample.

Finally, $i$ and $g$ induce morphisms of graded algebras:
\[  \bigoplus_{d \geq 0} \Sym^d V \to \bigoplus_{d \geq 0} \Gamma(X', g^\ast \O_{\mathbb{P}(V^\vee)}(d)) \to \bigoplus_{d \geq 0} \Gamma(X, L^d). \]
Note that we can assume without loss of generality that $i(X)$ is dense in $X'$, so that the morphism on the right is injective. Since by construction the composition of the two morphisms is also injective, the morphism on the left is injective too. The morphism $g$ being finite, it follows that the extension 
\[  \bigoplus_{d \geq 0} \Sym^d V \to \bigoplus_{d \geq 0} \Gamma(X', g^\ast \O_{\mathbb{P}(V^\vee)}(d))  \]
is finite. But $S$ is integrally closed in $\bigoplus_{d \geq 0} \Gamma(X, L^d)$ by assumption, hence we get that $\bigoplus_{d \geq 0} \Gamma(X', g^\ast \O_{\mathbb{P}(V^\vee)}(d))$ is contained in $S$.   
\end{proof}

In what follows, given a reduced algebraic space $Y$, we say that a projective log smooth pair $(\bar{X},D)$ is a log-resolution of $Y$ if, setting $X := \bar{X} - D$, one is given a proper morphism $X \rightarrow Y$ which is birational in restriction to any irreducible component of $Y$. Existence of log-resolutions can be proved as follows. Thanks to Chow's lemma \cite[\href{https://stacks.math.columbia.edu/tag/088U}{Tag 088U}]{stacks-project}, there exists a complex projective scheme $\bar W$, a dense open $W \subset \bar W$ and a proper morphism $W \rightarrow Y$ which is an isomorphism in restriction to a dense open of $Y$. A log-resolution is then obtained by first replacing $\bar W$ with its normalization and then applying Hironaka desingularization Theorem to its irreducible components.

\begin{setting}\label{setup}
Let $L$ be a line bundle on an algebraic space $Y$ with the following property.  For every reduced closed subspace $Z \hookrightarrow Y$ and any log-resolution $(\bar{X},D)$ of $Z$, the pull-back of the restriction $L_{Z}$ extends as a nef and big line bundle $L_{\bar X}$ on $\bar X$, and this extension is functorial with respect to morphisms of log-resolutions of $Z$.
\end{setting}

\begin{defn}\label{vanishingsect}
Assume Setting \ref{setup}.  Given a closed subscheme $Z \hookrightarrow Y$, we say a section $s$ of $L^m_Z$ \emph{vanishes at the boundary} if for some log-resolution $(\bar X,D)$ of $Z$ the section $s$ pulls backs and extends to a section of $L_{\bar X}^m(-D)$.  We let $\Gamma_{van}(Z,L^m_Z)\subset \Gamma(Z,L^m_Z)$ denote the linear subspace of sections vanishing at the boundary, which is finite-dimensional as $\Gamma_{van}(Z,L^m_Z)$ injects into $\Gamma(\bar X,L_{\bar X}^m)$.
\end{defn}
 
Note that if the condition on $s$ holds for one log-resolution then it holds for any log-resolution, since any morphism of log-pairs $(\bar X',D') \rightarrow (\bar X,D)$ which is birational in restriction to any irreducible component of $\bar X$ induces an isomorphism of $\C$-vector spaces $\Gamma(\bar X,L_{\bar X}^m(-D) \rightarrow \Gamma(\bar X',L_{\bar X'}^m(-D')$, and any two log-resolutions are dominated by a third-one. Moreover, $s$ vanishes at the boundary if and only if $s^\reduced$ does.  
Note finally that the ring $\bigoplus_n \Gamma_{van}(Y,L_Y^n) $ is integrally closed in
$\bigoplus_n \Gamma(Y,L_Y^n) $, since a meromorphic section $s$ which satisfies a monic polynomial relation with coefficients that vanish at the boundary must also vanish at the boundary.

\begin{thm}\label{qproj gen}
Assume Setting \ref{setup}.  Then $Y$ is a scheme and $L$ is an ample line bundle. Moreover, for every $n \gg 1$, the natural morphism $Y \rightarrow \mathbb{P}(\Gamma_{van}(Y,L^n)^\vee)$ is defined everywhere and is an immersion.
\end{thm}

\begin{proof}
The theorem is a consequence of the following more precise result, thanks to Proposition \ref{embedding criterion}:

\begin{claim} For any closed, reduced zero-dimensional subscheme $P\subset Y$, the restriction $\Gamma_{van}(Y,L_Y^n)\to \Gamma(P,L_P^n)$ is surjective for some positive integer $n$.
\end{claim}

Observe that all the closed subschemes of $Y$ satisfy the assumptions of Theorem \ref{qproj gen}, therefore by Noetherian induction we can assume that the claim is satisfied by any closed subscheme distinct from $Y$. The case where $Y$ has dimension zero being trivial, we assume from now on that $d=\dim Y \geq 1$.

\item\vskip1em
\emph{Step 1.}  We first show we may assume $Y$ is reduced.  If $Y$ is non-reduced, then we can write $Y$ as a thickening of a subspace $Y_0$ by a square-zero sheaf of ideals $I$. By the induction statement and Proposition \ref{embedding criterion} applied to $Y_0$, we can pick an embedding $Y_0\ra \PP^m$  corresponding only to sections that vanish at the boundary. Thus, we can find a section $g\in \Gamma_{van}(Y_0,L_{Y_0}^n)$ such that $(Y_0)_g$ is
affine and contains $P$. Also, we may pick vanishing sections $s_1,\dots, s_k$ of $L_{Y_0}^m$ whose image span $\Gamma(P,L_{P}^m)$. It follows that the images of the vanishing sections $g \cdot s_1,\dots, g \cdot s_k$ of $L_{Y_0}^{n +m}$ span $\Gamma(P,L_{P}^{n+m})$. Finally, since the open subschemes $(Y_0)_{g \cdot s_i}$ are affine, the vanishing sections $(g \cdot s_i)^r$ lift to vanishing sections of $L^{r (n+m)}$ for some $r \geq 1$ thanks to \cite[Lemme 4.5.13.1]{EGA2}.

\item\vskip1em\noindent\emph{Step 2.}  By Step 1, we assume $Y$ is reduced. Take a log-resolution $(\bar X,D)$ of $Y$. Letting $X := \bar{X} - D$, the corresponding morphism $f:X\ra Y$ is an isomorphism outside of a dimension $d-1$ subset.

\begin{lemma}\label{descent}
Let $X,Y$ be algebraic spaces and $f:X\to Y$ a proper dominant morphism. Then there is an algebraic subspace $S\subset Y$ supported on the locus where $f$ is not an isomorphism,  such that for any line bundle $L$ on $Y$, a section $s\in \Gamma(X,f^*L)$
is in the image of $\Gamma(Y,L)$ if and only if its restriction $s|_T\in \Gamma(T,f^*L|_T)$ is in the image of $\Gamma(S,L|_S)$, where $T=S\times_Y X$. 
\end{lemma}
  \begin{proof}
  Let $Q$ be the cokernel of the map $\O_Y\to f_*\O_X$ and $S$ its scheme-theoretic support.  Then we have a diagram
  \[\xymatrix{
  0\ar[r]&\O_Y\ar[d]\ar[r]&f_*\O_X\ar[d]\ar[r]&Q\ar[r]\ar@{=}[d]&0\\
  &\O_S\ar[r]&f_*\O_T\ar[r]&Q\ar[r]&0.
  }  \]
  Tensoring by $L$ and taking cohomology, the result follows.
  \end{proof}

In the present context, let $S\subset Y$ and $T\subset X$ be the closed subspaces guaranteed by the lemma, and let $Z$ be the scheme theoretic union of $S$ and $P$, that is the closed subscheme of $X$ defined by the intersection of the two ideal sheaves defining $S$ and $P$.  Likewise, let $W$ be the scheme theoretic union of $T$ and $f^{-1}(P)$. 

\item\vskip1em\noindent\emph{Step 3.}
\begin{lemma}
There is a (nonzero) effective divisor $\bar E$ in $\bar X$ containing $W$ such that for $m\gg 1 $ and for every section $s\in \Gamma(\bar E,L_{\bar X}^m(-D)|_{\bar E})$ whose restriction $s|_W\in\Gamma(W,L_W^m)$ is in the image of $\Gamma(Z,L_Z^m)$, there is a section $t\in\Gamma_{van}(Y,L_Y^m)$ with $s|_{\bar E\cap X}=(f^*t)|_{\bar E\cap X}$.	
\end{lemma}
\begin{proof}
Let $A$ be an ample divisor on $\bar{X}$. The line bundle $L_{\bar{X}}$ is big on every component by the assumptions, so for some $n$ there is a section $\alpha$ of $L_{\bar X}^n(-A)$ whose zero locus $\bar E_0$ contains $W$.  For any $r>0$, setting $\bar E=r\bar E_0$ we thus have an exact sequence
\[ H^0(\bar X,L_{\bar X}^m(-D))\to H^0(\bar E,L_{\bar X}^m(-D)|_{\bar E})\to H^1(\bar X, L_{\bar X}^{m-nr}(-D+rA )).\] 
The line bundle $L_{\bar X}$ is nef, so by Fujita vanishing (\cite[Theorem 1]{Fujita}, see also \cite[Theorem 1.4.35]{LazarsfeldI}) the rightmost group is zero---and thus the first map is surjective---for some $r$ and any $m\geq nr$.  Now apply the previous step.
\end{proof}
\item\vskip1em\noindent\emph{Step 4.}
\begin{lemma}
There is a (nonzero) effective divisor $E'$ of $X$ containing $W$ such that for some integer $k$ and all $m\gg 1$, denoting $\bar E'$ the closure of $E'$ in $\bar X$, we have that for every section $s\in\Gamma(\bar E',L_{\bar X}^m(-kD)|_{\bar E'})$ whose restriction $s|_W\in \Gamma(W,L_W^m)$ is in the image of $\Gamma(Z,L^m_Z)$, there is a section $t\in\Gamma_{van}(Y,L_Y^m)$ with $s|_{E'}=(f^*t)|_{E'}$.
\end{lemma}
\begin{proof}
Write $\bar E=\bar{E}'+D'$ where $D'$ is supported on the boundary and every component of $\bar{E}'$ meets $X$.  Set $E'=\bar{E}'\cap X$.  Note that we have an exact sequence
\[0\to\O_{\bar{E}'}(-D')\to\O_{\bar E}\to\O_{D'}\to0\]
and so $\Gamma(\bar E',L_{\bar X}^m(-kD)|_{\bar E'})$ injects into $\Gamma(\bar E,L_{\bar X}^m(-D)|_{\bar E})$ for some fixed $k$ and all $m\geq 0$.  Now apply the previous step.
\end{proof}

\item\vskip1em\noindent\emph{Step 5.}  Let $F\subset Y$ be the image of $E'$. Applying the induction step to $F$, it follows that for some positive integer $n$ the map
$\Gamma_{van}(F,L_{F}^n)\ra \Gamma(P,L_{P}^n)$ is surjective. Pulling an appropriate symmetric power of these sections back to $\bar{E}'$ and applying Step 4, we see that these sections extend to vanishing sections of $Y$, as desired.
\end{proof}

\section{Algebraicity and quasi-projectivity of period maps}\label{sectionperiodmap}
In this section we prove the Theorem \ref{maingriffiths}.  \textbf{For this section we work over the o-minimal structure $\R_{\an,\exp}$.}

\subsection{Period images}
For background on period domains see for example \cite{CMSP}.  Let $\Omega$ be a pure polarized period domain with generic Mumford--Tate group $\G$ and $\Gamma\subset \G(\Q)$ an arithmetic lattice.  By \cite[Theorem 1.1]{BKT}, if $\Gamma$ is neat then $\Gamma\backslash \Omega$ has a canonical structure of a definable complex analytic variety (in fact, even over $\R_{\mathrm{alg}}$). Since every arithmetic lattice has a normal neat subgroup $\Gamma'$, using Proposition \ref{finitequotient}  we can equip $\Gamma\backslash\Omega$ with a definable complex analytic space structure as the categorical quotient of $\Gamma'\backslash\Omega$ by $G=\Gamma/\Gamma'$.

\begin{cor}\label{hodge}
Let $X$ be a reduced algebraic space, and $\phi: X^{\an}\ra (\Gamma\backslash \Omega)^\an$ a period map as in the introduction\footnote{That is, a locally liftable map satisfying Griffiths transversality on the regular locus.}. Then $\phi$ (uniquely) factors as $\phi =\iota^\an\circ f^\an$ for a dominant map $f:X\to Y$ of algebraic spaces and a closed immersion $\iota:Y^\df\to \Gamma\backslash\Omega$ of definable complex analytic varieties.
\end{cor}

\begin{proof}
Taking a resolution, it is enough to assume $X$ is smooth, and by a theorem of Griffiths \cite[Theorem 9.5]{G2} we may then assume that $\phi$ is proper.  By \cite[Theorem 1.3]{BKT}, $\phi: X^\an\to (\Gamma\backslash\Omega)^\an$ is the analytification of a map $X^\df\to\Gamma\backslash \Omega$ of definable complex analytic varieties.  Now apply Theorem \ref{defpropmap}.
\end{proof}

The uniqueness of the factorization is in the same sense as in Theorem \ref{defpropmap} (see Remark \ref{unicity}).  In fact, we obtain a version of Corollary \ref{hodge} over non-reduced bases, but as the following example illustrates we must require an admissibility condition for period maps on non-reduced bases. 
\begin{eg}\label{needdef}

Let $S^{\an}=\Gamma\backslash\Omega$
be a modular curve with level structure so that it is a smooth scheme, and let $Y=S\times_{\spec\C}\spec\C[\epsilon]/(\epsilon^2)$ be the trivial thickening of it. Given a global holomorphic derivation $D$ on $S^\an$ we can define a map $\phi:Y^{\an}\ra S^{\an}$ extending the identity map via $\phi^{\sharp}(s) = s+\epsilon Ds$. Since $S$ is affine we can pick $D$ to be non-algebraic, and then the map  $\phi$ will be
non-algebraizable.

\end{eg}
Definability provides a natural notion of admissibility for which the conclusion of Corollary \ref{hodge} holds true for non-reduced bases.  Moreover, Proposition \ref{vartomap} below shows that period maps associated to variations coming from algebraic families are automatically definable.

\begin{defn}\label{defn def period}
Let $X$ be an algebraic space (possibly non-reduced).  A definable period map of $X$ is a locally liftable map $\phi:X^\df\to \Gamma\backslash \Omega$ of definable complex analytic spaces such that for each irreducible component $Y$ of $X$ equipped with its reduced structure the associated (locally liftable definable) map $\phi_Y:Y^\df\to  \Gamma\backslash \Omega$ satisfies Griffiths transversality---that is, the (locally defined) map $T_{Y^\df}\to \phi_Y^* T_\Omega$ on the tangent sheaf $T_{Y}=(\Omega^1_{Y})^\vee$ factors through the Griffiths transverse subbundle.
\end{defn}

Note that we do not require $\phi$ to be Griffiths transverse in the nilpotent tangent directions.  Moreover, note that the definition is functorial in the sense that for any definable period map $\phi:X^\df\to\Gamma\backslash\Omega$ and any map $f:Y\to X$, we have that $f\circ\phi $ is a period map.  Finally, for $X$ integral, the Griffiths transversality condition is equivalent to the usual condition on the regular locus $X^{\mathrm{reg}}\subset X$.

The local liftability condition is equivalent to $\phi$ factoring through the stack quotient $[\Gamma\backslash \Omega]$ which is naturally a definable complex analytic Deligne--Mumford stack using the proof of Proposition \ref{finitequotient}.  There are no new subtleties in the definition of a definable complex analytic Deligne--Mumford stack, but we do not pursue these ideas here. Note that by Remark \ref{simply}, $\phi$ is definably locally liftable if and only if it is analytically locally liftable.

With these preliminaries, we now state a more general version of Corollary \ref{hodge}, to be proven in the next subsection.
\begin{thm}\label{hodgebetter}
Let $X$ be an algebraic space and $\phi:X^\df\to \Gamma\backslash\Omega$ a definable period map.  Then $\phi$ (uniquely) factors as $\phi =\iota\circ f^\df$ for a dominant map $f:X\to Y$ of algebraic spaces and a closed immersion $\iota:Y^\df\to \Gamma\backslash\Omega$ of definable complex analytic spaces.

\end{thm}

\begin{defn}  We refer to an algebraic space $Y$ with a closed immersion $\iota:Y^\df\to\Gamma\backslash\Omega$ of definable complex analytic spaces arising from the theorem as a \emph{definable period image}.
\end{defn}

Note that for a \emph{proper} definable period map $X^\df\to\Gamma\backslash\Omega$ Theorem \ref{hodgebetter} holds over an
arbitrary o-minimal structure.

\subsection{Algebraicity of the Hodge filtration}  
We make the following definition along the same lines as in the previous subsection:
\begin{defn}Let $Y$ be an algebraic space (possibly non-reduced).  A definable variation of Hodge structures on $Y$ is a triple $(V_\Z,F^\bullet, Q)$ where $V_\Z$ is a local system $V_\Z$ on $Y^\df$, $F^\bullet$ is a definable coherent locally split filtration of $V_\Z\otimes_\Z\O_{Y^\df}$ satisfying Griffiths transversality (in the same sense as Definition \ref{defn def period}), and $Q$ is a quadratic form on $V_\Z$, such that $(V_\Z,F^\bullet,Q)$ is a pure polarized integral Hodge structure fiberwise.

\end{defn}

As above, every local system on $Y^\an$ is definable by definable triangulation, see Remark \ref{simply}.  By the following lemma, if $\Gamma$ is torsion-free the triple $(V_\Z,F^\bullet,Q)$ exists universally on $\Gamma\backslash\Omega$ as a definable analytic variety, although of course it is \emph{not} in general a variation as it does not satisfy Griffiths transversality.

\begin{lemma}\label{the universal Hodge filtration is definable}Suppose $\Gamma$ is torsion-free and equip $V_\Z\otimes_\Z\O_{\Gamma\backslash\Omega}$ with its canonical definable structure.  Then the Hodge filtration $F^\bullet$ of $V_\Z\otimes_\Z\O_{\Gamma\backslash\Omega}$ is by definable coherent subsheaves.
\end{lemma}
\begin{proof}
For any definable fundamental set $\Xi\subset\Omega$, letting $\pi:\Xi\to \Gamma\backslash\Omega$ be the restriction of the quotient map, we must check that $\pi^*F^\bullet\subset V_\Z\otimes_\Z\O_{\Omega}$ is a definable coherent filtration, but this is obvious as it extends algebraically to $\check{\Omega}$.
\end{proof}

When $Y$ carries a definable variation that's clear from context, we denote by $F^\bullet_{Y^\df}$ the filtered Hodge bundle.  If $Y$ is smooth (in particular reduced) with a log smooth compactification $\bar Y$ and the variation has unipotent monodromy at infinity, we know that $F^\bullet_{Y^\an}:=(F^\bullet_{Y^\df})^\an$ has a canonical algebraic structure $F^\bullet_Y$.  In fact, the ambient flat bundle $V_\Z\otimes_\Z\O_{Y^\an}$ has a canonical extension $\bar {\mathcal{V}}$ (the Deligne canonical extension \cite{DeligneHodge}, uniquely determined by the condition that the connection have log poles with nilpotent residues\footnote{The proof in the algebraic space case is the same as that of varieties, as it relies on the existence and uniqueness of the analytic extension and ordinary GAGA.}), in which the filtration $F^\bullet_{Y^\an}$ extends as a filtration $F^\bullet_{\bar Y^\an}$ by subbundles (which we call the Schmid extension) as a consequence of the nilpotent orbit theorem \cite[Theorem 4.12]{Schmid}.  By ordinary GAGA, the vector bundle $\bar {\mathcal{V}}$ has a unique algebraic structure $\bar V$, as does the filtration.

We now show the following generalization of the first claim of the second part of Theorem \ref{maingriffiths}:
\begin{thm}\label{hodgealg}For $Y$ an algebraic space with a definable variation, $F^\bullet_{Y^\df}$ is the definabilization of a (unique) algebraic filtered bundle $F^\bullet_Y$.
\end{thm}

\begin{proof}
For a positive integer $N$, let $Y_N$ denote an irreducible component of the finite \'etale cover of $Y$ trivializing $N$-torsion in the local system $V_\Z$. The monodromy of $Y_N$ is then a subset of $I_{\dim V}+N \cdot M_{\dim V}(\Z)$. For $N\geq 3$, the eigenvalues of such an element cannot be roots of unity except for $1$, since if $(\epsilon-1)/N$ is integral and $\epsilon$ is a root of unity, we must have $N\leq 3$. It follows that for the pullback variation on $Y_N$, the monodromy at infinity, which is a priori only quasi-unipotent thanks to a well-known result of Borel \cite[Lemma 4.5]{Schmid}, must in fact be unipotent.
As $F_{Y^\df}^\bullet$ embeds in $f^\df_*F_{Y_N^\df}^\bullet$, by Theorem \ref{gaga} we may assume that the monodromy at infinity is unipotent.  

Let $Y_0$ be the reduced space of $Y$. Then $Y_0$ can be resolved by successive blow-ups, and performing the same blow-ups on $Y$ we obtain $X\to Y$ whose reduced space $X_0$ is smooth.  By taking some compactification and again blowing up to resolve the reduced boundary, we obtain a compactification $\bar X$ of $X$ whose reduced space is log smooth.

\begin{lemma}\label{deligne}  Let $\bar X$ be a proper algebraic space, and $D$ a closed subspace such that the reduced spaces $(\bar X_0,D_0)$ are a log smooth pair, and such that $X=\bar X\smallsetminus D$ has a definable variation with unipotent monodromy at infinity.  There is a unique map $f:\widetilde X\ra \bar X$ which is an isomorphism on reductions and over $X$, and minimal with respect to the following property:  $F^\bullet_{X^\df}$ extends as a filtered vector bundle to $\widetilde X^\df$ and restricts to the Schmid extension on the reduced space $(\bar X_0)^\df$.  Moreover, $F^\bullet_{X^\df}$ is algebraic, $F^\bullet_{X^\df}\cong (F^\bullet_X)^\df$, and the restriction of $F^\bullet_{X}$ to $X_0$ agrees with the canonical algebraic structure on $F^\bullet_{X_0}$. \end{lemma}
\begin{proof}  
The space $\bar X_0^\df$ admits a definable cover by polydisks $P=\Delta^n$ such that $X_0^\df$ is locally $P^*=(\Delta^*)^m\times \Delta^{n-m}$.  Let $\mathcal{R}$ be the restriction of the definable structure sheaf of $\bar X^\df$ to $P$.  Since an analytic space is Stein if and only if its reduction is Stein,  $(P^\an, \mathcal{R}^\an)$ is a Stein space, and so we may and do choose lifts $t_k$ of the coordinate functions $z_k$ on the reduction (by possibly shrinking further, these lifts are also definable). Note that a surjective exponential map $\mathcal{R}\ra \mathcal{R}^{\times}$ is still well defined with kernel $\Z^n$. 

Let  $q_k$ be a choice
of logarithm of $t_k$ for each $k$, definable on vertical strips, and  $N_1,\dots,N_m$ the nilpotent monodromy logarithms. We have a definable map $\phi:(\Delta^*)^m\times \Delta^{n-m} \ra \Gamma\backslash \Omega$, and so $\psi=\exp(-\sum q_kN_k)\phi $ lifts definably to $\check{\Omega}$. Thanks to Schmid's nilpotent orbit theorem \cite[Theorem 4.12]{Schmid}), the map on the reduction extends to $P$.

Let $i:X\to\bar X$ and $j:P^*\to P$ be the inclusions, and consider the sheaf $j_*j^*\mathcal{R}$ as a sheaf of rings on $P$. We have a pullback map of sheaves of rings $\psi^{-1}\O_{\check\Omega^
\df}\ra j_*j^*\mathcal{R}$, and we take $\mathcal{T}$ to be the subsheaf of rings of $j_*j^*\mathcal{R}$ generated by its image and $\mathcal{R}$.  We first claim that $(P,\mathcal{T})$ has the structure of a definable complex analytic space.  Consider the pullback $f\in j_*j^*\mathcal{R}(P)$ of an algebraic coordinate on $\check{\Omega}$.  As $(P^\an,\mathcal{R}^\an)$ is Stein, we may assume the reduction $f_0$ of $f$ lifts to a definable section $\tilde f$ of $\mathcal{R}$ (after shrinking $P$), and as $f-\tilde f$ is nilpotent, $f$ satisfies a monic polynomial.  Thus, $\mathcal{R}[f]$ is a definable coherent $\mathcal{R}$-module, and it follows that $(P,\mathcal{R}[f])$ is a definable analytic subspace of $(P,\mathcal{R})\times \C$.  Adjoining the pullbacks of all algebraic coordinates we conclude that $(P,\mathcal{T})$ is a definable complex analytic space. 

The subsheaf $\mathcal{T}\subset j_*j^*\mathcal{R}$ is uniquely determined as the ``minimal thickening" of $X$ such that $\psi$ extends, so we globally obtain a well-defined definable complex analytic space $\widetilde{\mathcal{X}}=(\bar X_0,\mathcal{T})$.  By ordinary GAGA, $(\widetilde{\mathcal{X}})^\an$ is the analytification of an algebraic space $\widetilde{X}$.  As $(\widetilde{\mathcal{X}})^\an$ is proper it admits a unique $\R_\an$-definable structure, and so we must in fact have $\widetilde{\mathcal{X}}=(\widetilde{X})^\df$.

We can pull back the Hodge filtration on $\check{\Omega}$ to get a definable filtered vector bundle $F^\bullet_{\widetilde {\mathcal{X}}}$ on $\widetilde {\mathcal{X}}$ extending $F^\bullet_{X^\df}$, and $\widetilde {\mathcal{X}}$ is locally determined as the minimal such extension of $X$ over $\bar X$. The gluing follows because the filtration on $\check{\Omega}$ is invariant under $\G(\C)$.  Once again by ordinary GAGA, $(F^\bullet_{\widetilde {\mathcal{X}}})^\an$ has a unique algebraic structure $F^\bullet_{\widetilde X}$ which must agree with the definable structure, $F^\bullet_{\widetilde{ \mathcal{X}}}\cong (F^\bullet_{\widetilde X})^\df$.  The construction is evidently functorial with respect to morphisms $\bar X\to\bar X'$ which are isomorphisms on reductions, and for $\bar X$ reduced agrees with the construction of the canonical algebraic structure, whence the last claim.  
\end{proof} 
Note that $X\to Y$ may not be dominant, but its image $Y''$ is isomorphic to $Y$ on a dense open set $U$.  Let $Z$ be a sufficiently thick nilpotent neighborhood of the complement of $U$ and $A=Y''\times_Y Z$.   Then $Y$ is naturally the pushout
\[\xymatrix{
A\ar[d]\ar[r]& Y''\ar[d]\\
Z\ar[r]&Y
}\]  
As $f:X\to Y''$ is proper dominant, $F^\bullet_{Y''^\df}$ embeds in $f_*(F^\bullet_{X})^\df$, so it is the definabilization of some algebraic $F^\bullet_{Y''}$ by Theorem \ref{gaga}. Since $Z$ has smaller dimension than $Y$, by induction $F^\bullet_{Z^\df}=(F^\bullet_{Z})^\df$ is algebraic, and $F^\bullet_{Y^\df}$ is the definabilization of the pushout of $F_{Y''}^\bullet$ and $F_{Z}^\bullet$.
\end{proof}

\begin{proof}[Proof of Theorem \ref{hodgebetter}]  Let $X$ be an algebraic space and $\phi:X^\df\to\Gamma\backslash\Omega$ a definable period map.  The proof of the previous theorem implies we can produce a proper $X'\to X$ such that the definable period map of $X'$ has unipotent monodromy at infinity and $X'\to X$ is dominant on some dense open set $U$ of $X$.  Moreover, we get a partial compactification $\bar X'$ which admits a definable proper period map $\bar\phi:\bar X'^\df\to\Gamma\backslash\Omega$ restricting to that of $X'$.  Applying Theorem \ref{defpropmap} to $\bar X'$, we obtain $\bar X'\to Y'$ (proper) dominant and $Y'^\df\to \Gamma\backslash\Omega$ a closed immersion.

Let $X''$ be the image of $X'$ in $X$, and let $W$ be a sufficiently thick nilpotent neighborhood of the complement of $U$ such that $X$ is the pushout of $W$ and $X''$.  By induction we may apply Theorem \ref{defpropmap} to $W$ to obtain a dominant $W\to Z$ and a closed immersion $Z^\df\to \Gamma\backslash\Omega$.  The sought for $Y$ is then the pushout of $Z$ and $Y'$ (which exists by \cite[\href{https://stacks.math.columbia.edu/tag/07VX}{Tag 07VX}]{stacks-project}). 
\end{proof}

Thanks to Lemma \ref{the universal Hodge filtration is definable}, every definable period map yields a definable variation by pulling back\footnote{Strictly speaking, pulling back from the stack.  Alternatively, one can take a definable cover by simply-connected opens, lift to $\Omega$, pull back and glue.}, and we conclude this subsection with a converse.
\begin{prop}\label{vartomap}Let $Y$ be an algebraic space.  An analytic period map $\phi:Y^\an\to(\Gamma\backslash\Omega)^\an$ associated to a definable variation is definable.
\end{prop}
\begin{proof}  Again we may produce a proper $X\to Y$ such that the pull back of the variation to $X$ has unipotent monodromy at infinity, has smooth reduced space, and for which $X\to Y$ is dominant on a dense open set $U$ of $Y$. Let $X'$ be the image of $X$ in $Y$, and let $Z$ be a sufficiently thick nilpotent thickening of the complement of $U$ such that $Y$ is the pushout of $X'$ and $Z$.  By induction we may assume the claim for $Z$.  It will be enough to show the claim for $X$, for then by Proposition \ref{prop mapdescent} we have it for $X'$, hence for the disjoint union of $X'$ and $Z$ and finally for $Y$ by applying again Proposition \ref{prop mapdescent} (note that the map from the disjoint union to the pushout is dominant, see for example \cite[\href{https://stacks.math.columbia.edu/tag/07VX}{Tag 07VX}]{stacks-project}).

Therefore, replacing $Y$ with $X$, we may assume $Y$ has smooth reduced space $Y^\reduced$.  From \cite{BKT}, there is a definable fundamental set $\Xi$ for $\Gamma$ such that the quotient map $\Xi\to\Gamma\backslash\Omega$ realizes $\Gamma\backslash\Omega$ as a definable complex analytic space as the quotient of $\Xi$ by a closed definable equivalence relation.  As above, the reduced period map $\phi^\reduced:(Y^\reduced)^\an\to (\Gamma\backslash\Omega)^\an$ is definable, so there is a definable open cover $\mathcal{Y}_i$ of $Y^\df$ such that we can choose lifts $\mathcal{Y}_i\to\Xi$ which are definable on reduced spaces.  But $\check{\Omega}$ is a flag variety and maps $\mathcal{Y}_i\to \check{\Omega}$ are clearly definable if and only if $\mathcal{F}^\bullet|_{\mathcal{Y}_i}$ is definable, and this implies $\mathcal{Y}_i\to\Xi$ is definable.
\end{proof}

Thus a definable variation on $Y$ is equivalent to a definable period map.  

\begin{cor}\label{comingfromgeometry}Let $Y$ be an algebraic space.  A period map associated to an algebraic subquotient of a variation $R^kf_*\Z$ for a smooth projective family $f:X\to Y$ is definable.
\end{cor}
\begin{proof}In this case the filtered Hodge bundle is algebraic. 
\end{proof}

\subsection{The Griffiths bundle}
Let as before $\Omega$ be a pure polarized period domain with generic Mumford--Tate group $\G$ and $\Gamma\subset \G(\Q)$ an arithmetic lattice. The Hodge filtration $F^\bullet$ and the Griffiths line bundle $L:=\bigotimes_i\det F^i$ exist universally on $\Omega$. Moreover, $L$ descends to $\Gamma\backslash\Omega$, but only as a $\Q$-bundle in general due to the possible torsion in $\Gamma$. For any algebraic space $Y$ with a definable map $Y^\df\to\Gamma\backslash\Omega$ we denote by $L_{Y^\df}$ the pullback of the Griffiths $\Q$-bundle.
\begin{lemma}Let $Y$ be a definable period image.  Then $L_{Y^\df}$ is the definabilization of a (unique) algebraic $\Q$-bundle $L_Y$.
\end{lemma}
\begin{proof}By definition there is an algebraic space $X$ with a definable period map factoring through $Y^\df$ such that $f:X\to Y$ is dominant.  By a similar argument as in the proof of Theorem \ref{hodgealg}, by possibly thickening $Y$ we may assume $f$ is proper.  By Theorem \ref{hodgealg} the Griffiths bundle on $X$ is the definabilization of an algebraic bundle $L_{X}$. Let $m$ be a positive integer such that $L_{Y}^m$ is a bona fide line bundle. As $L_{Y^\df}^m$ embeds in $f_*(L_X^m)^\df$, it follows from Theorem \ref{gaga} that $L_{Y^\df}^m$ is the definabilization of a (unique) algebraic line bundle.
\end{proof}

\subsection{Quasi-projectivity of period images} \label{subsectionquasiproj}

Let $Y$ be a definable period image in  $\Gamma\backslash\Omega$.  From the last subsection, we know that the Griffiths $\Q$-bundle $L_Y$ is algebraic.  In this subsection we apply Theorem \ref{qproj gen} to prove the second part of Theorem \ref{maingriffiths}.

 \begin{defn}\label{vanishing schmid} Assume $Y$ is a definable period image. For $Y$ reduced, we say a section $s$ of $L^m_Y$ \emph{vanishes at the boundary} if the following condition holds:  for some period map $X^\df\to\Gamma\backslash\Omega$ factoring through $Y$ such that $X\to Y$ is generically finite, $X$ is smooth, and the variation on $X$ has unipotent monodromy at infinity, $s$ pulls backs to a section of $L_{\bar X}^m(-D)$ where $(\bar X,D)$ is a log smooth compactification of $X$ and $L_{\bar X}$ is the Schmid extension \cite{Schmid} of $L_X$ to the Deligne canonical extension \cite{DeligneHodge} of the ambient flat vector bundle.  We let $\Gamma_{van}(Y,L^m_Y)\subset \Gamma(Y,L^m_Y)$ denote the linear subspace of sections vanishing at the boundary, which is finite-dimensional as $\Gamma_{van}(Y,L^m_Y)$ injects into $ \Gamma(\bar X,L_{\bar X}^m)$.

  \end{defn}
   Note that as in Definition \ref{vanishingsect}, the condition is independent of $(\bar X,D)$:  any two $X,X'$ satisfying the conditions can be dominated by a third $X''$, and for the resulting map $f:(\bar X'',D'')\to(\bar X,D)$ of log smooth pairs we naturally have an isomorphism $f^*L_{\bar X}\to L_{\bar X''}$.  Once again, $s$ vanishes at the boundary if and only if $s^\reduced$ does, and the ring $\bigoplus_n \Gamma_{van}(Y,L_Y^n) $ is integrally closed in
$\bigoplus_n \Gamma(Y,L_Y^n) $.

We are now in a position to state the main result of this section:

\begin{thm}\label{qproj}
Let $Y$ be a definable period image.  Then the Griffiths $\Q$-bundle $L_Y$ is ample on $Y$. Moreover, sections of some positive power $L^n_Y$ which vanish at the boundary realize $Y$ as a quasi-projective scheme.
\end{thm}

In the proof of the theorem, the positivity of the Griffiths bundle will be deduced from the special case of variations over smooth bases, where we have the following:

\begin{lemma}\label{nef and big}
Let $X$ be the complement of a normal crossing divisor $D$ in a compact K\"ahler manifold $\bar{X}$. Consider a polarized real variation of Hodge structure $\mathbb{V} = (V_\R,F^\bullet, Q)$ over $X$ with unipotent monodromies around $D$ and let $L_{\bar X}$ be as defined above. Then $L_{\bar X}$ is a nef line bundle. Moreover, $L_{\bar X}$ is big if and only if the associated period map is generically immersive.
\end{lemma}

\begin{proof}  
Letting $r_p = \mathrm{rank}(F^p)$, observe that the Griffiths bundle of $\mathbb{V}$ appears as the lowest nonzero piece in the Hodge filtration of the auxiliary polarized real variation of Hodge structure $ \mathbb{V}^\prime := \otimes_{p \in \Z} \bigwedge^{r_p} \mathbb{V} $, and $L_{\bar X}$ is the Schmid extension of the latter. Therefore the canonical metric $h$ on $L_X$ induced by the polarization $Q$ has non-negative curvature and extends as a singular metric on $L_{\bar X}$ with zero Lelong numbers, see \cite[Theorem 1.1]{Fujino-Fujisawa} or \cite{Brunebarbe}. Therefore the line bundle $L_{\bar X}$ is nef thanks to \cite[Corollary 6.4]{Demailly}. In particular, $L_{\bar X}$ is big if and only if $c_1(L_{\bar X})^{\dim X} >0$, cf. \cite[Theorem 1.2]{Boucksom}. 
Denoting by $C_1(L_X,h)$ the Chern form of the hermitian line bundle $(L_X,h)$, it follows by applying \cite[Theorem 5.1]{Kollar} to the auxiliary variation $ \mathbb{V}^\prime$  that the integral $\int_X{(C_1(L_X,h))^{\dim X}}$ is convergent and that we have the equality:
\[ c_1(L_{\bar X})^ {\dim X} = \int_X{(C_1(L_X,h))^{\dim X}} . \] 
We conclude using that the real $(1,1)$-form $C_1(L_X,h)$ is strictly positive at a point $x \in X$ if and only if the period map is immersive at $x$, cf. \cite[Proposition 7.15]{G2}.
\end{proof}

\begin{proof}[Proof of Theorem \ref{qproj}] 
We first reduce to the case that $\Gamma$ is neat.  Take $\Gamma'\subset\Gamma$ be a normal, neat subgroup of finite index $\ell$, with quotient $G$.  Let $Y$ be the period image of $X$, and $Y'$ the period image of the level cover of $X'$ in $\Gamma'\backslash\Omega$.  Then we have a surjective, dominant, finite map $\pi:Y'\ra Y$, and a group action $G$ on $Y'$ such that 
$\pi$ is $G$-invariant. For any section $\sigma$ of $L_{Y'}^k$ we claim that $\Nm(\sigma):=\prod_{g\in G}g\sigma$ descends
to $Y$ as a section of $L_Y^{k\ell}$. It is enough to work on stalks, so we may assume $L_Y$ and $L_{Y'}$ are trivial.  Let $y\in Y$, $R=\O^{\an}_{Y,y}$ and $S=\O^{\an}_{Y', \pi^{-1}(y)}$. For $f\in S$, we have that
$f$ lifts to a section $f_U$ on some ($G$-invariant) open neighborhood $U$ of $\pi^{-1}(y)$. Now 
$\Nm(f_U)$ is in the image of $\O^G_{(\Gamma'\backslash\Omega)^\an} = \O_{(\Gamma\backslash\Omega)^\an}$ and thus has an image $r\in R$, whose image in
$S$ is therefore $\Nm(f)$. 
We therefore have norm maps \[\Nm:\Gamma(Y',L^k_{Y'})\ra\Gamma(Y,L^{k\ell}_Y)\] for each $k$.  Clearly the norm of a vanishing section in the sense of Definition \ref{vanishing schmid} is a vanishing section.  Thus, if vanishing sections of $Y'$ yield an embedding of $Y'$ as a quasi-projective scheme, then by taking norms and applying Proposition \ref{embedding criterion} the same will be true of $Y$. 

We therefore assume $
\Gamma$ is neat, and in particular that the restriction of the variation to any subvariety of $Y$ has unipotent monodromy at infinity.  By the above remarks and Lemma \ref{nef and big}, the Schmid extension $L_{\bar X}$ satisfies the conditions in Setting \ref{setup} and the two notions of vanishing sections in Definitions \ref{vanishingsect} and \ref{vanishing schmid} agree.  Therefore, the claim follows from Theorem \ref{qproj gen}.
\end{proof}

\subsection{An ampleness criterion for the Hodge bundle}\label{secthodge}
Often in applications the Hodge bundle (namely the determinant of the deepest piece of the Hodge filtration) is more accessible than the Griffiths bundle, and we prove in this section an ampleness criterion for the Hodge bundle.

Let $(V_\Z,F^\bullet, Q)$ be a pure polarized integral variation of Hodge structure on a (reduced) separated algebraic space $X$. Using that the induced connection on $V_{\O_X} := V_\Z \otimes_\Z \O_X$ satisfies Griffiths transversality, we get for every integer $p$ an induced $\O_X$-linear map of $\O_X$-modules $\psi_p: T_X \rightarrow \Hom_{\O_X}(F^p / F^{p+1} , F^{p-1} / F^p)$. 

\begin{thm}\label{An ampleness criterion}
Let $F^{n}$ be the lowest piece of the Hodge filtration, meaning that $F^{n} \neq 0$ but $F^{n+1} = 0$. Assume that for any germ of a curve $\phi : \Delta \rightarrow X$, the $\O_{\Delta}$-linear map of $\O_{\Delta}$-modules $\phi^\ast(\psi_n) : T_{\Delta} \rightarrow \Hom(\phi^\ast (F^n) , \phi^ \ast (F^{n-1} / F^n))$ is injective. Then the line bundle $\det (F^n)$ is ample on $X$.
\end{thm}

Observe that for $X$ smooth the condition in the theorem is satisfied when the $\O_X$-linear map of $\O_X$-modules $\psi_n : T_X \rightarrow \Hom(F^n , F^{n-1} / F^n)$ is injective with image a locally split $\O_X$-submodule of $\Hom(F^n , F^{n-1} / F^n)$. This last condition implies that the period map is immersive, but the converse is not true.\\

The proof of Theorem \ref{An ampleness criterion} is parallel to the proof of Theorem \ref{qproj} (replace Lemma \ref{nef and big} by the lemma below). Note that in fact the latter is a particular case of the former since one easily check that Griffiths line bundle is the lowest piece of the Hodge filtration of the auxiliary variation $\otimes_{p \in \Z} \wedge^{r_p} \mathbb{V}$ where $r_p = \rk F^p$.

\begin{lemma}

Let $X$ be a smooth algebraic variety, $X \subset \bar{X}$ a smooth compactification such that $\bar X - X = D$ is a normal crossing divisor. Let $ (V_{\R}, F^{\bullet}, Q)$ be a polarized real variation of Hodge structure over $X$ with unipotent monodromies around $D$. Let $F^n_{\bar X}$ be the Schmid extension of the lowest piece of the Hodge filtration. Then the  line bundle $\det (F^n_{\bar X})$ is nef. Moreover, $\det (F^n_{\bar X})$ is big if and only if
the $\O_X$-linear map of $\O_X$-modules $ \psi_n : T_X \rightarrow \Hom_{\O_X}(F^n , F^{n-1} / F^n)$ is injective.
\end{lemma}

\begin{proof}
The polarization $Q$ permits to define a canonical positive definite Hermitian metric $h$ on $\det(F^n)$. Denote by $C_1(\det(F^n),h)$ the Chern form of the hermitian line bundle $(\det(F^n),h)$. It follows from the computation of the curvature of the Hodge bundles (see \cite[Theorem 5.2]{G2} or \cite[Lemma 7.18]{Schmid}) that $C_1(\det(F^n),h)$ is a positive real $(1,1)$-form on $X$, and $C_1(\det(F^n),h)$ is strictly positive at a point $x \in X$ if and only if the $\O_X$-linear map of $\O_X$-modules $ \psi_n : T_X \rightarrow \Hom_{\O_X}(F^n , F^{n-1} / F^n)$ is injective at $x$. With this fact at hand, the rest of the proof is parallel to the proof of Lemma \ref{nef and big}.
\end{proof}

\section{Applications}\label{sectionapp}
We start by making some remarks related to the first two applications below.  We may more generally speak of period maps from a separated Deligne--Mumford stack $\mathcal{M}$ of finite type over $\C$ as follows.  We say a period map $\mathcal{M}^\an\to \Gamma\backslash \Omega$ consists of an \'etale atlas $U\to \mathcal{M}$ by an algebraic space and a period map $\phi:U^\an\to \Gamma\backslash\Omega$ for which the resulting two compositions $(U\times_{\mathcal{M}}U)^\an \rightrightarrows \Gamma\backslash \Omega$ are equal.  For example, for a smooth projective family $\pi:\mathcal{X}\to \mathcal{M}$, the local system $R^k\pi_*\Z$ will underly such a variation.  We say that the period map is either quasi-finite or $\R_{\an,\exp}$-definable if this is so for the period map on the atlas, and we say the image of the period map in $\Gamma\backslash\Omega$ is the image of the period map on the atlas (since any two atlases are the same up to an etale cover, these definitions are independent of the atlas).

Recall that the definability condition is again automatic if $\mathcal{M}$ is reduced \cite[Theorem 1.3]{BKT}, and is satisfied for all period maps arising from geometry, by Corollary \ref{comingfromgeometry}.

\subsection{Borel algebraicity}\label{sectborel}

The following is an analog of a theorem proven by Borel \cite[Theorem 3.1]{Borel} (see also \cite[Theorem 5.1]{Deligne}) for locally symmetric varieties:
\begin{cor}\label{borel}Let $\mathcal{M}$ be a separated Deligne--Mumford stack of finite type over $\C$ admitting a quasi-finite $\R_{\an,\exp}$-definable period map, and let $Z$ be a reduced algebraic space.  Then any analytic map $Z^\an\to \mathcal{M}^\an$ is algebraic. 
\end{cor}
\begin{proof}  Let $U\to\mathcal{M}$ be a finite-type \'etale atlas.  It is enough to algebraize the base-change of the map $Z^\an\to \mathcal{M}^\an$ to $U$ along with the descent data, so we may assume $\mathcal{M}=U$.  Let $Y$ be the period image of the period map $U^\df\to\Gamma\backslash\Omega$.  The composition $Z^\an\to U^\an\to(\Gamma\backslash\Omega)^\an$ is a period map and thus by Corollary \ref{hodge} it follows that $Z^\an\to Y^\an$ is $\R_{\an,\exp}$-definable.  As $U\to Y$ is quasi-finite, $Z^\an\to U^\an$ is also $\R_{\an,\exp}$-definable, and therefore by Theorem \ref{defchow} algebraic.
\end{proof}
Applied to a separated Deligne--Mumford moduli stack of smooth polarized varieties with an infinitesimal Torelli theorem, for example, Corollary \ref{borel} implies that any analytic family of such varieties over (the analytification) of a reduced algebraic base $Z$ is in fact algebraic.
\begin{cor}  For $\mathcal{M}$ as above, if $\mathcal{M}$ is in addition reduced, then $\mathcal{M}^\an$ admits a unique algebraic structure.
\end{cor}

\subsection{Quasi-projectivity of moduli spaces}\label{sectcomplete}

Recall by a well-known result of Keel--Mori \cite{KM} that a separated Deligne--Mumford stack $\mathcal{M}$ of finite type over $\C$ admits a coarse moduli space $M$ which is a separated  algebraic space of finite type over $\C$.
\begin{cor}\label{coarse}Let $\mathcal{M}$ be a separated Deligne--Mumford stack of finite type over $\C$ admitting a quasi-finite $\R_{\an,\exp}$-definable period map.  Then the coarse moduli space of $\mathcal{M}$ is quasi-projective.
\end{cor}

\begin{proof}
The Griffiths bundle exists on the coarse moduli space $M$ as a $\Q$-bundle by general results \cite[Lemma 2]{KV}.  Let $U\to\mathcal{M}$ be a finite-type \'etale atlas by an algebraic space, so that we have a definable period map $\phi:U^\df\to\Gamma\backslash\Omega$.  Let $Y$ be the period image.  We claim that the map $U\to Y$ factorizes through the coarse moduli space $M$ of $\mathcal{M}$.  Let $\mathcal{M'}\to\mathcal{M}$ and $U'\to U$ be the \'etale covers corresponding to a normal finite index neat $\Gamma'\subset\Gamma$ with quotient $G$.  Let $Y'$ be the period image of $U'$ in $\Gamma'\backslash\Omega$.  Then as the variation on $U'$ is pulled back from $Y'$, the map $U'\to Y'$ factorizes through $\mathcal{M'}$.  As $U=[G\backslash U']$ and $\mathcal{M}=[G\backslash \mathcal{M}']$, it follows that $U\to [G\backslash Y']$ factorizes through $\mathcal{M}$.  Therefore, the map $U\to[G\backslash Y']\to Y$ factorizes through $M$.

Thus we get a quasi-finite map $M\ra Y$.  By Theorem \ref{qproj}, $L_Y$ is ample, so we have an immersion $Y\to \PP^n$.  We then have a quasi-finite map $M\to \PP^n$, which by Zariski's main theorem factors as an open immersion and a finite map.  It follows that $L_M$ is ample. 
\end{proof}

\begin{remark}\label{rmk stacky}  The construction of $[G\backslash Y']$ as in the proof can be used to construct the algebraic image of a period map in the quotient stack $[\Gamma\backslash\Omega]$.  One could also develop the theory of definable complex analytic Deligne--Mumford stacks, although we have not pursued this level of generality.
\end{remark}

Corollary \ref{coarse} applies to any (separated finite-type) smooth Deligne--Mumford stack that is the moduli stack of smooth polarized varieties $X$ with an infinitesimal Torelli theorem.  By work of Viehweg \cite{Viehweg}, such results are known for varieties $X$ with a semi-ample canonical bundle, and so the case of Fano varieties is of particular interest. For concreteness, we deduce some new results about moduli spaces of complete intersections, on which previous work has been done for hypersurfaces by Mumford \cite{Mumford} and more generally by Benoist \cite{Benoist1,Benoist2}.

We fix a collection of integers $T= (d_1, \cdots, d_c ;n)$ with $n \geq 1$, $c \geq 1$ and $2 \leq d_1 \leq \cdots \leq d_c$. Recall that a complete intersection of type $T$
is a closed subscheme of codimension $c$ in $\mathbb{P}^{n+c}_{\C}$ which is the zero locus of $c$ homogeneous polynomials of degrees $d_1, \cdots, d_c$ respectively. 
Let $H$ be the Zariski-open subset of the Hilbert scheme of $\mathbb{P}^ {n+ c}_{\C}$ that parametrizes the smooth complete intersections of type $T$. Let $\mathcal{M}_T$ be the moduli stack of smooth complete intersections polarized by $\mathcal{O}(1)$, i.e. the quotient stack $[PGL^{n+c+1}(\C) \backslash H]$. 

When $T \neq (2;n)$ Benoist proved that $\mathcal{M}_T$ is a separated smooth Deligne-Mumford stack of finite type \cite[Theorem 1.6 and 1.7]{Benoist1}, and therefore has a coarse moduli space $M_T$. If in addition $d_1 = \cdots = d_c$ then $M_T$ is an affine scheme, \cite[Theorem 1.1.i)]{Benoist2}, while if $c>1$ and $d_2=\cdots=d_c$, $M_T$ is quasi-projective by \cite[Corollary 1.2]{Benoist2}.  Finally, for $T=(3;2)$, $M_T$ is quasi-projective by \cite{ACT}.

\begin{cor}\label{completeint}
For all $T\neq (2;n)$, the coarse moduli space $M_T$ is quasi-projective.
\end{cor}
\begin{proof} This follows from Corollary \ref{coarse} and Flenner's infinitesimal Torelli theorem \cite[Theorem 3.1]{Flenner}, which applies for $T\neq (3;2)$ and $T\neq (2,2;n)$ for $n$ even---in particular, to all remaining cases.
\end{proof}

\subsection{A factorization result}\label{sectkahler}

We prove here a result which, intuitively, says that all interesting variations of Hodge structures on compact K\"ahler manifolds come from algebraic geometry. 

\begin{thm}
Let $X$ be a dense Zariski open subset of a compact K\"ahler manifold $\bar X$, and let $(V_\Z,\mathcal{F}^\bullet, Q)$ be a pure polarized integral variation of Hodge structure on $X$. Assume that the monodromy of $V_\Z$ is torsion-free (this is always achieved by going to a finite \'etale cover of $X$) and that $X$ is the biggest open subset of $\bar X$ on which $V_\Z$ extends.

Then there exist a proper surjective holomorphic map with connected fibres $\pi: X \rightarrow Y$ for a normal quasi-projective variety $Y$ such that $(V_\Z,\mathcal{F}^\bullet, Q)$ is the pull-back by $\pi$ of a polarized integral variation of Hodge structure on $Y$.
\end{thm}

\begin{proof}
By hypothesis, the monodromy $\Gamma$ of $(V_\Z,\mathcal{F}^\bullet, Q)$ is torsion-free and the associated period map $\phi: X \rightarrow \Gamma\backslash\Omega$ is proper. We denote by $X \xrightarrow{\pi} Y \rightarrow \Gamma\backslash\Omega$ its Stein factorization, so that $Y$ is a normal analytic space and $\pi :X \rightarrow Y$ is surjective with connected fibres. Since $\Gamma$ is torsion-free, $(V_\Z,\mathcal{F}^\bullet, Q)$ descends to $Y$. To finish the proof, it remains to prove that $Y$, a priori only an analytic space, is in fact a quasi-projective variety. We cannot apply directly Theorem \ref{maingriffiths} since $X$ is not assumed to be algebraic. However one can proceed as follows. First observe that thanks to the following result of Sommese $Y$ admits a proper modification $Y^\prime \rightarrow Y$ such that $Y^\prime$ is a dense Zariski open subset of a compact K\"ahler manifold $\overline{Y^\prime}$. 

\begin{thm}[Sommese {\cite[Proposition III and Remark III-C]{Som2}}] \label{sommese thm}Let $X$ be a dense Zariski open subset in a compact K\"{a}hler manifold $\bar X$, $Y$ be a complex analytic space  and $\pi: X \to Y$ be a surjective proper holomorphic map with connected fibres. Then there exists $X^\prime$ (resp. $Y^\prime$) a dense Zariski open subset in a compact K\"{a}hler manifold $\bar X^\prime$ (resp. $\bar Y^\prime$) and a commutative diagram
\begin{center}
	
\begin{tikzcd}[row sep=scriptsize, column sep=scriptsize]
  & \bar X  &  & \bar X^\prime \arrow{dd}{\pi^\prime}  \arrow{ll}{\alpha^\prime}\\
X \arrow{dd}{\pi}	\arrow[hookrightarrow]{ur}{} &	   & X^\prime \arrow{ll}{\alpha} \arrow[hookrightarrow]{ur}{} \arrow{dd}{\pi^\prime_{|X^\prime}}	\\
    &  & & \bar Y^\prime    \\		
Y &     & Y^\prime \arrow{ll}{\beta}  \arrow[hookrightarrow]{ur}{}  \\
	\end{tikzcd}
\end{center}
where $\alpha : X^\prime \rightarrow X$ (resp. $\beta : Y^\prime \rightarrow Y$) are proper modifications and $\pi^\prime$, $\pi^\prime_{|X^\prime}$ are surjective proper maps with connected fibres.
\end{thm}

The composition $Y^\prime \rightarrow Y \rightarrow\Gamma \backslash \Omega$ endows $Y^\prime$ with a polarized integral variation of Hodge structure. 
Take $\Gamma'\subset\Gamma$ neat of finite index and let $Y'' \ra Y'$ be the base-change along $\Gamma'\backslash\Omega\to\Gamma\backslash\Omega$. If $\overline{Y''}$ denotes a compactification of $Y''$ whose boundary is a normal crossing divisor, the polarized integral variation of Hodge structure induced on $Y''$ has unipotent monodromy at infinity. Thanks to Lemma \ref{nef and big} the associated Griffiths line bundle $L_{\overline{Y''}}$ is big, hence $\overline{Y''}$ is Moishezon. It follows that the compact
K\"ahler manifold $\overline{Y'}$ is Moishezon, hence it is in fact projective algebraic. Since $Y^\prime \rightarrow Y $ is the Stein factorization of the composition $Y^\prime \rightarrow Y \rightarrow\Gamma \backslash \Omega$, it follows now from Theorem \ref{maingriffiths} and Riemann existence theorem
that $Y$ is quasi-projective.
\end{proof}

\end{document}